\documentclass[a4paper,12pt]{amsart}
\usepackage{diagrams,amssymb}
\newtheorem{lemma}{Lemma}
\newtheorem{prop}{Proposition}
\newtheorem{obs}{Observation}
\newtheorem{thm}{Theorem}
\newtheorem{cor}{Corollary}
\theoremstyle{definition}
\newtheorem{defn}{Definition}
\newtheorem{cond}{Condition}
\newtheorem{notn}{Notation}
\newtheorem{strgm}{Stratagem}
\theoremstyle{remark}
\newtheorem{rem}{Remark}

\newtheorem{ex}{Example}
\newtheorem{exs}[ex]{Examples}
\newcounter{numl}
\newcommand{\labelnuml}{\textup{(\roman{numl})}}
\newenvironment{numlist}{\begin{list}{\labelnuml}%
{\usecounter{numl}\setlength{\leftmargin}{0pt}%
\setlength{\itemindent}{2\parindent}%
\setlength{\itemsep}{\smallskipamount}\def
\makelabel ##1{\hss \llap {\upshape ##1}}}}{\end{list}}
\newenvironment{bulletlist}{\begin{list}{\labelitemi}%
{\setlength{\leftmargin}{\parindent}\def
\makelabel ##1{\hss \llap {\upshape ##1}}}}{\end{list}}

\DeclareSymbolFont{script}{U}{eus}{m}{n}
\DeclareSymbolFontAlphabet{\mathscr}{script}
\DeclareMathSymbol{\Wedge}{0}{script}{"5E}
\DeclareMathAlphabet{\mathrmsl}{OT1}{cmr}{m}{sl}
\newcommand{\R}{{\mathbb R}}
\newcommand{\C}{{\mathbb C}}
\newcommand{\Z}{{\mathbb Z}}
\newcommand{\N}{{\mathbb N}}
\newcommand{\Q}{{\mathbb Q}}

\newcommand{\Sph}{{\mathbb S}}
\newcommand{\T}{{\mathbb T}}
\newcommand{\cA}{{\mathcal A}}
\newcommand{\cB}{{\mathcal B}}
\newcommand{\cC}{{\mathcal C}}

\newcommand{\cE}{{\mathcal E}}

\newcommand{\cH}{{\mathcal H}}
\newcommand{\cI}{{\mathcal I}}
\newcommand{\cJ}{{\mathcal J}}

\newcommand{\cL}{{\mathcal L}}

\newcommand{\cS}{{\mathscr S}}

\newcommand{\g}{\mathfrak g}

\newcommand{\un}{\mathfrak u}
\newcommand{\tor}{{\mathfrak t}}
\newcommand{\torh}{{\mathfrak h}}

\newcommand{\su}{\mathfrak{su}}
\newcommand{\crv}{\mathfrak{cr}}

\newcommand{\ham}{\mathfrak{ham}}
\newcommand{\Ab}[1][\cS]{\T_{#1}}
\newcommand{\ab}[1][\cS]{\tor_{#1}}
\newcommand{\Rv}{V}
\newcommand{\bK}{{\mathbf K}}
\newcommand{\Rb}[1]{{\mathscr K}^{#1}}
\newcommand{\Lb}[1]{{\mathscr R}^{#1}}
\newcommand{\Ds}{{\mathscr D}}
\newcommand{\Dsm}{\check\Ds}
\newcommand{\Lv}{L}
\newcommand{\restr}[1]{|_{#1}^{\vphantom x}}
\newcommand{\Restr}[1]{\Big|_{#1}}
\newcommand{\ip}[1]{\langle #1 \rangle}
\newcommand{\mult}{^{\scriptstyle\times}}
\newcommand{\transp}{^\top}
\newcommand{\del}{\partial}
\renewcommand{\d}{\mathrmsl{d}}
\newcommand{\into}{\hookrightarrow}
\newcommand{\sub}{\subseteq}
\newcommand{\dsum}{\oplus}
\renewcommand{\emptyset}{\varnothing}
\newcommand{\st}{\mathrel{|}}
\newcommand{\eps}{\varepsilon}
\newcommand{\Lam}{{\mathrmsl\Lambda}}
\newcommand{\Sig}{{\mathrmsl\Sigma}}
\newcommand{\im}{\mathop{\mathrm{im}}\nolimits}
\newcommand{\diag}{\mathop{\mathrm{diag}}\nolimits}
\newcommand{\rank}{\mathop{\mathrm{rank}}\nolimits}

\newcommand{\grad}{\mathop{\mathrm{grad}}\nolimits}
\newcommand{\Hess}{\mathop{\mathrm{Hess}}}

\newcommand{\Hom}{\mathrm{Hom}}
\newcommand{\Stab}{\mathrm{Stab}}

\newcommand{\spns}{\mathrm{span}}

\newcommand{\Proj}{\mathrm P}
\newcommand{\Id}{\mathit{Id}}
\newcommand{\gM}{\cE}
\newcommand{\mm}{{\boldsymbol\sigma}}
\newcommand{\ang}{{\boldsymbol\vartheta}}
\newcommand{\angU}{{\boldsymbol\varphi}}
\newcommand{\bnu}{{\boldsymbol\nu}}
\newcommand{\bt}{{\mathbf t}}
\newcommand{\bG}{{\mathbf G}}
\newcommand{\bH}{{\mathbf H}}
\newcommand{\al}{\alpha}
\newcommand{\be}{\beta}
\newcommand{\ga}{\gamma}
\newcommand{\de}{\delta}

\newcommand{\lamc}{\lambda}
\newcommand{\Qt}{{\mathbf Q}}
\newcommand{\ins}{^\circ}
\newcommand{\gip}{\zeta}
\newcommand{\ibm}{\psi}
\newcommand{\afs}{\eps}
\newcommand{\As}{{\mathscr A}}
\newcommand{\Cb}{\Phi}
\newcommand{\Fan}{{\mathsf Y}}
\newcommand{\Pol}{\Delta}

\newcommand{\Fa}{F}
\newcommand{\afn}{w}
\newcommand{\ul}{{\mathbf u}}
\newcommand{\Ll}{{\mathbf L}}
\newcommand{\red}{^{\mathrm{red}}}
\newcommand{\fib}{^{\mathrm{tor}}}
\begin{document}

\title[Levi--K\"ahler reduction and toric geometry]
{Levi--K\"ahler reduction of CR structures,\\ 
products of spheres, and toric geometry}
\author[V. Apostolov]{Vestislav Apostolov}
\address{Vestislav Apostolov \\ D{\'e}partement de Math{\'e}matiques\\
UQAM\\ C.P. 8888 \\ Succursale Centre-ville \\ Montr{\'e}al (Qu{\'e}bec) \\
H3C 3P8 \\ Canada}
\email{apostolov.vestislav@uqam.ca}
\author[D.M.J. Calderbank]{David M. J. Calderbank}
\address{David M. J. Calderbank \\ Department of Mathematical Sciences\\
University of Bath\\ Bath BA2 7AY\\ UK}
\email{D.M.J.Calderbank@bath.ac.uk}
\author[P. Gauduchon]{Paul Gauduchon}
\address{Paul Gauduchon \\ Centre de Math\'ematiques\\
{\'E}cole Polytechnique \\ UMR 7640 du CNRS\\ 91128 Palaiseau \\ France}
\email{pg@math.polytechnique.fr}
\author[E. Legendre]{Eveline Legendre}
\address{Eveline Legendre\\ Universit\'e Paul Sabatier\\
Institut de Math\'ematiques de Toulouse\\ 118 route de Narbonne\\
31062 Toulouse\\ France}
\email{eveline.legendre@math.univ-toulouse.fr}
\thanks{V.A. is supported in part by an NSERC discovery grant. E.L. is partially supported by France ANR project EMARKS No ANR-14-CE25-0010. The authors are grateful to the Institute of Mathematics and Informatics of the Bulgarian Academy of Sciences, the London Mathematical Society, and the labex CIMI (Toulouse) for hospitality and financial support.}

\date{\today}
\begin{abstract} We study CR geometry in arbitrary codimension, and
introduce a process, which we call the Levi--K\"ahler quotient, for
constructing K\"ahler metrics from CR structures with a transverse torus
action.  Most of the paper is devoted to the study of Levi--K\"ahler quotients
of toric CR manifolds, and in particular, products of odd dimensional spheres.
We obtain explicit descriptions and characterizations of such quotients, and
find Levi--K\"ahler quotients of products of $3$-spheres which are extremal in
a weighted sense introduced by G.~Maschler and the first author~\cite{AM}.
\end{abstract}

\maketitle

\section*{Introduction}

In recent years there has been considerable interest in the interaction
between K\"ahler geometry and its odd-dimensional younger cousin, Sasaki
geometry. On the one hand, ideas in K\"ahler geometry, such as toric methods
or extremal metrics, have led to the development of analogues in Sasaki
geometry.  On the other hand, Sasaki manifolds have a canonical
$1$-dimensional foliation generated by the Reeb vector field, which both
provides a construction of K\"ahler metrics on the leaf space when the latter
is a manifold or orbifold, as well as a ``horizontal'' generalization of such
quotients when it is not.

Our thesis herein is that these ideas need not be limited to $1$-dimensional
foliations. Indeed, any Sasaki manifold has an underlying codimension one CR
structure, whereas CR manifolds arise naturally in arbitrary codimension. This
prompts us to introduce transverse ``Reeb foliations'' on arbitrary CR
manifolds $(N,\Ds,J)$ (a theory of such foliations has recently been developed
in~\cite{Meerss}, but here we focus on the horizontal K\"ahler geometry of
$(\Ds, J)$). However, whereas in codimension one, the exterior derivative of
the contact form equips the horizontal distribution $\Ds$ with a nondegenerate
$2$-form (which, together with the complex structure $J$ on $\Ds$, defines the
horizontal K\"ahler structure on $\Ds$), in higher codimension, the
non-integrability of $\Ds$ is measure by a $2$-form on $\Ds$ with values in
$TN/\Ds$, called the \emph{Levi form} $L_\Ds$. In order to construct a
K\"ahler metric on the leaf space of the Reeb foliation, we therefore need
also to choose a nondegenerate component of $L_\Ds$. This construction, which
we call a \emph{Levi--K\"ahler quotient}, is our main topic of study.

In particular, the Levi form of $\Ds$ must \emph{have} a nondegenerate
component. Rank $2m$ distributions $\Ds$ of this type, on manifolds $N$ of
dimension $2m+\ell$, were studied in a companion paper~\cite{tcg}, to which
the present work may be viewed as a sequel, although we do not here rely upon
knowledge of that paper, or its main results. Indeed, whereas in~\cite{tcg} we
study the general theory of toric contact manifolds in higher codimension, the
applications in K\"ahler geometry we develop herein use only the simplest
examples: toric CR submanifolds of flat space, and in particular, products of
spheres.

Our prime motivation is the construction of interesting toric K\"ahler
orbifolds $M$, which are K\"ahler orbifolds of real dimension $2m$ admitting a
isometric hamiltonian action of a real $m$-torus $\tor/2\pi\Lam$. As a
symplectic orbifold~\cite{Audin,Delzant,LT} or complex toric
variety~\cite{danilov,fulton}, $M$ is classified respectively by the image
$\Pol$ of its momentum map, which is a convex polytope in an $m$-dimensional
real affine space $\As$ modelled on $\tor^*$, or by the underlying fan $\Fan$
in $\tor$ of normal rays to the facets (codimension one faces) of
$\Delta$. Explicitly, $\Delta$ is an intersection of half-spaces $L_s\geq 0$,
where $s\in\cS$ indexes the facets of $\Delta$, and $L_s\in\torh$, the
$(m+1)$-dimensional vector space of affine functions on $\As$; the normals of
$\Fan$ are then the linear parts $u_s\in\tor$ of $L_s$ for $s\in\cS$. It is
convenient to encode these data in linear maps $\Ll\colon \R_\cS\to \torh$ and
$\ul\colon \R_\cS\to \tor$, where $\R_\cS$ is the standard real vector space
with a basis $e_s:s\in\cS$, $\Ll(e_s)=L_s$ and $\ul(e_s)=u_s$. The kernel $\g$
of $\ul$ is the Lie algebra of a subtorus $G$ of $\R_\cS/2\pi\Z_\cS$, and we
let $\lamc\in\g^*$ be the restriction of $\Ll$ to $\g$ (which takes values in
the constant affine functions, i.e., $\R$).  Equipping the complexification
$\C_\cS$ of $\R_\cS$ with its standard flat K\"ahler structure, $M$ is then
equivariantly symplectomorphic to the symplectic quotient of $\C_\cS$ by $G$
at momentum level $\lamc$~\cite{Delzant,LT}, and equivariantly biholomorphic
to a GIT quotient of $\C_\cS$ by the complexification of $G$~\cite{Cox}.

However, $M$ need not be isometric to the K\"ahler quotient of $\C_\cS$ by
$G$, and the K\"ahler quotient metric, known as the Guillemin
metric~\cite{Guillemin}, has not been found to be particularly interesting,
except in the simplest cases. We are thus motivated to make the following
observations. First, the level set $\mu_\g^{-1}(\lamc)$ of the momentum map of
$G$ is a toric CR submanifold $N_{\g,\lamc}$ of $\C_\cS$, in fact an
intersection of quadric hypersurfaces. Secondly, since $G$ preserves
$N_{\g,\lamc}$ with orbits transverse to the CR distribution $\Ds$, we may
regard the Levi form as a $\g$-valued $2$-form on $\Ds$. We find that the
$\lamc$ component is nondegenerate and induces a K\"ahler metric on $M\cong
N_{\g,\lamc}/G$---in other words, the data $(\g,\lamc)$ is exactly what we
need to define a Levi--K\"ahler quotient of $N_{\g,\lamc}$. While this metric
does not seem to be particularly interesting either, there is a third
observation that proves to be decisive: we do not need to use the same pair
$(\g,\lamc)$ to define the toric CR submanifold $N$ as we use to take the
Levi--K\"ahler quotient. Indeed, $(\g,\lamc)$ determines a Levi--K\"ahler
quotient of $N_{\g_o,\lamc_o}$ biholomorphic to $M$ provided only that the
data $(\g_o,\lamc_o)$ arises from a polytope $\Pol_o$ with the same
combinatorial type as $\Pol$.

We take advantage of these observations by using the principle that a
Levi--K\"ahler quotient of a CR manifold $N$ can be expected to have nice
curvature properties if $N$ does. The simplest examples, in codimension one,
are round CR $(2m+1)$-spheres, which are the toric CR submanifolds associated
to $m$-simplices, and are circle bundles over complex projective spaces.
Products of such spheres provide examples in higher codimension. Thus if
$\Pol$ is a polytope with the same combinatorial type as a product of
simplices, we obtain a distinguished toric K\"ahler metric on the toric
symplectic orbifold associated to $\Pol$ as a Levi--K\"ahler quotient of a
suitable product of spheres.

The main results come in three courses, which we serve up in
Sections~\ref{s:lkr-toric},~\ref{s:lkr-sph} and~\ref{s:c-lkq}, after
presenting some background and preliminary results in Section~\ref{s:lkq}.
The preliminary material reviews the notion of a CR $(2m+\ell)$-manifold of
codimension $\ell$ and studies local CR torus actions transverse to the CR
distribution. These have associated K\"ahler cones of dimension $2(m+\ell)$
and so may be viewed as a natural generalization of Sasaki structures.  Such
a transverse local action of an $\ell$-dimensional abelian Lie algebra $\g$,
together with an element $\lamc\in\g^*$ is called a \emph{positive Levi pair}
if it defines a horizontal K\"ahler structure on the CR distribution. We
obtain from this the Levi--K\"ahler quotient construction when the action of
$G$ integrates to an action of a Lie group $G$.

Section~\ref{s:lkr-toric} presents our general results on toric CR
submanifolds of $\C_\cS$ and their Levi--K\"ahler quotients, thus establishing
the observations made above. To do this, we first review the elements of toric
geometry and combinatorics, and make precise the notion of combinatorial type.
Then we show in Theorem~\ref{t:toric} that if $(\g,\lamc)$ is a positive Levi
pair associated to polytope $\Pol$ and $N$ is a toric submanifold with the
same combinatorial type, then the image of the momentum map of the horizontal
Levi--K\"ahler structure (or the Levi--K\"ahler quotient when that exists) has
image $\Pol$. Here we need a small part of the toric contact theory developed
in~\cite{tcg}, which we summarize in Theorem~\ref{t:polytope}. In the case
that $N$ is toric intersection of quadric hypersurfaces contained in a round
hypersphere, the Levi--K\"ahler structure can be made explicit, as we show in
Theorem~\ref{t:sph-red-metric}.

In Section~\ref{s:lkr-sph}, we study the construction of toric K\"ahler
metrics as Levi--K\"ahler quotients of products of spheres. Building on
Theorem~\ref{t:toric}, we characterize such quotients in
Theorem~\ref{t:product-spheres} as those associated to a polytope with the
combinatorial type of a product of simplicies. In turn,
Theorem~\ref{t:symplectic-potential} builds on Theorem~\ref{t:sph-red-metric}
by giving explicit formulae for the Levi--K\"ahler quotients of products of
spheres and their symplectic and K\"ahler potentials. In the remainder of the
section, we explore relations between this construction and other explicit
methods in K\"ahler geometry. The explicit form of certain Levi--K\"ahler
quotients of products of $3$-spheres motivates a new ansatz for toric K\"ahler
metrics whose Delzant polytope is projectively equivalent to a cube, extending
the ambitoric ansatz of Segre type~\cite{ACG1,ACG2} to arbitrary dimension.

Another source of toric K\"ahler orbifolds $M$ whose Delzant polytope has the
combinatorics of a product of simplices can be obtained from the generalized
Calabi construction, where both the base and the fibre are toric orbifolds
with Delzant polytopes having the combinatorics of product of simplices
(see~\cite{hfkg5}). This includes the complex Hirzebruch surfaces, or more
generally any holomorphic projective bundle over a projective space, and,
inductively, toric fibrations where the base and the fibre are one of the
mentioned smooth complex manifolds.  We show that in this setting, the
K\"ahler metric corresponding to the Levi--Kahler quotient of the product of
spheres associated to $M$ is obtained from the generalized Calabi
construction, where the metrics on the base and on the fibre are themselves
Levi--Kahler quotients of product of spheres.

The final tranche of results concern curvature properties of a Levi--K\"ahler
quotient $M$ of a product of spheres $N$. The product structure on $N$ induces
distributions on $M$ and we show that the curvature of $M$ has vanishing
Bochner component on each such distribution, simply because CR spheres have
vanishing Chern--Moser tensor. For $3$-dimensional CR manifolds, the vanishing
of the Chern--Moser tensor is automatic. We compute instead the scalar
curvature of a Levi--K\"ahler quotient of a product of $3$-spheres and observe
that when the polytope is projectively equivalent to a cube, the
Levi--K\"ahler quotient is extremal in a weighted sense that was introduced
(in a special case) in~\cite{AM}.

More precisely, given a ``conformal dimension'' $p\in\R$ and a positive
function $w$ on a compact symplectic orbifold $(M,\omega)$ whose hamiltonian
vector field is quasiperiodic (i.e., it belongs to the Lie algebra of a torus
$\Ab[]$ in $\mathrm{Ham}(M,\omega)$), we can generalize the approach of
Donaldson~\cite{donaldson} and Fujiki~\cite{fujiki} to Calabi's extremal
K\"ahler metrics~\cite{calabi} by using $w^{-(p-1)}$ as a weight for the
formal Frech\'et symplectic structure on the space of $\T$-invariant
compatible complex structures. Then the action of
$\mathrm{Ham}^{\Ab[]}(M,\omega)$ on this space is hamiltonian, and if we
weight the inner product on its Lie algebra of by $w^{-(p+1)}$ then the
momentum map at $J$ may be identified with a modification $s_{J,w,p}$ of the
scalar curvature of $g_J=\omega(\cdot,J\cdot)$, which is what the scalar
curvature of the conformally related metric $w^2 g_J$ would be if $M$ had
dimension $p$.

If the polytope $\Pol$ of $M$ is projectively equivalent to a cube, then
opposite facets of $\Pol$ meet in coplanar lines, so there is a unique affine
function $w$ up to scale which is positive on $\Pol$ and vanishes on the
intersection of opposite facets. We prove in Theorem~\ref{t:LK-EM} that the
unique (up to scale) toric metric on $(M,\omega)$ for which $s_{J,w,m+2}$ is
an affine function is the one arising as a Levi--K\"ahler quotient of a
product of $3$-spheres.

The Levi--K\"ahler quotient metrics of products of $\ell$ odd dimensional
spheres do not lead to new extremal K\"ahler metrics (in the classical sense)
unless $\ell\leq 2$.  In the case $\ell =1$, the quotient of a sphere are
Bryant's Bochner-flat K\"ahler metrics on weighted projective
spaces~\cite{Bryant}, which are extremal. We end the paper with some new
extremal examples obtained as the quotient of a product of two spheres.

\section{Levi--K\"ahler quotients of CR manifolds}\label{s:lkq}

\subsection{CR structures of arbitrary codimension}

\begin{defn} A \emph{CR structure $(\Ds,J)$ of rank $m$ and codimension $\ell$}
on a real $(2m+\ell)$-dimensional manifold $N$ is a real rank $2m$
distribution $\Ds\sub TN$ equipped with an almost complex structure $J\colon
\Ds \to \Ds$, which satisfies the integrability conditions
\begin{equation}\label{eq:CR-integrability}
\begin{split}
[X,Y]- [JX,JY] &\in \Gamma(\Ds),  \\
[X,JY] + [JX,Y] &= J([X,Y] - [JX,JY]), \quad\forall\, X,Y\in \Gamma(\Ds),
\end{split} 
\end{equation}
or, equivalently,
\[
[\Gamma(\Ds^{1,0}), \Gamma(\Ds^{1,0})] \sub \Gamma(\Ds^{1,0}),
\]
where $\Ds^{1,0}\sub TN\otimes \C$ is the subbundle of $(1,0)$ vectors in
$\Ds\otimes\C$.

$(N,\Ds,J)$ is then called a \emph{CR manifold} (of codimension $\ell$).
\end{defn}

The underlying rank $2m$ distribution $\Ds$ on $N$ may be viewed as a
codimension $\ell$ generalization of a contact structure on $N$~\cite{tcg}.
The fundamental invariant of $\Ds$ is its \emph{Levi form} $\Lv_\Ds\colon
\Wedge^2 \Ds\to TN/\Ds$, defined, via $X,Y\in \Gamma(\Ds)$, by the tensorial
expression
\begin{equation}\label{eq:levi-form}
\Lv_\Ds(X,Y) = - q_\Ds([X,Y])
\end{equation}
where $q_\Ds\colon TN\to TN/\Ds$ is the quotient map. The transpose of $q_\Ds$
identifies $(TN/\Ds)^*$ canonically with the annihilator $\Ds^0$ of $\Ds$,
which is a rank $\ell$ subbundle of $T^*N$. The normalization convention for
$\Lv_\Ds$ is chosen so that for any section $\al$ of $\Ds^0$, the restriction
of $\d\al$ to $\Wedge^2\Ds\sub \Wedge^2TN$ is $\al\circ \Lv_\Ds$.

The \emph{nondegeneracy locus} of $\Ds$ is the open subset $U_\Ds=\{\al\in
\Ds^0\cong (TN/\Ds)^*\st \al\circ \Lv_\Ds$ is nondegenerate$\}$ of $\Ds^0$.
If $U_\Ds\cap \Ds^0_z$ is nonempty then, since nondegeneracy is an open
condition, $\Ds^0_z$ has a basis $\al_1,\ldots \al_\ell$ in $U_\Ds$ and so
$U_\Ds\cap \Ds^0_z$ is the complement of the set where $(\sum_{i=1}^\ell
t_i\al_i)\circ\Lv_\Ds$ degenerates, which is the cone over a projective
hypersurface $V_{\Ds,z}$ of degree $m$ (the zero set of a homogeneous degree
$m$ polynomial in the $\ell$ variables $t_1,\ldots t_\ell$). In~\cite{tcg},
$V_\Ds\sub\Proj(\Ds^0)$ is called the \emph{degeneracy variety} of $\Ds$.
Therein it is shown that $U_\Ds$ is a canonical ``symplectization'' of
$(N,\Ds)$: $U_\Ds$ is the open subset of $\Ds^0$ over which the pullback of
the canonical symplectic form $\Omega$ on $T^*N$ to $\Ds^0$ is nondegenerate.

\begin{defn} Let $(N,\Ds,J)$ be a \emph{CR manifold} (of codimension $\ell$).
We say $\Ds$ is \emph{Levi nondegenerate} if $U_\Ds$ has nonempty intersection
with each fibre of $p\colon \Ds^0\to N$.  A (local) section of $U_\Ds$ is
called a (local) \emph{contact form} on $N$.

Note that the Levi form $\Lv_\Ds$ satisfies
\begin{equation*}
\Lv_\Ds(X,Y) = - \tfrac12 q_\Ds([X,Y]+[JX,JY])
\end{equation*}
and hence is $J$-invariant or ``type (1,1)'' on $\Ds$. It follows that
$h_\Ds(X,Y):= \Lv_\Ds(X,JY)$ is a section of $S^2\Ds^*\otimes TN/\Ds$. We say
$(N,\Ds,J)$ is \emph{Levi definite} if at each $z\in N$ there exists $\al\in
\Ds^0_z$ such that $\al\circ h_\Ds\in S^2\Ds^*_z$ is positive definite.
\end{defn}

Clearly Levi definite CR manifolds are Levi nondegenerate: more generally
$U_\Ds^+:=\{\al\in \Ds^0\st \al\circ h_\Ds$ is positive definite$\}$ is an
open and closed submanifold of $U_\Ds$.

\begin{exs} \begin{numlist}
\item A maximally real codimension $\ell$ submanifold of $\C^{m+\ell}$ is
a smooth submanifold $N\sub\C^{m+\ell}$ for which $\Ds:= TN \cap JTN$, where
$J$ is the standard complex structure of $\C^{m+\ell}$, has rank $2m$ (i.e.,
corank $\ell$ in $TN$). Then $(N,\Ds)$, with the induced action of $J$ on
$\Ds$, is a CR manifold of rank $m$ and codimension $\ell$. A model example,
in codimension one, is the unit sphere $\Sph^{2m+1}$ in $\C^{m+1}$.
\item If $(N_i,\Ds_i,J_i)$ are CR manifolds, with codimensions $\ell_i$,
for $i\in\{1,\ldots n\}$, then so is $(\prod_{i=1}^n N_i, \Ds_1\dsum\cdots\dsum
\Ds_n,J_1\dsum\cdots\dsum J_n)$, with codimension $\ell=\ell_1+\cdots +\ell_n$
and $U_\Ds=\prod_{i=1}^n U_{\Ds_i}$. In particular, the product of $n=\ell$
codimension one CR spheres $\Sph^{2m_1+1}\times \cdots \times \Sph^{2m_\ell+1}$
is a CR manifold with codimension $\ell$.
\end{numlist}
\end{exs}

\begin{rem} The Levi form of a CR manifold $(N,\Ds,J)$ is traditionally
defined to be the hermitian form $h_\Ds+i\Lv_\Ds\colon\Ds\times\Ds\to
\C\otimes TN/\Ds$; however it is uniquely determined (given $J$) by its real
or imaginary part, and the imaginary part is an invariant of the underlying
(real) distribution $\Ds$. The rank is usually called the \emph{CR dimension}.

Levi nondegeneracy implies that $X\mapsto\Lv_\Ds(X,\cdot)$ is an injective
bundle homomorphism from $\Ds$ to $\Hom(\Ds,TN/\Ds)$. This condition, together
with the assumption that $\Lv_\Ds$ is surjective onto $TN/\Ds$, appears in the
study~\cite{beloshapka} of CR automorphisms of real quadrics
$N_\sigma:=\{(z,w)\in \C^{m+\ell} \st \Im (w)=\Im \sigma(z,z)\}$, where
$\sigma\colon \C^m \times \C^m \to \C^{\ell}$ is hermitian and $\Im$ denotes
the imaginary part. Such quadrics are homogeneous CR manifolds of rank $m$ and
codimension $\ell$, with Levi form isomorphic to $\sigma$ (or $\Im \sigma$ in
our sense).

Levi definiteness extends the codimension one notion of strict
pseudoconvexity.
\end{rem}

\subsection{Local CR actions, generalized Sasaki structures and K\"ahler
cones}

\begin{defn} Let $(N,\Ds,J)$ be a CR manifold; then the space $\crv(N,\Ds,J)$
of \emph{CR vector fields} is the Lie subalgebra of vector fields $X$ on $N$
such that
\begin{equation}
\cL_X \Gamma(\Ds)\sub\Gamma(\Ds) \quad\text{and}\quad \cL_X J =0.
\end{equation}
A (\emph{local, effective}) \emph{CR action} of a Lie algebra $\g$ on
$(N,\Ds,J)$ is a Lie algebra monomorphism $\bK\colon\g\to\crv(N,\Ds,J)$.  For
$v\in\g$, we write $K_v$ for the induced vector field $\bK(v)$, and we define
$\kappa^\g\colon N\times\g\to TN$ by $\kappa^\g(z,v)=K_{v,z}$. Let $\Rb\g\sub
TN$ be the image of $\kappa^\g$, i.e., ${\Rb\g}_z:=\spns\{K_{v,z}\st
v\in\g\}$.  Since $\bK\colon\g\to\crv(N,\Ds,J)$ is a Lie algebra morphism,
$\Rb\g$ is an integrable distribution.
\end{defn}

\begin{ex}\label{cr-pb} Let $\pi\colon N\to M$ be a principal $G$-bundle
with connection $\eta\colon TN\to \g$, where $\dim G=\ell$ and $\dim M=2m$.
Then $\Ds:=\ker\eta$ is a rank $2m$ distribution on $N$, and $\eta$ induces a
bundle isomorphism of $TN/\Ds$ with $N\times\g$.  In this trivialization, the
Levi form of $\Ds$ is $\d\eta+\frac12[\eta\wedge\eta]_\g$, the pullback to $N$
of the curvature $F^\eta$ of $\eta$. If $M$ has an (integrable) complex
structure $J$ for which $F^\eta$ is $J$-invariant, the horizontal lift of $J$
to $\Ds$ equips $N$ with a $G$-invariant CR structure. For $G$ abelian, this
is the principal torus bundle construction of \cite{wang-ziller}.
\end{ex}
Suppose $\bK\colon\g\to\crv(N,\Ds,J)$ is a local CR action, where $(N,\Ds,J)$
is CR of codimension $\ell=\dim\g$. Abstracting the local geometry of
Example~\ref{cr-pb}, we say the action of $\g$ is \emph{transversal} if the
following condition holds.

\begin{cond}\label{cond:trans} At every point of $N$, $\Ds+\Rb\g=TN$.
Equivalently\textup:
\begin{numlist}
\item $\rank\Rb\g=\ell$ everywhere on $N$\textup;
\item $\Ds\cap\Rb\g$ is the zero section of $TN$ (and thus
  $TN=\Ds\oplus\Rb\g$).
\end{numlist}
\end{cond}
The composite $q_\Ds\circ\kappa^\g\colon N\times\g\to\Rb\g\to TN/\Ds$ is a
bundle isomorphism and so there is a canonically defined $1$-form
$\eta^\g\colon TN\to \g$, characterized by
\begin{equation*}
\ker\eta^\g=\Ds\quad \text{and}\quad \forall\, v\in \g, \quad \eta^\g(K_v)=v. 
\end{equation*}
We also denote by $\eta^\g$ the induced map from $TN/\Ds$ to $\g$.
For any $\lamc\in\g^*$, define $\eta^\lamc=\eta^{\g,\lamc}\colon
N\to\Ds^0$ by $\eta^\lamc_z(X) =\ip{\eta_z(X),\lamc}$, so that
$\eta^\lamc(K_v)=\lamc(v)$ and
$\d\eta^\lamc\restr{\Ds}=\ip{\d\eta\restr{\Ds},\lamc}=\eta^\lamc\circ\Lv_\Ds$
is the $\lamc$-component of the Levi form of $\Ds$.

If $\bK$ integrates to an action of a connected Lie group $G$ on $N$, then
Condition~\ref{cond:trans}(i) implies that the $G$-action is locally free, so
that $M:=N/G$ is a compact orbifold. Condition~\ref{cond:trans}(ii) then
ensures that $\Ds$ is isomorphic to the pullback of $TM$ to $N$, and hence
$G$-invariant data on $\Ds$ descend to $M$. Invariant components of the Levi
form provide examples of such data.

In codimension one, a transversal CR action is essentially a CR Reeb vector
field, or equivalently, a compatible Sasaki structure, which makes the
symplectic cone K\"ahler.

To generalize this to arbitrary codimension, note that the total space $\Ds^0$
of $p\colon \Ds^0\to N$ inherits from $T^*N$ a \emph{tautological $1$-form}
$\tau$: using the exact sequence
\begin{equation}\label{eq:vb-es}
0\to p^*\Ds^0\to T\Ds^0\stackrel{p_*}{\to} p^*TN\to 0,
\end{equation}
$\tau_\al=\al\circ p_*\colon T_\alpha \Ds^0\to \R$ for any $\al\in\Ds^0$.  We
set $\Omega^\Ds=\d\tau$; this is the pullback of the tautological symplectic
form on $T^*N$ to $\Ds^0$.

Any $X\in\Gamma(TN)$ has a lift to a hamiltonian vector field $\tilde X$ on
$T^*N$ with $p_*(\tilde X)=X$ and hamiltonian $f_X=\tau(\tilde X)$, i.e.,
$f_X(\alpha)=\alpha(X)$; furthermore $\{f_X,f_Y\}=f_{[X,Y]}$.  (Explicitly,
$\d f_X=-\Omega^\Ds(\tilde X,\cdot)$, where $\tilde X_\al =
\al_*(X_z)-(\cL_X\al)_z$ for any extension of $\al\in T^*_zN$ to a local
section.) If $X\in\crv(N,\Ds,J)$, then $\tilde X$ is tangent to $\Ds^0\sub
T^*N$.

\begin{obs}\label{o:lift-ham-act} Let $\bK\colon\g\to\crv(N,\Ds,J)$ be a
local CR action of $\g$ on $(N,\Ds)$ and define $\mu_\g\colon\Ds^0\to
\g^*$ by $\ip{\mu_\g(\al),v}=\al(K_v)$ for $\al\in\Ds^0$ and $v\in \g$. Then
the lift of $\bK$ to $T^*N$ preserves $\Ds^0$, and the induced local action
$\tilde\bK$ is hamiltonian on $U_\Ds$ with momentum map
$\mu_\g\restr{U_\Ds}$\textup; in particular $\ip{\d\mu_\g(\tilde
  K_v),w}=-\ip{\mu_\g,[v,w]_\g}$ for all $v,w\in\g$.
\end{obs}
This is immediate. Now $(p,\mu_\g)\colon \Ds^0\to N\times \g^*$ is a bundle
isomorphism with inverse $\ibm_\g(z,\lamc):=\ip{\eta_z,\lamc}$: if
$\al=\ip{\eta_z,\lamc}$ for some $\lamc\in\g^*$ and $z\in N$, then
$\mu_\g(\al)=\lamc$.

\begin{lemma}\label{l:pb-sympc} Let $\tau$ be the tautological $1$-form on
$\Ds^0$. Then
\begin{equation}
(\ibm_\g^*\tau)_{(z,\lamc)}(X+a)=\ip{\eta(X),\lamc},
\end{equation}
and hence
\begin{equation}\label{eq:om-g}
  (\ibm_\g^*\Omega^\Ds)_{(z,\lamc)}(X+a,Y+b) =
  \ip{a, \eta (Y)} - \ip{ b, \eta (X)} + \ip{\d \eta (X, Y),\lamc}.
\end{equation}
\end{lemma}
\begin{proof} Since $\tau_\al(Z)=\al(p_*(Z))$, $(\ibm_\g^*\tau)_{(z,\lamc)}(X+a)
=\tau_{\ip{\eta_z,\lamc}}\bigl((\ibm_\g)_*(X+a)\bigr) =\ip{\eta(X),\lamc}$.
Hence $\ibm_\g^*\tau=\ip{p_2,p_1^*\eta}$, where $p_1$ and $p_2$ are the first
and second projections of $N\times\g^*$. Now $\ibm_\g^*\Omega^\Ds=\ibm_\g^*\d\tau
=\d(\ibm_\g^*\tau)=\ip{\d p_2\wedge p_1^*\eta}+\ip{p_2,p_1^*\d\eta}$, which
yields~\eqref{eq:om-g}
\end{proof}

For $(z,v)\in N\times \g$, let $\cJ_{(z,v)}$ be the complex structure on
$T_{(z,v)} (N\times\g)=T_z N\dsum\g$ defined by
\begin{equation*}
\cJ_{(z,v)}(X+w)=J X^\Ds + K_{w,z}-\eta(X)
\end{equation*}
where $X^\Ds$ denotes the $\Ds_z$-component of $X\in T_z N=\Ds_z\dsum
(\Rb\g)_z$.

\begin{lemma}\label{l:sympc-int}  The almost complex structure $\cJ$ is
integrable if and only if $\g$ is abelian.
\end{lemma}
\begin{proof} Let $X,Y$ be vector fields on $N$ with $\eta(X)$ and $\eta(Y)$
  constant and let $u,v\in \g^*$; as vector fields on $N\times \g$, these are
  constant in the $\g$ direction. First observe that
\begin{align*}
N_\cJ(X,Y) &=[X,\cJ Y] + [\cJ X,Y] - \cJ([X,Y] - [\cJ X,\cJ Y])\\
&=[X,JY^\Ds]+[JX^\Ds,Y]- J [X,Y]^\Ds + \eta([X,Y]) + \cJ [JX^\Ds,JY^\Ds]=0,
\end{align*}
since $[X,Y]^\Ds=[X^\Ds,Y^\Ds]+[K_{\eta(X)},Y^\Ds]+[X^\Ds,K_{\eta(Y)}]$.  Next
\begin{align*}
\cJ N_\cJ(u,Y)&=\cJ([u,\cJ Y]+[\cJ u,Y]) + [u,Y]-[\cJ u,\cJ Y]\\
&=J [K_{u},Y]^\Ds-\eta([K_{u},Y])
- [K_{u}, JY^\Ds-\eta(Y)] = 0.
\end{align*}
Finally,
\begin{align*}
\cJ N_\cJ(u,v)&=
\cJ([u,\cJ v]+[\cJ u, v]) + [u,v]-[\cJ u,\cJ v]\\
&= [K_{u},K_{v}] = K_{[u,v]},
\end{align*}
which vanishes for all $u,v\in\g$ iff $\g$ is abelian.
\end{proof}

Using $\eta$ to identify $\Rb\g$ with $N\times\g$, we observe that
\begin{equation} \label{eq:TN-decomp}
T_{(z,\lamc)} (N \times\g^*) \cong \Ds_z \oplus \g \oplus \g^*
\quad\text{and}\quad
T_{(z,v)} (N \times \g) \cong  \Ds_z \oplus \g \oplus \g;
\end{equation}
in these terms, $(\ibm_\g^*\Omega^\Ds)_{(z,\lamc)}$ is the sum of
$\ip{\d\eta_z,\lamc}$ and the standard symplectic structure on $\g \oplus
\g^*$, while $\cJ_{(z,v)}$ is the sum of the complex structure $J$ on $\Ds_z$
and the standard complex structure on $\g\oplus\g$. Thus if $\g$ is abelian
and we identify $N\times\g$ with $N\times\g^*$ using a symmetric positive
definite bilinear form $\gip$ on $\g$, we obtain a K\"ahler structure on the
open subset of $(z,\lamc)\in N\times\g^*$ with $\ibm_\g(z,\lamc)\in U_\Ds^+$.

\begin{prop} Let $(N,\Ds,J)$ be a codimension $\ell$ CR manifold, and $\g$ a
transversal CR action of an \textup($\ell$-dimensional\textup) abelian Lie
algebra $\g$ with a positive definite inner product.  Then $U_\Ds^+\sub
N\times \g^*$ has a canonical K\"ahler metric on which $\g$ defines a local
isometric hamiltonian action whose momentum map is the projection $p_2\colon
N\times\g^*\to \g^*$.
\end{prop}
\begin{proof} Lemmas~\ref{l:pb-sympc} and~\ref{l:sympc-int} imply that $U_\Ds^+$
is K\"ahler. The hamiltonian vector field generated by the component
$\ip{p_2,v}$ of $p_2$ is the pullback of $K_v$, which clearly preserves the
complex structure $\cJ$, hence the metric $h$ on $N\times\g^*$.
\end{proof}
\begin{defn} Let $(N,\Ds,J)$ be Levi definite, and let $\g$ be
an ($\ell$-dimensional) abelian Lie algebra.  Then a \emph{$\g$-Sasaki
  structure} on $N$ is a transversal CR action of $\g$ together with a
positive definite inner product $\gip$; given such an action, we say
$(N,\Ds,J,\g,\gip)$ is a \emph{codimension $\ell$ Sasaki manifold} with
\emph{generalized K\"ahler cone} $U_\Ds^+$.
\end{defn}

\subsection{CR torus actions and Levi--K\"ahler quotients}\label{s:toric-CR}

\begin{defn} A local CR action $\bK\colon\ab[N]\to \crv(N,\Ds,J)$ of an abelian
Lie algebra $\ab[N]$ on a CR manifold $(N,\Ds,J)$ is called a \emph{local CR
  torus action}, and is said to be a \emph{CR torus action} of $\Ab[N]=
\ab[N]/2\pi\Lam_N$ if it integrates to an effective (i.e., faithful) action of
$\Ab[N]$. If $\bK\colon\g\to\crv(N,\Ds,J)$ satisfies
Condition~\ref{cond:trans}, we refer to $\Rb\g$ and its integral submanifolds
as the associated \emph{Reeb distribution} and \emph{Reeb foliation}
transverse to $\Ds$.
\end{defn}

Given a local CR torus action $\bK\colon\ab[N]\to \crv(N,\Ds,J)$, let
\begin{align*}
\kappa&:=\kappa^{\ab[N]}\colon N\times\ab[N]\to TN, && \text{with}&
\kappa(z,v)&=K_{v,z},\quad\text{and}\\
\mu&:=\mu_{\ab[N]}\colon\Ds^0\to \ab[N]^*,&& \text{with}&
\ip{\mu(\al),v}&=\al(K_v),
\end{align*}
so that $(p,\mu)\colon\Ds^0\to N\times\ab[N]^*$ is the pointwise transpose of
$q_\Ds\circ\kappa\colon N\times\ab[N]\to TN/\Ds$.

\begin{defn}\label{d:lkq} An $\ell$-dimensional
subalgebra $\iota\colon\g\into\ab[N]$ and an element $\lamc\in\g^*\setminus 0$
together form a \emph{Levi pair} $(\g,\lamc)$ for a local CR torus action
$\bK$ if:
\begin{bulletlist}
\item $\g$ acts transversally on $N$ via $\bK$, i.e.,
  $\Rb\g:=\spns\{K_{v,z}\st v\in\g\}$ satisfies Condition~\ref{cond:trans}.
\end{bulletlist}
Let $\eta\colon TN\to \g$ be the connection $1$-form of $\g$,
$\eta^\lamc:=\ip{\eta,\lamc}$, and $h_{\Ds,\lamc}:=
\d\eta^\lamc\restr{\Ds}(\cdot,J\cdot)$. Then $(\Ds,h_{\Ds,\lamc},J)$ is called
a \emph{Levi structure} and we say that $(g,\lamc)$ or $(\Ds,h_{\Ds,\lamc},J)$
is
\begin{bulletlist}
\item \emph{nondegenerate} if $\eta^\lamc$ is a contact form, i.e.,
  $h_{\Ds,\lamc}$ is nondegenerate on $\Ds$;
\item \emph{positive} if $h_{\Ds,\lamc}$ is positive definite on $\Ds$.
\end{bulletlist}
We say $(N,\Ds,J,\bK)$ is \emph{Reeb type} if it admits a nondegenerate Levi
pair, and if positive, we say that $(N,\Ds,J,\g,\lamc)$ or
$(N,\Ds,J,h_{\Ds,\lamc}$) is \emph{Levi--K\"ahler}.

If $\bK$ is a CR torus action of $\Ab[N]$ and $\g$ is the Lie algebra of a
closed subgroup $G$ of $\Ab[N]$, then $M/G$, with the K\"ahler metric induced
by $(h_{\Ds,\lamc},J,\d\eta^\lamc\restr{\Ds})$ is called the
\emph{Levi--K\"ahler quotient} of $(N,\Ds,J)$ by $(\g,\lamc)$.
\end{defn}

If $N$ is compact, $\{\lamc\in\g^*\setminus 0\st \eta^\lamc\in\Gamma(U_\Ds)\}$
is an open cone $\cC_\g\sub \g^*$.

Let $(\g,\lamc)$ be a Levi pair.  For any $v\in\ab[N]$,
$(\d\eta^\lamc)(K_v,\cdot)= -\d(\eta^\lamc(K_v))$.  We may thus view
$\eta^\lamc(K_v)= \ip{\mu(\eta^\lamc),v}$ as the ``horizontal momentum'' of
$K_v$ with respect to the Levi structure $(\Ds,\d\eta^\lamc\restr{\Ds})$.
Observe that if $v\in\g$, $\eta^\lamc_z(K_v)= \ip{v,\lamc}$, which vanishes
for $v\in\ker\lamc\sub\g$.  Hence $z\mapsto \mu(\eta^\lamc_z)\in \ab[N]^*$
takes values in $(\ker\lamc)^0\cong (\ab[N]/\ker\lamc)^*$.

\begin{strgm}\label{str:setup} For any pair $(\g,\lamc)$ with $\g\sub\ab[N]$
and $\lamc\in\g^*\setminus 0$, the quotient $\ab[N]/\ker\lamc$ is an extension
by $\R$ of the quotient $\ab[N]/\g$.  To allow $(\g,\lamc)$ to vary, it is
convenient to fix this extension $\torh\to\tor$ (where $\torh$ and $\tor$ are
abelian Lie algebras of dimensions $m+1$ and $m$); then the commutative
diagram
\begin{diagram}[size=1.5em,nohug]
0 &\rTo & \g &\rTo^\iota & \ab[N] & \rTo^{\ul_N} & \tor &\rTo &0\\
  &     & \dTo^\lamc&   & \dTo^{\Ll_N} &   & \dEq & \\
0 &\rTo & \R&\rTo^{\afs}&\torh & \rTo^\d & \tor &\rTo &0.
\end{diagram}
of short exact sequences associates pairs $(\g,\lamc)$, with $\ab[N]/
\ker\lamc\cong\torh$, to surjective linear maps $\Ll_N\colon\ab[N]\to \torh$
(thus $\g$ is the kernel of $\ul_N:=\d\circ\Ll_N$, and $\lamc$ is induced by
$\Ll_N\restr\g$).

Let $\As\sub\torh^*$ be the affine subspace $(\afs\transp)^{-1}(1)$ modelled
on $\tor^*$; then $\torh$ may be identified with the affine linear functions
$\ell\colon\As\to\R$, whence $\d\ell\in \tor$ is the linear part of
$\ell\in\torh$.
\end{strgm}

By Observation~\ref{o:lift-ham-act}, $\bK$ lifts to a hamiltonian action on
$U_\Ds$ with momentum map $\mu\restr{U_\Ds}$. If $(\g,\lamc)$, defined by
$\Ll_N\colon \ab[N]\to\torh$, is a Levi pair, then the map $\mu^\lamc\colon
N\to\As\sub\torh^*$, determined uniquely by the formula
\begin{equation}\label{eq:momentum}
\ip{\mu^\lamc(z),\Ll_N(v)}= \eta^\lamc_z(K_v)
\end{equation}
for all $z\in N$ and $v\in \ab[N]$, will be called the \emph{horizontal
  momentum map} of $(\Ds,\d\eta^\lamc\restr{\Ds})$. Equivalently the diagram
\begin{diagram}[size=1.5em,nohug]
N & \rTo^{\eta^\lamc}    & U_{\Ds} \\
\dTo^{\mu^\lamc}&        & \dTo_\mu \\
\torh^*&\rTo^{\Ll_N\transp}& \ab[N]^*
\end{diagram}
commutes, i.e., $\Ll_N\transp\circ\mu^\lamc= \mu\circ \eta^\lamc$.

\begin{obs} Let $(N,\Ds,J,\bK)$ be a CR manifold of Reeb type, and
let $(\g,\lamc)$ be a nondegenerate Levi pair, where $\g$ is the Lie algebra
of a subtorus $G$ of $\Ab[N]$. Then $M:=N/G$, equipped with the $2$-form
induced by $\d\eta^\lamc\restr{\Ds}$, is the symplectic quotient
${\mu_\g}^{-1}(\lamc)/G$ of $U_\Ds$ by the lifted $G$-action\textup; it is
therefore a compact symplectic orbifold with a hamiltonian action of
$\T=\Ab[N]/G$ whose momentum map is induced by the $G$-invariant map
$\mu^\lamc\colon N\to\As\sub\torh^*$ defined in~\eqref{eq:momentum}.
\end{obs}

Let $\Rb{}=\Rb{\ab[N]}=\im\kappa$, $\gM:=\im(p,\mu)\sub N\times\ab[N]^*$ and
$\Theta:=\gM^0=\ker(q_\Ds\circ\kappa)= \kappa^{-1}(\Ds)\sub N\times\ab[N]$.
If $\rank\Ds=2m$, then $\rank\Rb{}\cap\Ds\leq m$, and hence $\dim\ab[N]\leq
m+\ell$.

\begin{prop}\label{p:legendrian} Let $\bK\colon\ab[N]\to \crv(N,\Ds,J)$ be
a local CR torus action. Then $\Rb{}\cap\Ds$ is an integrable distribution,
i.e., $\Lv_\Ds(X,Y)=0$ for all $X,Y\in \Rb{}\cap\Ds$.
\end{prop}
\begin{proof} For any $v,w\in\ab[N]$ and any section $\al$ of $\Ds^0$,
  $\d\alpha(K_v,K_w)=(\cL_{K_v}\al)(K_w)-(\cL_{K_w}\al)(K_v)$.  Hence if
  $X=\sum_i f_i K_{v_i}$ and $Y=\sum_j g_j K_{w_j}$ are sections of $\Rb{}
  \cap\Ds$ (for functions $f_i$ and $g_j$ on $N$), then
\begin{equation*}\textstyle
(\al\circ\Lv_\Ds)(X,Y)=\d\alpha(X,Y) =\sum_i f_i (\cL_{K_{v_i}}\al)(Y)-\sum_j g_j
(\cL_{K_{w_j}}\al)(X)=0
\end{equation*}
since $K_v$ preserves $\Ds^0$ for any $v\in\ab[N]$.
\end{proof}

\begin{rem} If $(\g,\lamc)$ is a Levi pair then $(p,\mu_\g)
\colon\Ds^0\to N\times\g^*$, with $\mu_\g:=\iota\transp\circ\mu$, is an
isomorphism. In particular, $(p,\mu)$ injects, i.e., $\gM$ is a rank $\ell$
subbundle of $N\times\ab[N]^*$ (with $\gM^*\cong (N\times\ab[N])/\Theta\cong
TN/\Ds$). Equivalently, the transpose $q_\Ds\circ\kappa$ surjects (pointwise),
i.e., $\Rb{}\cap\Ds$ has codimension $\ell$ in $\Rb{}$. Conversely, this
suffices for the local existence of a transversal subalgebra
$\iota\colon\g\into\ab[N]$, hence also a Levi pair: $U_\Ds$ is open with
nonempty fibres, so we can find $\lamc\in\g^*$ such that $\eta^{\lamc}$ is
locally a contact form.
\end{rem}
\begin{defn} We say $(N,\Ds,J,\bK)$ is \emph{locally Reeb type} if $q_\Ds
\circ\kappa$ surjects and \emph{toric} if it is locally Reeb type with
$\dim\ab[N]=m+\ell$.
\end{defn}
On any open set where $\rank\Rb{}=m+\ell$, $(N,\Ds,\bK)$ is locally Reeb type.

\begin{prop} Let $(N,\Ds,J)$ be a toric CR manifold under $\Ab[N]$ and
let $N^o$ be the open subset on which the action of $\Ab[N]$ is free, and
$U_\Ds^o$ its inverse image in $U_\Ds\sub \Ds^0$.

Then $U_\Ds$ is a toric symplectic manifold with momentum map $\mu\colon
U_\Ds\to \ab[N]^*$ defined by $\ip{\mu(\al),v}=\al(K_v)$, where $\al\in
U_\Ds\sub T^*N$ and $K_v:=\bK(v)$ for $v\in \ab[N]$.  Further, there are
angular coordinates $\angU\colon U_\Ds^o\to \Ab[N]=\ab[N]/2\pi\Lam_N$, unique up
to an additive constant, with $\Omega^\Ds=\ip{\d \mu\wedge \d\angU}$ and
$\ker\d\angU=p_*^{-1}(J(\Rb{}\cap\Ds))$ \textup(restricting $p\colon\Ds^0\to
N$ to $U_\Ds^o$\textup).
\end{prop}
\begin{proof} The first part is immediate from
Observation~\ref{o:lift-ham-act}. By Proposition~\ref{p:legendrian},
$\Rb{}\cap\Ds$ is an integrable rank $m$ subbundle of $\Ds$, hence so is
$J(\Rb{}\cap\Ds)$ by the $J$-invariance of the Levi form, and
$TN^o=\Rb{}\oplus J(\Rb{}\cap\Ds)$. The $1$-form $\beta\colon
TN^o\to \ab[N]$ defined by $\ker\beta=J(\Rb{}\cap\Ds)$ and
$\beta(K_v)=v$ is therefore closed. It is not exact, but by definition of
$N^o$, the local primitives fit together to give a primitive with values in
$\ab[N]/2\pi\Lam_N$ unique up to a constant, whose pullback to $U_\Ds^o$ yields
$\angU$.
\end{proof}
In terms of the short exact sequence~\eqref{eq:vb-es} restricted to $U_\Ds^o$,
we have
\begin{diagram}[size=1.5em,nohug]
0 &\rTo & p^*\Ds^0 &\rTo &\ker\d\angU &\rTo & p^*J(\Rb{}\cap\Ds)&\rTo &0\\
  &     & \dEq     &     & \dTo    &         & \dTo    & \\
0 &\rTo & p^*\Ds^0 &\rTo & TU_\Ds^o &\rTo^{p_*} & p^*TN^o & \rTo &0,
\end{diagram}
where for any $\al\in U_\Ds^o$, $(\d\mu)_\al(\ker\d\angU)=\ab[N]^*$, and
$(\d\mu)_\al(p^*\Ds^0)=\mu(\Ds^0_{p(\al)})=\gM_{p(\al)}$. This identifies
$(\d\mu)_\al(p^*J(\Rb{}\cap\Ds))$ with $\Theta^*_{p(\al)}$.

A subalgebra $\iota\colon\g\into \ab[N]$ satisfying Condition~\ref{cond:trans}
splits the exact sequence
\begin{diagram}[size=1em,nohug]
0 &\rTo & \gM & \rTo & N^o\times \ab[N]^* & \rTo &\Theta^* &\rTo & 0
\end{diagram}
i.e., $\ker\iota\transp\cong\tor^*$ is transverse to $\gM_z$ for all $z\in
N^o$. Thus $\mu(U_\Ds^o)\sub\ab[N]^*$ is foliated by its intersection with the
$m$-dimensional family $\gM_z$ of $\ell$-dimensional linear subspaces of
$\ab[N]^*$.

\begin{rem}\label{r:Kq} The Levi--K\"ahler quotient of $N$ by $(\g,\lamc)$ is
the K\"ahler quotient of any $\Ab[N]$-invariant, $\Omega^\Ds$ compatible
metric $\hat g$ on $U^+_\Ds$ whose pullback to $\mu_\g^{-1}(\lamc)\cong N$ is
the orthogonal sum of a metric on $\Rb\g$ and the metric $h_{\Ds,\lamc}$ on
$\Ds$.  We may assume $\hat g$ has angular coordinates $\d\angU$ on
$U_\Ds^{o,+}=U_\Ds^o\cap U_\Ds^+$, so that it is determined uniquely there by
the induced $\Ab[N]$-invariant metric on $\ker\d\angU\cong N^o\times
\ab[N]^*$, which descends to a metric $\bG$ on $\mu(U_\Ds^{o,+})\sub
\ab[N]^*$. Since $h_{\Ds,c\lamc}=c\, h_{\Ds,\lamc}$ for $c\in\R^+$, we assume
$\bG$ is homogeneous of degree $1$ on $\ab[N]^*$, i.e., as an
$S^2\ab[N]$-valued function on $\mu(U_\Ds^{o,+})$, it is homogeneous of degree
$-1$. Examples of such metrics include the generalized K\"ahler cone metrics.

If $(\g,\lamc)$ is given by $\Ll_N\colon\ab[N]\to\torh$, then the
Levi--K\"ahler quotient metric depends only on the pullback of $\bG$ to
$\Ll_N\transp(\As)=(\iota\transp)^{-1}(\lamc)$, an affine subspace transverse
to $\gM$.
\end{rem}

\section{Levi--K\"ahler reduction in toric geometry}\label{s:lkr-toric}

\subsection{Polytopes, fans, combinatorics, and toric contact
manifolds} \label{s:toric-contact}

Suppose that $(N,\Ds,J,\bK)$ is a toric CR manifold of rank $m$ and
codimension $\ell$, under a (real) torus $\Ab[N]=\ab[N]/2\pi\Lam_N$ with
(abelian) Lie algebra $\ab[N]=\Lam_N\otimes_\Z\R$.

The theory of effective actions of tori~\cite{Audin,GGK} implies that for any
subtorus $H\leq\Ab[N]$,
\begin{equation*}
N_{(H)} = \{ z\in N \st H = \Stab_{\Ab[N]}(z) \}\sub
N^H = \{ z\in N \st H \sub \Stab_{\Ab[N]}(z)\}
\end{equation*}
is an open submanifold of a closed submanifold of $N$, and if $N_{(H)}$ is
nonempty then $N_{(H)}$ is dense in $N^H$. The connected components of
$N_{(H)}$ and their closures in $N$ are called \emph{open} and \emph{closed
  orbit strata} of $(N,\bK)$.  Let $\Cb_N$ be the set of closed orbit strata,
partially ordered by inclusion, and let $N_s:s\in\cS$ index the closed orbit
strata stabilized by a circle. We refer to $\Cb_N$ as the \emph{combinatorics}
of $N$; it is a ``poset over $\cS$''.

\begin{defn} A \emph{poset} (partially ordered set) \emph{over a set} $\cS$ is
a set $\Cb$ equipped with a partial ordering (reflexive antisymmetric
transitive relation) and a map $\cS\to\Cb$. A morphism $\Cb\to\Cb'$ of posets
over $\cS$ is an order preserving map whose composite with the map $\cS\to\Cb$
is the map $\cS\to\Cb'$. We say $\Cb$ and $\Cb'$ have the \emph{same
  combinatorial type} if they are isomorphic as posets over $\cS$. The
combinatorics arising in toric geometry are typically isomorphic to subposets
of the power set $P(\cS)$ or its opposite $P(\cS)^{op}$, which are posets over
$\cS$ under inclusion or reverse inclusion respectively, with the map from
$\cS$ being the singleton map $s\mapsto \{s\}$.
\end{defn}

To illustrate this, we start, as in Stratagem~\ref{str:setup}, with an exact
sequence
\[
0\to \R\stackrel{\afs}{\to}\torh\stackrel{\d}{\to}\tor\to 0
\]
of vector spaces, viewed as an extension of abelian Lie algebras with
$\dim\tor=m$, and let $\As:= (\afs\transp)^{-1}(1)$ be the corresponding
$m$-dimensional affine subspace of $\torh^*$. Recall that a convex polytope
$\Pol$ in $\As$ is a subset of the form
\begin{equation*}
\Pol:=\{\xi\in \As\st \forall s\in\cS,\enskip L_s(\xi)\geq 0\}
\end{equation*}
where $\cS$ is a finite set, and $L_s\in\torh$ (an affine function on $\As$)
for each $s\in \cS$.

\begin{rem} The \emph{combinatorics} $\Cb_\Pol$ of $\Pol$ is the poset
over $\cS$ of closed faces of $\Pol$. More precisely, for $\xi\in \As$, let
$S_\xi= \{s\in\cS\st L_s(\xi)=0\}$; we assume that $\Pol\sub\As$ has nonempty
interior $\mathring{\Pol}$ (so for $\xi\in\mathring{\Pol}$, $S_\xi=\emptyset$)
and that for any $s\in\cS$, there exists $\xi\in\Pol$ with $S_\xi=\{s\}$
(otherwise we may discard $s$ without changing $\Pol$). The map sending
$S\sub\cS$ to $\Fa_S:=\{\xi\in\Pol\st S\sub S_\xi\}=\{\xi\in\Pol\st \forall
s\in S,\enskip L_s(\xi)=0\}$ restricts to an isomorphism from $\{S_\xi\in
P(\cS)^{op}\st \xi\in\Pol\}$ to $\Cb_\Pol$ over $\cS$. Any closed face is thus
the intersection of the \emph{facets} $\Fa_s:=\Fa_{\{s\}}$ containing it:
$\Fa_S=\bigcap_{s\in S}\Fa_s$.  We assume that the empty face is an element of
$\Cb_\Pol$, so that $\Fa_S\in\Cb_\Pol$ for all $S\in P(\cS)$.
\end{rem}

Given a compact convex polytope $\Pol\sub\As$, the positive span $\R^+\Pol$ is
a cone in $\torh^*$; the \emph{dual cone} to $\Pol$ is
\begin{equation*}
\Pol^*:=\{L\in\torh\st \forall\xi\in\Pol,\enskip L(\xi)\geq 0\},
\end{equation*}
and its projection onto $\tor$ defines a decomposition $\Fan$ of $\tor$,
called the associated (complete) \emph{fan}, into a union of polyhedral cones
\begin{equation*}
\cC_S:=\{\d L\in \tor \st \forall\xi\in\Pol,\enskip L(\xi)\geq 0
\text{ with equality for all } \xi\in\Fa_S\}=\spns\{u_s\st s\in S\}
\end{equation*}
corresponding to the faces $\Fa_S$ of $\Pol$. These cones form a poset
$\Cb_\Fan$ over $\cS$ under inclusion, and a (complete) fan $\Fan$ is uniquely
determined by its combinatorics $\Cb_\Fan$ and its rays (one dimensional
cones) $\cC_s:=\cC_{\{s\}}$. When $\Fan$ is constructed from $\Pol$ as above
then $\Cb_\Fan$ is canonically isomorphic to $\Cb_\Pol^{op}$ over $\cS$, and
in particular there is a canonical bijection between the facets of $\Pol$ and
the rays $\cC_s:=\cC_{\{s\}}$ of $\Fan$.

The rays of $\Fan$ determine $u_s:s\in\cS$ up to positive scale, and similarly
$\Pol$ determines $L_s:s\in\cS$ up to positive scale. Given a choice of these
scales, we say $\Pol$ is a \emph{labelled polytope} with \emph{affine normals}
$L_s\in\torh$ and \emph{inward normals} $u_s:=\d L_s\in\tor$ for $s\in\cS$,
and that $\Fan$ is a \emph{labelled fan} with generators $u_s\in\tor$ for
$s\in\cS$.

\begin{defn} A (complete) fan $\Fan$ is \emph{simplicial} if the rays in any
cone are linearly independent. A (compact) convex polytope $\Pol$ is
\emph{simple} if its fan is simplicial. In terms of a labelling, this means
for all $S\in\Cb_\Fan$, $u_s:s\in S$ is linearly independent, or (for
polytopes) $\forall\xi\in\Pol$, $\cB_\xi := (u_s:s\in S_\xi)$ is linearly
independent. This condition only depends on the vertices $\xi$ of $\Pol$,
where it means that $\cB_\xi$ is a basis; in particular, each vertex is
$m$-valent.
\end{defn}

Returning to $(N,\Ds,J,\bK)$, the underlying contact manifold $(N,\Ds)$ is
toric under $(\bK,\Ab[N])$ in the sense of~\cite{tcg}, where the following
result is established.

\begin{thm} \label{t:polytope} Let $(N,\Ds,\bK)$ be a \textup(compact,
connected\textup) toric contact manifold under $\Ab[N]$.  Then the stabilizers
in $\Ab[N]$ of points in $N$ are connected \textup(i.e., subtori\textup) and
the fibres of the momentum map $\mu$ on $U_\Ds$ are $\Ab[N]$-orbits. For any
nondegenerate Levi pair $(\g,\lamc)$ the signs of the primitive generators
$v_s\in\ab[N]$ of the circle stabilizers of $N_s:s\in\cS$ may be chosen
uniquely such that the image of the horizontal momentum map $\mu^\lamc\colon
N\to\As\sub\torh^*$ is the compact simple convex polytope $\Pol$ in $\cA$
defined by the affine functions $L_s:=\Ll_N(v_s)$ and $\mu^\lamc$ is a
submersion over the interior of each face.
\end{thm}
In particular, $\mu^\lamc$ induces a poset isomorphism over $\cS$ of $\Cb_N$
with $\Cb_\Pol$.
\begin{cor} $N$ and $\Pol$ \textup(i.e., $\Cb_N$ and $\Cb_\Pol$\textup) have
the same combinatorial type.
\end{cor}

\begin{rem} The primitive generators $v_s:s\in\cS$ need not be linearly
independent in $\ab[N]$, nor even distinct, although after taking a quotient
of $N$ by a subtorus acting freely, we may assume that they span.
\end{rem}

Suppose now that $G$ is a closed subgroup of $\Ab[N]$ with Lie algebra $\g$.
Such a subgroup exists if and only if the lattice of circle subgroups $\Lam_N$
in $\ab[N]$ is mapped to a (rank $m$) lattice $\Lam$ in $\tor$, which holds if
and only if $u_s:s\in\cS$ span a lattice in $\tor$.

\begin{defn} If $\Lam\sub\tor$ is a lattice, then a polytope $\Pol$ or
fan $\Fan$ labelled by $L_s$ \textup(with $u_s=\d L_s$\textup) or
$u_s:s\in\cS$ is \emph{rational} with respect to $\Lam$ if for all $s\in\cS$,
$u_s\in\Lam$.
\end{defn}

We are now ready to study the Levi--K\"ahler quotient $N/G$, which is a
compact toric K\"ahler orbifold $M$ of real dimension $2m$ under an $m$-torus
$\T$ with Lie algebra $\tor$ and hamiltonian generators $\torh$.  Indeed, with
respect to the symplectic form on $N/G$ induced by $\d\eta^\lamc|_\Ds$,
$\mu^\lamc\colon N\to\As\sub\torh^*$ descends to a (natural) momentum map for
the action of $\T=\Ab[N]/G$ on $N/G$, whose image is the the rational simple
convex polytope $\Pol\sub\As$. Rational Delzant theory~\cite{Delzant,LT}
asserts that any such toric K\"ahler orbifold $M$ is determined up to
symplectomorphism or biholomorphism by its labelled polytope or fan. The
construction of $M$ from these data is relevant here, so we review it now.

Let $\Z_\cS$ be the free abelian group generated by $\cS$, let
$\ab=\Z_\cS\otimes_\Z \R$ and $\C_\cS=\Z_\cS\otimes_\Z \C$ be the
corresponding free vector spaces over $\R$ and $\C$, and let
$\Ab=\ab/2\pi\Z_\cS$ and $\Ab^\C=\C_\cS/2\pi\Z_\cS\cong\C_\cS\mult$ be the
corresponding real and complex tori.  Denote the generators of
$\Z_\cS\sub\ab\sub \C_\cS$ by $e_s:s\in \cS$, and observe that $\Ab$ and
$\Ab^\C$ act diagonally on $\C_\cS$, via $[\sum_s t_s e_s]\cdot (\sum_s z_s
e_s)=\sum_s\exp (i t_s) z_s e_s$. The action of $\Ab$ on $\C_\cS$ is
hamiltonian (with respect to the standard symplectic form $\omega_\cS$ on
$\cS$) and has a momentum map $\mm\colon\C_\cS\to \ab^*$ defined by
\begin{equation}\label{eq:std-momentum}
\ip{\mm(z),e_s}=\sigma_s(z)=\tfrac{1}{2}|z_s|^2,
\end{equation}
where $z_s\colon \C_\cS\to \C$ denote the standard (linear) complex
coordinates on $\C_\cS$.

The labellings $s\to L_s$ and $s\mapsto u_s$ of $\Pol$ and $\Fan$ induce, and
are defined by (without loss, surjective) linear maps $\Ll\colon\ab\to\torh$
and $\ul\colon\ab\to\tor$ with $\Ll(e_s)=L_s$ and $\ul(e_s)=u_s$ for all
$s\in\cS$. Let $\tilde\g$ be the kernel of $\ul$; then $\Ll$ determines a
linear form $\tilde\lamc\in\tilde\g^*$ completing the following diagram:
\begin{equation}\label{eq:gc}
\begin{diagram}[size=1.5em,nohug]
0 &\rTo & \tilde\g &\rTo^{\tilde\iota} & \ab &\rTo^{\ul} & \tor &\rTo &0\\
  &   &\dTo^{\tilde\lamc}&    & \dTo^{\Ll} &   & \dEq & \\
0 &\rTo & \R &\rTo^{\afs} & \torh & \rTo^\d & \tor &\rTo &0.
\end{diagram}
\end{equation}
When $\Pol$ (or $\Fan$) is rational, then $\ul$ maps $\Z_\cS$ to $\Lam$ and
hence defines a map from $\Ab=\ab/2\pi\Z_\cS$ to $\tor/2\pi\Lam$ whose kernel
is a closed subgroup $\tilde G$ of $\Ab$ with Lie algebra $\tilde\g$.

The combinatorics $\Cb_\Fan$ of $\Fan$ (or $\Pol$) define an open subset
$\C_\cS\ins\sub \C_\cS$ as the union of $z\in \C_\cS$ for which
$S_z:=\{s\in\cS\st z_s=0\}$ is in $\Cb_\Fan$. In other words, $\C_\cS\ins$ is
the union of the $\Ab^\C$-orbits $\C_{\cS,S}:=\{z\in \C_\cS\st z_s=0 \text{
  iff } s\in S\}$ over $S\in\Cb_\Fan$. Thus the set of $\Ab^\C$ orbits in
$\C_\cS\ins$ is isomorphic to $\Cb_\Fan$, and $S\sub S'$ iff the closure of
$\C_{\cS,S}$ contains $\C_{\cS,S'}$.

\begin{lemma}\label{l:stabG} Let $(\Fan,\ul)$ be a simplicial fan with
combinatorics $\Cb_\Fan$.  Then $\tilde\g\sub\ab$ acts locally freely on
$\C_\cS\ins$. If in addition $(\Fan,\ul)$ is rational, and $\tilde G$ is the
corresponding closed subgroup of $\Ab$, then for any $S\in\Cb_\Fan$, the
stabilizer in $\tilde G$ of any $z\in \C_{\cS,S}$ is $\Stab_{\tilde G}(z)\cong
(\Lam\cap\spns_\R\{u_s\st s\in S\})/\spns_\Z \{u_s\st s\in S\}$.
\end{lemma}
\begin{proof} Let $S\in \Cb_\Fan$ and $z\in \C_{\cS,S}$, so that $z_s=0$ iff $s
\in \cS$. Then any element of the stabilizer of $z$ in $\ab$ has the form
$v=\sum_{s\in S}t_se_s$, which belongs to $\tilde\g$ iff $\sum_{s\in S} t_s
u_s=0$.  However, since $(\Fan,\ul)$ is simplicial, $u_s:s\in S$ is linearly
independent, hence $v=0$.

For the second part, an element $[v]=[\sum_{s\in S}t_se_s]$ of the stabilizer
of $z$ in $\Ab=\ab/2\pi\Z_\cS$ is in $\tilde G$ iff $\sum_{s\in S} t_s
u_s\in\Lam$, and is the identity element iff $t_s\in\Z$ for all $s\in S$.
The result follows.
\end{proof}

The Delzant--Lerman--Tolman correspondence and the relation between symplectic
and complex (GIT) quotients~\cite{Delzant,LT} now assert that:
\begin{numlist}
\item as a complete toric variety, $M$ is a complex (GIT) quotient
  $\C_{\cS}\ins/\tilde G^\C$ of $\C_\cS$ by $\tilde G^\C$;
\item as a compact toric symplectic orbifold, $M$ is a symplectic quotient
  $\tilde N/\tilde G$---where $\tilde N =
  (\tilde\iota\transp\sigma)^{-1}(\tilde\lamc)$---of $\C_\cS$ by $\tilde G$ at
  momentum level $\tilde\lamc\in\tilde\g^*$.
\end{numlist}
Here the labellings of the polytope and fan are related to the orbifold
structure groups of $M$ by Lemma~\ref{l:stabG}. Note that the complex
structure in (i) is biholomorophic to the complex structure of the K\"ahler
quotient in (ii), which is the quotient of the induced CR structure on $\tilde
N\sub \C_\cS$. However, $M$ is not typically \emph{isometric} to the K\"ahler
quotient of $\C_\cS$ by $G$, which is called the \emph{Guillemin
  metric}~\cite{Guillemin,Abreu0,Abreu1} of $\Pol$.

In particular, the Levi--K\"ahler quotient
$(N/G,\d\eta^\lamc\restr{\Ds},J|_\Ds)$ is symplectomorphic to the toric
symplectic orbifold obtained from $(\Pol,\Ll)$ by the Delzant--Lerman--Tolman
construction, but also biholomorphic to the Levi--K\"ahler quotient of the
toric CR submanifold $\tilde N$ of flat space by $(\tilde\g,\tilde\lamc)$. In
general $N$ and $\tilde N$ have different dimensions, but there is a map from
$\Lam$ to $\Lam_N$ sending $e_s$ to $v_s$, and if the latter form a basis for
$\Lam_N$, then we may identify $\Ab[N]$ with $\Ab$, $\Ll_N$ with $\Ll$, and
hence $(\tilde\g,\tilde\lamc)$ with $(\g,\lamc)$. This motivates the study of
Levi--K\"ahler quotients of toric CR submanifolds of flat space, which will
occupy us for the remainder of the paper.

\subsection{Toric CR submanifolds of flat space}\label{s:toric-CR-sub}

Let $\cS$ be a $d$ element set, and define $\Z_\cS$, $\ab$, $\C_\cS$, $\Ab$
and $\Ab^\C$ as in~\S\ref{s:toric-contact}. Let
$\bK\colon\ab\to\ham(\C_\cS,\omega_\cS)$ be the (local) hamiltonian action,
and let $\ang\colon\C\mult_\cS\to\Ab$ be angular coordinates conjugate to the
momentum components~\eqref{eq:std-momentum} on the open set $\C\mult_\cS$ of
$\C_\cS$ where $z_s\neq 0$ for all $s\in \cS$; thus $\d\ang\colon
T\C\mult_\cS\to\ab$ satisfies $\d\ang(K_v)=v$ and $\d\ang(JK_v)=0$ for all
$v\in\ab$. The flat K\"ahler metric in action-angle coordinates on
$\C\mult_\cS$ is then
\begin{equation}\label{eq:flat-metric}
g_\cS = \sum_{s\in \cS}\Bigl(\frac{{\d\sigma_s}^2}{2\sigma_s}
+2\sigma_s {\d\ang_s}^2\Bigr),
\quad\omega_\cS = \sum_{s\in \cS} \d\sigma_s\wedge\d\ang_s,
\quad \d^c\sigma_s:=J\d\sigma_s=2\sigma_s \d\ang_s.
\end{equation}
In particular, the metric $\bH(v,w):=g_\cS(K_v,K_w)$ on the $\Ab$-orbits is
given by the smooth function $\bH=2\delta\mm \colon\C_\cS\to S^2\ab^*$, where
$\delta\colon\ab^*\to S^2\ab^*\sub\ab^*\otimes\ab^*$ is the coproduct dual to
componentwise multiplication in $\ab$; thus if we write $v=\sum_{s\in\cS} v_s
e_s$ and $w=\sum_{s\in\cS} w_s e_s$ then $\bH_z(v,w)=\sum_{s\in\cS}
2\sigma_s(z) v_s w_s$, which is positive definite for $z\in\C\mult_\cS$.  Note
the following crucial property of the flat K\"ahler metric on $\C_\cS$:
$\d^c\sigma_s(K_v)=2\sigma_s v_s$ for all $v\in\ab$, i.e.,
\begin{equation}\label{eq:crucial}
\d^c\mm(K_v)=\bH(v)=2(\delta\mm)(v),
\end{equation}
where we use the natural inclusion $S^2\ab^*\sub\Hom(\ab,\ab^*)$ to evaluate
$\bH=2\delta\mm$ on $v$.

We now restrict attention to Levi--K\"ahler quotients in the following
setting.
\begin{defn}  A \emph{toric CR submanifold} of $\C_\cS$ is a
compact connected CR submanifold $(N,\Ds,J)$ which is invariant and locally
Reeb type under the action of $\Ab$.
\end{defn}
We assume that for any $S\sub\cS$, the intersection of $N$ with $\Ab^\C$ orbit
$\C_{\cS,S}$ is connected; these intersections are then the orbit strata, and
the combinatorics $\Cb_N$ of $N$ may be identified with the poset of
$S\sub\cS$ such that $N\cap \C_{\cS,S}$ is nonempty. We also assume that for
all $s\in\cS$, $\{s\}\in\Cb_N$, i.e., $N_s:=N\cap\C_{\cS,\{s\}}=\{z\in N\st
z_s=0\}$ is nonempty, with generic stabilizer $\mathopen<\exp(t
e_s)\mathclose>$ (if this did not hold, the $\mathopen<\exp(t e_s)\mathclose>$
circle action on $N$ would be free, and we could take a quotient).

We refer to a Levi--K\"ahler quotient $M$ of a toric codimension $\ell$ CR
submanifold $(N,\Ds,J)$ in $\C_\cS$ by a positive Levi pair $(\g,\lamc)$, where
$\g$ is the Lie algebra of an $\ell$-dimensional subtorus $G\sub\Ab$, as a
(codimension $\ell$) \emph{Levi--K\"ahler reduction} of $\C_\cS$.

The data $(N,\Ds,J)$ and $(\g,\lamc)$ are linked by Condition~\ref{cond:trans},
which may be viewed as a constraint on $(N,\Ds,J)$ given $(\g,\lamc)$ or vice
versa. We specify the choice of $(\g,\lamc)$ as in~\S\ref{s:toric-CR} via a
surjective linear map $\Ll\colon \ab\to \torh$, or equivalently, an indexed
family $L_s:s\in\cS$ of vectors in $\torh$ which span (where $L_s=\Ll(e_s)$).
In other words, for toric CR submanifolds $(N,\Ds,J)$ of $\C_\cS$, a pair
$(\g,\lamc)$ is associated canonically, via the set-up~\eqref{eq:gc}, with a
(not necessarily compact or nonempty) convex polytope $\Pol_{\g,\lamc}=
\{\xi\in\As\st \forall\,s\in\cS,\enskip L_s(\xi)\geq 0\}$ labelled (formally)
by $L_s:s\in\cS$ (although some facets $\Fa_s$ could be empty a priori). We
denote the combinatorics of $\Pol_{\g,\lamc}$ by $\Cb_{\g,\lamc}$.

\begin{lemma} Let $N$ be a toric CR submanifold of $\,\C_\cS$ satisfying
Condition~\textup{\ref{cond:trans}} relative to $(\g,\lamc)$. Then there is a
smooth pointwise surjective function $\chi_{N,\g}\colon N\to \Hom(\ab^*,\g)$
such that for all $z\in N$, $\eta_z=\chi_{N,\g}(z)\circ\d^c\mm_z$ and
$\d\eta_z\restr{\Ds_z}=\chi_{N,\g}(z)\circ\d\d^c\mm_z\restr{\Ds_z}$.
\end{lemma}
\begin{proof} The CR submanifold $N$ may be written (at least locally, and in
our examples globally) $N=(F\circ\mm)^{-1}(0)$ where $F\colon\ab^*\to W$ is a
smooth function with values in an $\ell$-dimensional vector space $W$, for
which $0$ is a regular value. Hence $T_z N=\ker \d F_{\mm(z)} \circ \d\mm_z$
and so $\Ds$ is the kernel of the pullback $\bnu$ of $\d F\circ\d^c\mm$ to
$N$, with $\bnu_z(K_v)=\d F_{\mm(z)}(\bH_z(v))$ for $v\in\ab$ and $z\in N$. By
Condition~\ref{cond:trans}, $\d F\circ\bH\circ\iota\colon N\to \Hom(\g,W)$ is
a pointwise isomorphism, and $\eta=(\d F\circ\bH\circ\iota)^{-1}\bnu$.

We may now set $\chi_{N,\g}=(\d F\circ\bH\circ\iota)^{-1}\circ\d F$; this
formula may only be valid locally, but the result is independent of the choice
of $F\colon \ab^*\to W$, so $\chi_{N,\g}$, with $\eta=\chi_{N,\g}\circ\d^c\mm$
is globally defined.  Since $\bnu\restr\Ds=0$, $\d\eta\restr\Ds= (\d
F\circ\bH\circ\iota)^{-1}\d\bnu\restr\Ds$. Now $\d(\d F\circ\d^c\mm)=(\Hess
F)(\d\mm\wedge \d^c\mm) +\d F\circ \d\d^c\mm$. Pulling back to $N$ and
restricting to $\Ds$, the first term vanishes (since $F\circ\mm$ is constant on
$N$). Hence $\d\eta\restr\Ds=\chi_{N,\g}\circ\d\d^c\mm\restr\Ds$.
\end{proof}
The \emph{characteristic function} of $(N,\g,\lamc)$ is the (nowhere
vanishing) function $\chi=\chi_{N,\g,\lamc}\colon N\to\ab$ with
$\chi_{N,\g,\lamc}(z)=\lamc\circ\chi_{N,\g}(z)$. Hence
$\eta^\lamc:=\ip{\eta,\lamc}=\ip{\d^c\mm,\chi_{N,\g,\lamc}}$ and the
horizontal momentum map $\mu^\lamc\colon N\to \cA\sub\torh^*$
satisfies\textup:
\begin{equation}\label{eq:M-momentum}
\ip{\mu^\lamc(z),\Ll(v)}=\eta^\lamc_z(K_v)=\bH_z(v,\chi_{N,\g,\lamc}(z))
=2\ip{\delta\mm(z),v\otimes\chi_{N,\g,\lamc}(z)}.
\end{equation}
Thus $\eta^\lamc=\sum_{s\in\cS} \chi_s\, \d^c \sigma_s$ and
$\ip{\mu^\lamc,\Ll(v)}=\sum_{s\in\cS} 2\sigma_s \chi_s v_s$, i.e.,
$L_s(\mu^\lamc)=2\sigma_s\chi_s$. Since $\d\d^c\sigma_s=2\d\sigma_s\wedge
\d\ang_s$, the induced metric on $\Ds$ (over $\C\mult_\cS\cap N$) is
\begin{equation}\label{eq:Ds-metric}
h_{\Ds,\lamc} = \sum_{s\in\cS} \frac{L_s(\mu^\lamc)}{\sigma_s} \Bigl(
\frac{{\d\sigma_s}^2}{2\sigma_s} +2\sigma_s {\d\ang_s}^2 \Bigr)\Restr{\Ds}
=\sum_{s\in\cS} 2\chi_s \Bigl(
\frac{{\d\sigma_s}^2}{2\sigma_s} +2\sigma_s {\d\ang_s}^2 \Bigr)\Restr{\Ds}.
\end{equation}

Theorem~\ref{t:polytope} shows that if $(\g,\lamc)$ is a nondegenerate Levi
pair, then the image of the horizontal momentum map $\mu^\lamc$ is the compact
simple convex polytope $\Pol$, defined by $\pm L_s:s\in\cS$ for some choice of
signs, and also that $\Pol$ has the same combinatorial type as $N$. Now if
$\Pol=\Pol_{\g,\lamc}$ (i.e., all signs are positive) then
equation~\eqref{eq:Ds-metric} shows that $(\g,\lamc)$ is a \emph{positive} Levi
pair. This motivates the introduction of the following constraint.

\begin{cond}\label{cond:comb} $\Pol_{\g,\lamc}$ is a compact convex polytope
with the same combinatorial type as $N$ (as a subset of $P(\cS)$).
\end{cond}

\begin{thm} \label{t:toric} Let $N$ be a toric submanifold of $\C_\cS$ and
suppose $(\g,\lamc)$ is a Levi pair. Then $\im\mu^\lamc=\Pol_{\g,\lamc}$ if and
only if $(\g,\lamc)$ is a positive Levi pair satisfying
Condition~\textup{\ref{cond:comb}}.
\end{thm}
\begin{proof} If $\im\mu^\lamc=\Pol_{\g,\lamc}$, then~\eqref{eq:Ds-metric}, applied
to each orbit stratum, shows that $h_{\Ds,\lamc}$ is positive definite over the
interior of each face of $\Pol_{\g,\lamc}$, hence everywhere. Thus $(\g,\lamc)$
is a positive Levi pair. Under this assumption, Theorem~\ref{t:polytope} shows
that $\im\mu^\lamc=\Pol_{\g,\lamc}$ if and only if Condition~\ref{cond:comb} holds.
\end{proof}

\subsection{Levi--K\"ahler reduction for quadrics}\label{s:lkr-quadrics}

We now specialize to the case that $N$ is an intersection of quadrics.  For
$N$ to be toric with codimension $\ell$, it is then the level set of an
$\ell$-dimensional family of components of $\mm\colon \C_\cS\to \ab^*$, hence
of the form $(F\circ\mm)^{-1}(0)$, with $F=\iota_o\transp-\lamc_o\colon \ab^*\to
\g_o^*$, where $\iota_o\colon\g_o\into\ab$ is an inclusion of an
$\ell$-dimensional subspace, and $\lamc_o\in \g_o^*$ is in the image of the
positive quadrant of $\ab^*$.

Thus $N=\mu_o^{-1}(\lamc_o)$, where $\mu_o=\iota_o\transp\mm$, is defined by the
same sort of data $(\g_o,\lamc_o)$ as the data $(\g,\lamc)$ which determines the
Levi--K\"ahler structure on $N$. These data may therefore be fixed in the same
way as $(\g,\lamc)$ using a diagram of linear maps
\begin{diagram}[size=1.5em,nohug]
0 &\rTo & \g_o &\rTo^{\iota_o} & \ab & \rTo^{\ul^o} & \tor &\rTo &0\\
  &     & \dTo^{\lamc_o}&   & \dTo^{\Ll^o} &   & \dEq & \\
0 &\rTo & \R&\rTo^{\afs}&\torh & \rTo^\d & \tor &\rTo &0.
\end{diagram}
We write $N_{\g_o,\lamc_o}$ or $N_{\Ll^o}$ for the CR submanifold
corresponding to these data.  We shall assume that $\Pol_{\g_o,\lamc_o}$ is a
compact convex polytope, so that it satisfies Condition~\ref{cond:comb}: the
image of $\mm\colon N\to \ab^*$ thus lies in the nonnegative quadrant of
$\ab^o:=\{\xi\in\ab^*\st \forall\, v\in \g_o, \enskip \xi_s v_s =
0\;\forall\,s\in \cS \Rightarrow v=0\}$, and $u^o_s:s\in\cS$ are the normals
of a complete fan.

Since $F$ is affine linear, $\d F$ is constant, equal to $\iota_o\transp$, and so
$\bnu_z(K_v)=\iota_o\transp(\bH_z(v))$. Hence $\gM_z=\im(\bH_z\circ\iota_o
\colon\g_o\to \ab^*)$ and so $\gM=\mm^*\gM^o$ where $\gM^o\to\ab^o$ has fibre
$\gM^o_\xi= \{(\xi_s v_s)_{s\in\cS}\in \ab^*\st v\in \g_o\}$.

We can satisfy Condition~\ref{cond:trans} by letting $(\g,\lamc)$ equal
$(\g_o,\lamc_o)$.
\begin{prop} On $N_{\g_0,\lamc_0}$, $(\g_o,\lamc_o)$ is a positive Levi pair.
\end{prop}
\begin{proof} Since $u^o_s:s\in\cS$ are the normals of a complete fan, $\{
\alpha\in\tor^*\st\forall\,s\in\cS,\enskip \alpha(u^o_s)\geq 0\}=\{0\}$.
Hence $(\ul^o)\transp(\tor^*)$ meets the positive quadrant of $\ab^*$ only at
$0$, so the image in $\g_o^*$ of this positive quadrant is a strictly convex
cone $\cC$ whose dual cone $\cC_*$ is the intersection of $\g_o$ with the
inverse image (under $\iota_o$) of the positive quadrant in $\ab$.  Since
$\bH_z$ is diagonal and positive definite, it maps the positive quadrant of
$\ab$ onto the positive quadrant of $\ab^*$. Thus
$\iota_o\transp\circ\bH_z\circ \iota_o$ maps $\cC_*$ onto $\cC$. Since
$\lamc_o\in\cC$, $\chi^o(z):=\iota_o
(\iota_o\transp\circ\bH_z\circ\iota_o)^{-1}(\lamc_o)$ has positive components,
and hence $h_{\Ds,\lamc_o}$ is positive definite.
\end{proof}
If the fan associated to $(\g_o,\lamc_o)$ is rational, $(M,J)$ is the
underlying complex orbifold of the Delzant--Guillemin K\"ahler quotient of
$\C_\cS$ by $(\g_o,\lamc_o)$. Its K\"ahler form belongs to the same K\"ahler
class as the Levi--K\"ahler quotient, but will not be the same in general.

\begin{rem} By continuity, we also obtain a positive definite metric for $(
\g,\lamc)$ in an open neighbourhood of $(\g_o,\lamc_o)$.  In particular, we
can fix $\g=\g_o$ and vary $\lamc$ to obtain an $\ell$-dimensional family of
Levi--K\"ahler quotients on the same complex orbifold. As $H^2_{dR}(M)$ is
$\ell$-dimensional, it is natural to ask whether $\lamc$ effectively
parametrizes the K\"ahler cone of $(M,J)$.
\end{rem}

The characteristic function $\chi=\chi_{N,\g,\lamc}$ of $N=N_{\g_o,\lamc_o}$
with respect to an arbitrary Levi pair $(\g,\lamc)$ is given by $\chi(z)
=(\iota\transp\circ\bH_z\circ\iota_o)^{-1}(\lamc)\in\g_o$, where we tacitly
omit the inclusion $\iota_o\colon\g_o\sub\ab$.

\begin{rem} The function $\chi\colon N\to \g_o$ is determined by $\bH(\chi)
\restr\g=\lamc$; since $\bH=2\delta\mm$, it is linear in $\mm$, which implies
$\chi$ is a rational function of $\mm$. Now for any $v\in\g_o$ and $z\in N$,
$\frac12\bH(\sum_{s\in\cS} e_s)= \sum_{s\in\cS}\ip{\sigma_s(z),v}=\lamc_o(v)$,
and so the characteristic function $\chi^o$ of the canonical Levi pair
$(\g_o,\lamc_o)$ is characterized by
$\iota_o\transp\bH_z\bigl(\chi^o(z)-\frac12\sum_{s\in\cS} e_s\bigr)=0$.
\end{rem}
\begin{prop} If $\sum_{s\in\cS} u^o_s=0$ then $2\chi^o_s=1$ for all $s\in\cS$.
\end{prop}
\begin{proof} Since $\sum_{s\in\cS} e_s\in \g_o$, $\chi^o(z)=
\frac12\sum_{s\in\cS} e_s$ for all $z\in N$ by the characterization.
\end{proof}
If this assumption holds, we say $N$ is \emph{spherical}: $N$ is then
contained in a hypersphere in $\C_\cS$. (The equivariant topology of such
manifolds has been studied in~\cite{BM}, but here we focus on the geometry of
their Levi--K\"ahler quotients.) In the spherical case,
$\sigma_s=\ip{\mu^{\lamc_o},L^o_s}$, and the canonical Levi--K\"ahler quotient
metric agrees with the Delzant--Guillemin K\"ahler quotient. In particular,
the reduced metric on $\As\sub\torh^*$ is $\sum_{s\in\cS}(\d L^o_s)^2/2L^o_s$,
which is the pullback by $(\Ll^o)\transp$ of the metric $h_o=\sum_{s\in\cS}{\d
  \zeta_s}^2/2\zeta_s$ on $\ab^*$ , where we write $\zeta_s$ for the linear
function $\zeta_s(\xi)=\xi_s$ on $\ab^*$ corresponding to $e_s\in\ab$ (thus
$\d\zeta_s$ is $e_s$, viewed as a constant $1$-form).

In order to compute the reduced metric for any Levi pair $(\g,\lamc)$, we
observe that the bundle $\gM^o\sub T\ab^o=\ab^o\times\ab^*$ (with
$\gM=\mm^*\gM^o$) is the orthogonal complement to
$\ab^o\times(\ul^o)\transp(\tor^*)$ with respect to $h_o$:
\begin{equation*}
\sum_{s\in\cS}\frac{\d\zeta_s((\xi_sv_s)_{s\in\cS}) \d\zeta_s(w\circ\ul^o)}
{\zeta_s(\xi)}=\sum_{s\in\cS} v_s (w\circ\ul^o)_s= w(\ul^o(v))=0.
\end{equation*}
Hence by Remark~\ref{r:Kq}, we have the following result.

\begin{thm}\label{t:sph-red-metric} Let $(\g,\lamc)$ be a positive Levi pair
on a spherical quadric $N=N_{\g_o,\lamc_o}$. Then the reduced metric of the
Levi--K\"ahler structure is the pullback by $\Ll\transp\colon\Pol_{\g,\lamc}\to
\ab^o$ of the restriction of $h_o$ to $\ab^o\times(\ul^o)\transp(\tor^*)\sub
T\ab^o$ \textup(extended by zero on $\gM^o\sub T\ab^o$\textup).
\end{thm}

\begin{ex}\label{ex:bochner-flat} The weighted projective space $\C P^m_{\bf a}$
of weight ${\bf a}=(a_0, \ldots, a_m)\in \N^{m+1}$ has the structure of a
toric symplectic orbifold whose Delzant polytope is a labelled simplex
$(\Pol,u)$. The corresponding momentum level set $N_\lamc \sub \C^{m+1}$ is CR
$G$-equivariantly isometric to the sphere $\Sph^{2m+1}\sub \C^{m+1}$, acted by
the ${\bf a}$-weighted diagonal $\Sph^1$ action. By a result of
S.~Webster~\cite{webster}, Levi--K\"ahler reduction defines on $\C P^m_{\bf
  a}$ a homothety class of {\it Bochner-flat} K\"ahler metrics~\cite{Bryant,
  david-gauduchon}, which are {\it extremal} K\"ahler metrics~\cite{calabi}.
The Bochner-flat metric coincides with the Guillemin symplectic-K\"ahler
reduction if and only if ${\bf a}=(1, \ldots, 1)$, i.e., only on $\C P^m$, see
e.g., \cite{Abreu1}. Thus one can obtain Levi--K\"ahler quotients of the same
(flat) CR structure on $\Sph^{2m+1}\sub \C^{m+1}$ on any labelled rational
simplex, by varying the subgroup $G\cong \Sph^1$ within a fixed maximal torus
$\T^{m+1}$ in the group $\mathrm{Aut}_{\mathrm{CR}}(\Sph^{2m+1})
=\mathrm{PU}(m+1,1)$ of CR transformations of $\Sph^{2m+1}$.

Locally, the construction is defined by a one-dimensional subspace $\g \sub
\ab$, generated by a non-zero element $v\in \ab$, with corresponding vector
field $K_v$ transverse to the CR distribution on $\Sph^{2m+1}$, and the choice
of a contact form $\eta^v$ with $\eta^v(K_v)=1$ and $\ker\eta^v=\Ds$.  In this
case $(K_v, \eta_v, \Ds, J)$ defines a {\it Sasaki structure} compatible
with the standard CR structure $(\Ds,J)$ on $\Sph^{2m+1}$, see
e.g.~\cite{BGS}. The horizontal K\"ahler geometry $(\d\eta^v, \Ds, J)$ may be
described by a compatible toric metric over a (perhaps not rational) labelled
simplex (see \cite{Eveline2}).  The fact that $\g$ is the Lie algebra of a
subgroup $G\leq\T^{m+1}$ implies a rationality condition on $\g$, hence on the
corresponding labelled simplex.
\end{ex}

\section{Levi--K\"ahler reduction for products of spheres}\label{s:lkr-sph}

Our main motivation for the study of toric Levi--K\"ahler quotients is the
construction of K\"ahler metrics on toric varieties with ``nice'' curvature
properties. For this we observe that CR submanifolds of $\C_\cS$ have local
invariants, and so one approach to constructing Levi--K\"ahler quotients with
nice curvature is to start from a nice CR submanifold $N$ of $\C_\cS$. In
particular, when $N$ is a product of spheres, it is flat as a CR manifold.
Hence we might hope that Levi--K\"ahler quotients of products of spheres have
interesting curvature properties.

\subsection{Products of simplices and products of spheres}

We specialize the set-up of~\S\ref{s:lkr-toric} as follows. Fix positive
integers $\ell$ and $m_1,m_2,\ldots m_\ell$, and let $\cI=\{1,2,\ldots\ell\}$,
$I_i=\{0,1,\ldots m_i\}$, $\cS=\{(i,r)\st i\in\cI$ and $r\in I_i\}$. Let
$m=\sum_{i\in\cI} m_i$ and $d=m+\ell$ as usual.  Thus $\C_\cS\cong
\C^{m_1+1}\times \C^{m_2+1}\times\cdots \times \C^{m_\ell+1} \cong \C^d$ and
$\ab$ has a natural subspace $\g_o= \{x\in \ab\st x_{iq}=x_{ir}$ for all
$i\in\cI$ and $q,r\in I_i\}$. We denote by $x_i$ the common value of the
$x_{ir}$ and thus identify $\g_o$ with $\R^\ell$.  On $\g_o$ we have a natural
linear form $\lamc_o$ sending $(x_1,x_2,\ldots x_\ell)$ to
$x_1+x_2+\cdots+x_\ell\in \R$, and we let $\Ll^o\colon \ab\to
\torh=\ab/\ker\lamc_o$ and $\ul^o\colon \ab\to \tor=\ab/\g_o$ be the quotient
maps.

Under the canonical identification of $\ab^*$ with $\ab$, $\torh^*$ is
isomorphic to the subspace of $\xi=(\xi_s)_{s\in \cS}$ such that $\sum_{r\in
  I_i} \xi_{ir}$ is independent of $i$, this constant being the natural
projection $\torh^*\to \R$. Hence $\tor^*$ is a product (over $i\in\cI$) of
the codimension one linear subspaces of $\R^{m_i+1}$ where the coordinate sum
is zero, and $\cA$ is the corresponding product of affine subspaces $\cA_i$
where the coordinate sum is one.

\begin{notn}[The faces of $\Sig$] \label{facesSIGMA} The polytope $\Sig$
in $\torh^*$ defined by $\Ll^o$ is a product of simplices $\Sig_i$ in the
affine spaces $\cA_i$. In the following, we sometimes write $i(s)$ and $r(s)$
for the components of $s\in\cS$.
\begin{bulletlist}
\item The facets of $\Sig_i$ are $\Fa^i_r=\{\xi_{ir} =0\}\cap \Sig_i$ for $r\in
  I_i$.
\item The vertices of $\Sig_i$ may also be indexed by $r\in I_i$: we let
  $p^i_r$ be the unique vertex of $\Sig_i$ that is not in $\Fa^i_r$.
\item The vertices of $\Sig$ are thus indexed by $(r_1,\ldots r_\ell)\in
  I_1\times \cdots\times I_\ell$:
\[
p_{(r_1,\ldots r_\ell)} = (p^1_{r_1},\ldots p^\ell_{r_\ell}).
\]
\item Each facet $\Fa^i_r$ (for $r\in I_i$) of the simplex $\Sig_i$ determines
    a facet
\[
\Fa_{ir}= \Sig_1\times\cdots\times\Sig_{i-1}\times \Fa^i_r \times\Sig_{i+1}
\times\cdots\times\Sig_\ell
\]
of $\Sig$ and a corresponding inward normal $u^o_{ir}=\d L^o_{ir}$.
\end{bulletlist}
\end{notn}

The corresponding CR submanifold of $\C_\cS$ is
\begin{equation*}
N=N_{\Ll^o}=\{z\in \C_\cS\st {\textstyle\sum_{r\in I_i} \sigma_{ir}}=1
\text{ for all }i\in \cI\},
\end{equation*}
where $\sigma_{ir}=\frac12|z_{ir}|^2$. Thus $N\cong \Sph^{2m_1+1}\times \cdots
\times \Sph^{2m_\ell+1}$. As in~\S\ref{s:lkr-quadrics}, $N$ is the level set, at
the regular value $\lamc_o$, of the momentum map $\mu_o=\iota_o\transp\mm$. Thus
\begin{equation*}
\mu_o (z)=(\mm^1(z), \ldots \mm^{\ell}(z)), \quad\text{where}\quad \mm^i(z)=
\sum_{r\in I_i}\sigma_{ir}(z)=\sum_{r\in\ I_i} \tfrac12 |z_{ir}|^2,
\end{equation*}
and we denote $z=(z^1, \ldots z^{\ell})$ with $z^i=(z_{i0}, \ldots, z_{im_i})$
the linear coordinates of $\C_\cS$.  These data are associated to the Delzant
construction for the product $\Sig= \Sig_1\times\cdots\times \Sig_{\ell}
\sub\cA$ of standard Delzant simplices $\Sig_i\sub\cA_i$.  More specifically,
$\g_o$ is the Lie algebra of a subtorus $G_o$ of $\Ab$, which acts freely on
$N$ preserving the CR structure $(\Ds, J)$, with quotient space $(M,J) = \C
P^{m_1} \times \cdots \times \C P^{m_{\ell}}$.  The Lie algebra $\g_o$ of
$G_o$ defines a Reeb foliation on $N\sub \C_\cS$ with induced horizontal Levi
structure consisting of scales of product of Fubini--Study metrics.

\begin{thm}\label{t:product-spheres}
Let $N= \Sph^{2m_1+1} \times \cdots \times \Sph^{2m_{\ell}+1} \sub \C_\cS$ be a
product of standard CR spheres. Then for a pair $(\g,\lamc)$, defined by
$\Ll\colon\ab\to\torh$, with associated polytope $\Pol_{\g,\lamc}\sub\As$, the
following are equivalent.
\begin{numlist}
\item $(\g,\lamc)$ is a positive Levi pair, i.e., defines a Levi--K\"ahler
  structure on $N$.
\item $(\g,\lamc)$ is a Levi pair \textup(i.e., $\g$ satisfies
  Condition~\textup{\ref{cond:trans}}\textup) whose horizontal momentum map 
$\mu^\lamc\colon N\to \As$ has $\im\mu^\lamc=\Pol_{\g,\lamc}$.
\item $(\g,\lamc)$ satisfies Condition~\textup{\ref{cond:comb}}, i.e.,
  $\Pol_{\g,\lamc}$ is a compact convex polytope with the same combinatorial
  type as $\Sig$.
\end{numlist}
\end{thm}
The proof makes use a couple of Lemmas.

\begin{lemma}\label{l:cond-trans} If $(\g,\lamc)$ satisfies
Condition~\textup{\ref{cond:comb}}, then it satisfies
Condition~\textup{\ref{cond:trans}}.
\end{lemma}
\begin{proof} Condition~\ref{cond:trans}(i) holds since it only depends on
the combinatorics of $\Pol$.

Suppose that $\g$ does not satisfy Condition~\ref{cond:trans}(ii). Then there
exist $z\in N$ and $v=(x_s)\in \ab\setminus 0$ such that
$K_v(z)\in\Ds=\bigcap_{i\in\cI} \ker\d^c\mm^i$ and $v\in \g$, that is
\begin{equation}\label{syst(ii)}
\sum_{s\in\cS} x_s u_s =0\in \tor \;\;\mbox{ and }\;\;
\sum_{r\in I_i} x_{ir} \sigma_{ir} =0 \qquad \mbox{ for } i\in\cI,
\end{equation}
where $\sigma_s=\frac12 |z_s|^2$. As equations on $v=(x_s)$ for fixed $z$,
this system is a linear map from $\ab$ to $\R^\ell\oplus\tor$, which both have
dimension $d$.  We may write the $d\times d$-matrix $A=A_z$ of this linear map
as follows: for $j\in \cI$ the $j$-th row is
$\sigma_s\delta_{i(s)j}$, while the lower part $B$ of the matrix is the
$m\times d$-matrix, whose $s$-th column is $u_s$ for $s\in\cS$ (written with
respect to some basis of $\tor$). We compute the determinant of $A=A_z$ by
expanding along the first $\ell$ rows. The nonzero terms are all obtained
by choosing, for each $j\in \cI$, $r_j\in I_j$ to obtain a minor
\begin{equation*}
\pm\sigma_{1 r_1}\sigma_{2 r_2}\cdots \sigma_{\ell r_\ell} \det B_{(r_1,\ldots r_\ell)},
\end{equation*}
where $B_{(r_1,\ldots r_\ell)}$ is the submatrix of $B$ obtained by removing
the columns $u_{1 r_1},\ldots u_{\ell r_\ell}$. Up to an overall sign
(depending on $|I_j|$) each such minor contributes to $\det A$ with sign
$(-1)^{\sum_{j\in\cI} r_j}$. Hence to show $\det A\neq 0$, it suffices to
show that (for a fixed basis of $\tor$) $(-1)^{\sum_{j\in\cI} r_j} \det
B_{(r_1,\ldots r_\ell)}>0$ for all $(r_1,\ldots r_\ell)\in I_1\times\cdots\times
I_\ell$, because $\sigma_s\geq0$ and the products $\prod_{j\in\cI} \sigma_{jr_j}$
do not all vanish at the same time.

Since $\Pol$ has the same combibatorial type as $\Sig$, we know that the
colomns of $B_{(r_1,\ldots r_\ell)}$ are inward normals of the facets meeting at
the vertex $p_{(r_1,\ldots r_\ell)}$, see Notation~\ref{facesSIGMA}. These form a
basis by the Delzant condition on $\Pol$, and so it suffices to show that
$(-1)^{\sum_{j=1}^lr_j}$ times the wedge product of the columns of
$B_{(r_1,\ldots r_\ell)}$ has sign independent of $(r_1,\ldots r_\ell)$. This will
hold for $\Pol$ if it holds for $\Sig$, so it suffices to check that for
each $j\in \cI$, $(-1)^{r_j} u^o_{j0}\wedge\cdots \hat
u^o_{jr_j}\wedge\cdots u^o_{j m_j}$ (with the $u^o_{jr_j}$ factor omitted) is
independent of $r_j\in I_j$. Since $\sum_{r\in I_j} u^o_{jr}=0$, this is a
triviality.
\end{proof}
\begin{lemma}\label{l:positivity} Suppose $(\g,\lamc)$ satisfies
Condition~\textup{\ref{cond:trans}} and let $\chi\colon N\to \g_o\cong\R^\ell$
be the characteristic function of $(N,\g,\lamc)$. Then $(\g,\lamc)$ is a
positive Levi pair if and only if $\chi_1,\ldots \chi_\ell$ are positive.
\end{lemma}
\begin{proof} First observe that since $\Ds=\bigcap_{i\in\cI}\ker\d^c\mm^i$,
we have
\begin{equation}\label{eq:eta-c}
\eta^\lamc = \sum_{i\in\cI} \chi_i\, \d^c \mm^i, \quad
\d\eta^\lamc\restr\Ds = \sum_{i\in\cI} \chi_i\, \d\d^c \mm^i\restr{\Ds}.
\end{equation}
Moreover, $\Ds$ splits as a sum $\Ds=\bigoplus_{i\in\cI}\Ds_i$ where $\Ds_i$
is tangent the $i$-th sphere in the product $N= \prod_{i\in\cI} \Sph^{2m_i
  +1}$ (that is $\Ds_i =T\Sph^{2m_i +1}\cap JT\Sph^{2m_i +1}$). For $i\in\cI$,
$\d\d^c\mm^i$ is nondegenerate on $\Ds_i$, and if $j\neq i$, $\Ds_j\sub
\ker \d\d^c\mm^i$.  Thus $\d\eta^\lamc\restr{\Ds}$ defines a positive definite
metric iff $\chi_i>0$ for all $i\in \cI$.
\end{proof}

Before proving the theorem we need a bit more notation.  For each $j\in \cI$,
we define
\[
C_j^{\pm} :=\{\xi\in \As \st \forall\, r\in I_j,\;\pm L_{jr}(\xi)\geq 0\}
\]
where $C_j^{-}$ is potentially empty but $C_j^{+}$ is not, and $\Pol =
\bigcap_{j\in \cI} C_j^+$.

\begin{proof}[Proof of Theorem~\textup{\ref{t:product-spheres}}] By
Lemma~\ref{l:cond-trans}, we may assume Condition~\ref{cond:trans} holds.
Then the formula~\eqref{eq:M-momentum} for the induced momentum map
$\mu^\lamc$ here reduces to
\begin{equation}\label{eq:PS-momentum}
L_{ir}(\mu^\lamc(z))= 2\chi_i(z) \sigma_{ir}(z)
\end{equation}
where $L_{ir}=\Ll(e_{ir})$ for the standard basis $e_{ir}$ of $\ab$.

If $(\g,\lamc)$ is a positive Levi pair, the functions $\chi_i$ are positive by
Lemma~\ref{l:positivity} and then equation~\eqref{eq:PS-momentum} and
Theorem~\ref{t:polytope} imply that $\im\mu^\lamc=\Pol_{\g,\lamc}$; thus
(i)$\Rightarrow$(ii).

Now (ii)$\Rightarrow$(i)\&(iii) by Theorem~\ref{t:toric}, which also
shows (i)$\Rightarrow$(iii).

Finally, to prove (iii)$\Rightarrow$(i), it suffices, by
Lemma~\ref{l:positivity}, to show that the functions $\chi_i$ are positive.
First note the following consequences of equation~\eqref{eq:PS-momentum}.
\begin{itemize}
\item[(a)] $\im\mu^\lamc$ contains all the vertices of $\Pol=\Pol_{\g,\lamc}$.
  Indeed, each $z\in N$ having only one nonzero coordinates in each spherical
  factor is sent to a vertex of $\Pol$. Moreover, on the vertices of $\Pol$,
  $L_{s'}\geq 0$ for each $s'\in\cS$; thus equation~\eqref{eq:PS-momentum}
  implies that $\chi_i(z)\geq 0$ for any $z\in N$ such that $\mu^\lamc(z)$ is a
  vertex of $\Pol$.
\item[(b)] If $L_{ir}(\mu^\lamc(z))\geq 0$ (resp. $L_{ir}(\mu^\lamc(z))\leq 0$)
then for all $q\in I_i$, $L_{iq}(\mu^\lamc(z))\geq 0$ (resp. $L_{iq}(\mu^\lamc(z))
\leq 0$).  That is $\mu^\lamc(z)\in \bigcap_{j\in\cI} (C_j^+\cup C_j^-)$.
\end{itemize}
Thanks to the statement (a) above, it is sufficient to prove that none of the
$\chi_i$'s vanishes on $N$. From statement (b) we have the following inclusion
\[
\im\mu^\lamc \sub \bigcap_{j\in\cI} (C_j^+\cup C_j^-) = \bigcup_{\cJ \sub \cI}
\Pol_{\cJ},
\]
where $\Pol_{\cJ} = \left(\bigcap_{j\in \cJ} C_j^-\right) \cap
\left(\bigcap_{j\in \cJ^c} C_j^+\right)$ with $\cJ^c := \cI \backslash \cJ$
(so $\Pol_{\emptyset}=\Pol$).

Statement (a) implies that $\Pol\cap\im\mu^\lamc$ is not empty, but this image
is connected and for $\cJ$ nonempty, $\Pol_\cJ$ does not meet $\Pol$. Hence
$\im\mu^\lamc$ is contained in $\Pol$.  However, if $\chi_i(z)=0$ for some
$i\in\cI$ and $z\in N$, then $L_{ir}(\mu^\lamc(z))=0$ for all $r\in I_i$,
contradicting the combinatorial type of $\Pol$.
\end{proof}

\begin{cor} Any toric symplectic orbifold whose rational Delzant
polytope $(\Pol,\Ll)$ has the combinatorics of a product of simplexes admits a
compatible toric Kahler metric $h_\Ll$ which is a Levi--K\"ahler reduction of
a product of spheres.
\end{cor}

We now give a closed formula for the symplectic potential of $h_\Ll$.

\begin{thm}\label{t:symplectic-potential} Let $N=\Sph^{2m_1+1}\times\cdots
\times\Sph^{2m_{\ell}+1}\sub\C_\cS$ be a product of standard CR spheres, and
suppose that the kernel $\g$ of $\ul=\d\circ \Ll$ satisfies
Condition~\textup{\ref{cond:trans}}.  Then
\begin{align*}
G_\Ll=\frac12\sum_{i\in\cI}\sum_{r\in I_i\cup\smash{\{\infty\}}} L_{ir} \log|L_{ir}|
=\frac12\sum_{i\in\cI}\sum_{r\in I_i}L_{ir}\log\,\Bigl|\frac{L_{ir}}{L_{i\infty}}\Bigr|
\end{align*}
is a symplectic potential for the Levi--K\"ahler metric, where
$L_{ir}\in\torh$ is viewed as a linear function on $\torh^*$, hence an affine
function on $\As\sub\torh^*$, and $L_{i\infty}=-\sum_{r\in I_i} L_{ir}$.
Equivalently, the reduced metric on the image of the horizontal momentum map
$\mu^\lamc$ is given by
\begin{equation*}
h_\Ll\red= \frac12 \sum_{i\in\cI}\sum_{r\in I_i\cup\smash{\{\infty\}}}
\frac{{\d L_{ir}}^2}{L_{ir}} =\frac12 \sum_{i\in\cI} \sum_{0\leq r<s\leq m_i}
\frac{L_{ir} L_{is}}{\sum_{t=0}^{m_i} L_{it}}\biggl(\frac{\d L_{ir}}{L_{ir}}
- \frac{\d L_{is}}{L_{is}}\biggr)^2.
\end{equation*}
\end{thm}
\begin{proof} The hessian of the stated potential $G_\Ll$ evaluates readily to
the stated reduced metric, which we can compute in two ways. Using
Theorem~\ref{t:sph-red-metric}, we decompose
\begin{equation*}
h_o:= \sum_{s\in\cS}\frac{{\d\zeta_s}^2}{2\zeta_s} =\frac12 \sum_{i\in\cI}
\biggl(\sum_{r\in I_i} \frac{{\d \zeta_{ir}}^2}{\zeta_{ir}} -
\frac{\bigl(\sum_{r\in I_i} \d\zeta_{ir}\bigr){}^2}{\sum_{r\in
    I_i}\zeta_{ir}}\biggr) +\frac12\sum_{i\in\cI}\frac{\bigl(\sum_{r\in I_i}
  \d\zeta_{ir}\bigr){}^2} {\sum_{r\in I_i}\zeta_{ir}}
\end{equation*}
into components orthogonal and parallel to $\gM^o$, and the reduced metric is
the pullback of the first term by $\Ll$.  Alternatively,
by~\eqref{eq:Ds-metric}, the horizontal metric on $J(\Rb{}\cap\Ds)$ is
\begin{equation*}\begin{split}
\sum_{i\in\cI} 2\chi_i \sum_{r\in I_i}\frac{{\d \sigma_{ir}}^2}{2\sigma_{ir}}\Restr{\Ds}
&=\mm^*\biggl[\frac12\sum_{i\in\cI} \rho_i
\biggl(\sum_{r\in I_i} \frac{{\d \zeta_{ir}}^2}{\zeta_{ir}}
-\frac{\bigl(\sum_{r\in I_i} \d \zeta_{ir}\bigr){}^2} {\sum_{r\in I_i}
  \zeta_{ir}}\biggr)\biggr]\Restr{\Ds}\\
&= \frac12 \sum_{i\in\cI} \mm^*\biggl(\sum_{r\in I_i} \frac{\d
    (\rho_i\zeta_{ir}){}^2}{\rho_i\zeta_{ir}} -\frac{\bigl(\sum_{r\in I_i} \d
  (\rho_i\zeta_{ir})\bigr){}^2}{\sum_{r\in I_i} \rho_i\zeta_{ir}}\biggr)\Restr{\Ds}
\end{split}\end{equation*}
where $\mm^*\rho_i=2\chi_i$ and we use that $\sum_{r\in I_i}\sigma_{ir}$ is
constant on $N$, so that $\mm^*\bigl(\sum_{r\in I_i}\d\zeta_i\bigr)=0$, and
then exploit rescaling invariance. The result is the pullback by $\mu^\lamc$
of the stated reduced metric, since $L_{ir}(\mu^\lamc) = 2\chi_i\sigma_{ir}$
by~\eqref{eq:PS-momentum}.
\end{proof}
As shown by Guillemin~\cite{Guillemin}, a K\"ahler potential may be computed
as a Legendre transform of the symplectic potential $G_\Ll$ with respect to
some basepoint $p\in\cA$:
\begin{align*}
H_\Ll&=\ip{\mu^\lamc-\mu^\lamc(p),\d G_\Ll}-G_\Ll\\
&= \sum_{i\in\cI} \sum_{r\in I_i\cup\smash{\{\infty\}}}
\Bigl(\tfrac12\bigl(\ip{\mu^\lamc-\mu^\lamc(p),\d L_{ir}}-L_{ir}\bigr)\log |L_{ir}|
+\tfrac12 \ip{\mu^\lamc-\mu^\lamc(p),\d L_{ir}}\Bigr)\\
&=\sum_{i\in\cI}\sum_{r\in I_i\cup\smash{\{\infty\}}} \tfrac12 L_{ir}(p) \log |L_{ir}|,
\qquad\text{since}\quad
\sum_{r\in I_i\cup\smash{\{\infty\}}} L_{ir}=0.
\end{align*}

\subsection{Products of 3-spheres}\label{s:S3s}

As a special case of Theorem~\ref{t:product-spheres}, consider an $\ell$-fold
product $N=\Sph^3\times\cdots\times \Sph^3$ of $3$-spheres as a codimension
$\ell$ submanifold of $\C^{2\ell}\cong \C^\ell\otimes\C^2$ with momentum
coordinates $\sigma_{ir}(z)$ ($i\in\cI=\{1,\ldots \ell\}$, $r\in\{0,1\}$).
Thus $\ab=\R^{2\ell}$ is the Lie algebra of $\Ab=\T^{2\ell}$ acting diagonally
on $\C^{2\ell}$. We study the Levi--K\"ahler metric on the open subset where
the quotient torus $\T$ acts freely, ignoring rationality conditions.

\subsubsection{Geometry of the Levi--K\"ahler metric} Consider, for any
$\ell$-dimensional subspace $\g\sub\ab=\R^{2\ell}$, the integrable
distribution $\Rb\g= \spns\{K_v\st v\in \g\}$ on $N$.  Then, around each
point of $N$ such that $\Rb\g$ is transversal to $\Ds$, the local quotient
space $M$ of leaves of $\Rb\g$ has induced complex structure $J$. We will
further assume that $\lamc\in\g^*$ is such that $\d\eta^\lamc$ induces a
K\"ahler metric $(h_\Ll, J, \omega_\Ll)$ on $M$.

The reduced metric of Theorem~\ref{t:symplectic-potential} specializes to
\begin{equation*}
h_\Ll\red=\frac12 \sum_{i\in\cI} \frac{L_{i0} L_{i1}}{L_{i0}+L_{i1}}
\biggl(\frac{\d L_{i0}}{L_{i0}} - \frac{\d L_{i1}}{L_{i1}}\biggr)^2,
\quad\text{i.e.,}\quad (\mu^\lamc)^*h_\Ll\red
=\sum_{i\in\cI} \frac{\chi_i\,{\d \sigma_i}^2}{\sigma_i(1-\sigma_i)},
\end{equation*}
where $L_{ir}(\mu^\lamc)=2\chi_i\sigma_{ir}$,
$2\chi_i=-L_{i\infty}(\mu^\lamc)$,
$\sigma_i:=\sigma_{i0}=-L_{i0}(\mu^\lamc)/L_{i\infty}(\mu^\lamc)$,
$L_{i0}+L_{i1}= -L_{i\infty}$ and hence $\sigma_{i1}=1-\sigma_i$ on $N$.  In
other words, the characteristic functions $\chi_i$ and the orthogonal
coordinates $\sigma_i$ are affine and birational functions (respectively) of
the momentum coordinates $\mu^\lamc$.  Thus the momentum images of the
coordinate hypersurfaces with $\sigma_i$ constant are hyperplanes in $\As$
through the codimension two affine subspace where
$L_{i0}(\mu^\lamc)=L_{i1}(\mu^\lamc)=0$, and $\sigma_i$ is inverse to the
unique affine coordinate on this pencil of hyperplanes sending $0$, $1$ and
$\infty$ to the facets $L_{i0}(\mu^\lamc)=0$, $L_{i1}(\mu^\lamc)=0$ (of
$\Pol_{\g,\lamc}$) and the ``characteristic hyperplane'' $\chi_i
=-\frac12L_{i\infty}(\mu^\lamc)=0$ (respectively). These pencils introduce a
factorization structure (in the sense of \cite{ACG2}) which is adapted to the
class of polytopes in $\As$ with the combinatorics of the product of
intervals.

To allow for more general coordinates, set $\cH=\{0,1,\infty\}$ (so that
$\sum_{r\in\cH} L_{ir}=0$) and introduce an arbitrary affine coordinate
$\xi_i=-N_{i0}(\mu^\lamc)/N_{i\infty}(\mu^\lamc)$ on the pencil which takes
the values $\alpha_{ir}$ at the points $[L_{ir}]$, meaning
\begin{equation*}
\sigma_i=\frac{(\xi_i-\alpha_{i0})(\alpha_{i1}-\alpha_{i\infty})}
{(\xi_i-\alpha_{i\infty})(\alpha_{i1}-\alpha_{i0})},\quad
1-\sigma_i=\frac{(\xi_i-\alpha_{i1})(\alpha_{i0}-\alpha_{i\infty})}
{(\xi_i-\alpha_{i\infty})(\alpha_{i0}-\alpha_{i1})},\quad
L_{ir}= \frac{2(\xi_i-\alpha_{ir})N_{i\infty}}{A_i'(\alpha_{ir})}
\end{equation*}
for affine functions $N_{ir}$ of $\mu^\lamc$ with $\sum_{r\in\cH} N_{ir}=0$,
and $A_i(y)=a_i\prod_{r\in\cH}(y-\alpha_{ir})$. Note that we also allow
$\alpha_{i\infty}=\infty$, in which case $A_i$ is of degree $2$. (This latter
case can be derived from the generic case $\deg A_i=3$ by letting
$\alpha_{i\infty} = 1/\varepsilon$ and taking a limit of $\varepsilon A_i$ as
$\varepsilon \to 0$.) Thus
\begin{equation*}
\d \sigma_i=\frac{(\alpha_{i1}-\alpha_{i\infty})(\alpha_{i0}-\alpha_{i\infty})}
{(\xi_i-\alpha_{i\infty})(\xi_i-\alpha_{i\infty})^2},\quad
\frac{\chi_i\,{\d\sigma_i}^2}{\sigma_i(1-\sigma_i)}
=\frac{L_{i\infty}(\mu^\lamc)\,A_i'(\alpha_{i\infty}){\d \xi_i}^2}
{2(\xi_i-\alpha_{i\infty})A_i(\xi_i)}
=N_{i\infty}(\mu^\lamc)\,\frac{{\d \xi_i}^2}{A_i(\xi_i)}.
\end{equation*}
On the other hand,
\begin{align*}
\d\xi_i&=\frac{N_{i0}(\mu^\lamc) \d N_{i\infty} -N_{i\infty}(\mu^\lamc) \d N_{i0}}
{N_{i\infty}(\mu^\lamc)^2}\circ\d\mu^\lamc
=-\frac{(\xi_i \d N_{i\infty}+ \d N_{i0})\circ\d\mu^\lamc}
{N_{i\infty}(\mu^\lamc)}\\
\therefore\quad\d\mu^\lamc&=-\sum_{j\in\cI}N_{j\infty}(\mu^\lamc)\d\xi_j\otimes\Qt_j,
\quad\text{where} \quad \Qt_i\colon M\to\tor^* \quad\text{satisfy}\\
\sum_{i\in\cI} &(\xi_i \d N_{i\infty}+ \d N_{i0})\otimes \Qt_i=\Id_\tor\quad
\text{and}\quad \ip{(\xi_i \d N_{i\infty}+ \d N_{i0}),\Qt_j}=\delta_{ij}.
\end{align*}
Hence $\d\Qt_i = -\sum_j \ip{\Qt_i,\d N_{j\infty}}\d\xi_j\otimes\Qt_j$ and the
Levi--K\"ahler metric is
\begin{equation}\label{eq:gen-s3s}
\begin{split}
h_\Ll&=\sum_{i\in\cI}N_{i\infty}(\mu^\lamc)
\biggl(\frac{{\d \xi_i}^2}{A_i(\xi_i)}+A_i(\xi_i){\theta_i}^2\biggr),\\
\omega_\Ll&=\ip{\d\mu^\lamc\wedge \d\bt}=-\sum_{i\in\cI} N_{i\infty}(\mu^\lamc)
  \d\xi_i\wedge\theta_i,
\end{split}
\end{equation}
with $\theta_i=\ip{\Qt_i,\d\bt}$ for angular coordinates $\bt$ on $M$ such
that $\d\bt^{(1,0)}$ is holomorphic:
\begin{equation*}
\d\d^c\bt =\d\biggl(\sum_{i\in\cI}\frac{\xi_i\d N_{i\infty}+\d N_{i0}}{A_i(\xi_i)}
\d\xi_i\biggr)=0.
\end{equation*}
The canonical affine coordinates may be obtained by setting
$A_i(t)=2t(1-t)(1-\eps t)$ in the limit $\eps\to 0$; we then have
$N_{ir}=-L_{ir}$ and $\xi_i=\sigma_i$ so that $N_{i\infty}(\mu^\lamc)=2\chi_i$
and
\begin{equation}\label{eq:can-s3s}
\begin{split}
h_\Ll&=\sum_{i\in\cI}2\chi_i\,
\biggl(\frac{{\d \sigma_i}^2}{2\sigma_i(1-\sigma_i)}+2\sigma_i(1-\sigma_i)
{\theta_i}^2\biggr),\\
\omega_\Ll&=-\sum_{i\in\cI} 2\chi_i\,\d\sigma_i\wedge\theta_i.
\end{split}
\end{equation}

\subsubsection{Explicit normal form} Recall that $\g_o\sub\ab$ has a canonical
basis $e_i:=e_{i0}+e_{i1}$ for $i\in\cI$.  We obtain a similar basis
$w_i:i\in\cI$ for $\g$ by observing that the vectors $e_{i0}$ in $\ab$ project
to $\tor=\ab/\g_o\cong \ab/\g$ to the normals $u_{i0}:i\in\cI$ at a vertex of
$\Pol_{\g,\lamc}$. Hence $\g$ is transversal to $\spns\{e_{i0}\st i\in\cI\}$
and there are canonical $w_i\in\g$ of the form
\begin{equation}\label{eq:basis}
w_i = e_i + \sum_{j\in\cI} C_{ji} e_{j0} = e_{i0}+e_{i1}+
\sum_{j\in\cI} C_{ji} e_{j0}
\end{equation}
for an $\ell\times\ell$ matrix of real numbers $C$. We denote by
$K_i^o=K_{e_i}$ and $K_i=K_{w_i}$ the induced vector fields in $\Rb{\g_o}$ and
$\Rb{\g}$ respectively. We shall also write $K_{ir}$ as a shorthand for
$K_{e_{ir}}=\partial/\partial\ang_{ir}$. Thus
\begin{align*}
\d^c\mm^i(K_j^o)&=2(\sigma_{i0}\d\ang_{i0}+\sigma_{i1}\d\ang_{i1})(K_{j0}+K_{j1})
=2\delta_{ij}\\
\d^c\mm^i(K_j)&=2(\sigma_{i0}\d\ang_{i0}+\sigma_{i1}\d\ang_{i1})\bigl(K_j^o
+{\textstyle\sum}_k C_{kj} K_{k0}\bigr)=2(\delta_{ij} + \sigma_i C_{ij}).
\end{align*}
on $N$, since $\sigma_{i0}+\sigma_{i1}=1$. Now if $\eta_i:i\in\cI$ are the
$1$-forms on $N$ defined by $\eta_i(K_j)= \delta_{ij}$ and
$\bigcap_{i\in\cI}\ker\eta_i=\Ds= \bigcap_{i\in\cI}\ker\d^c\mm^i$, we may
write
\begin{equation}
\tfrac12 \d^c\mm^i = \sum_{j\in\cI} P_{ij} \eta_j
\quad\text{where}\quad
P_{ij}:= \tfrac12 \d^c\mm^i(K_j)=\delta_{ij} + \sigma_i C_{ij}.
\end{equation}
We have noted that $e_{i0}:i\in\cI$ project onto a bases for
$\tor=\ab/\g_o\cong \ab/\g$. In order to compute the toral part of the
quotient metric, we need to project the corresponding vector fields $K_{i0}$
onto $\Ds$. We thus define projections
\begin{align*}
X_i^o&= K_{i0}-\sum_{j\in\cI} \tfrac12\d^c\mm^j(K_{i0}) K_j^o=
K_{i0} - \sigma_i (K_{i0}+K_{i1}) = (1-\sigma_i)K_{i0}-\sigma_i K_{i1},\\
X_i&= K_{i0}- \sum_{j\in\cI} \eta_j(K_{i0}) K_j,
\end{align*}
along $\Rb{\g_o}$ and $\Rb{\g}$ respectively. Let $I_\sigma=
\diag(\sigma_1,\ldots\sigma_\ell)$.
\begin{lemma} $X_i^o=\sum_{j\in\cI} X_j \tilde P_{ji}$ with
$\tilde P_{ji}=\delta_{ji}+C_{ji}\sigma_i$, i.e., $PI_\sigma= I_\sigma\tilde P$.
\end{lemma}
\begin{proof} Since $\sum_{k\in\cI} P_{jk}\eta_k(K_{i0})=\frac12\d^c\mm^j(K_{i0})
=\sigma_i\delta_{ij}$, it follows from the relation between $P$ and $\tilde P$
that $\frac12\d^c\mm^j(K_{i0})=\sum_{k\in\cI} \eta_j(K_{k0})\tilde
P_{ki}$. We now have
\begin{align*}
X_i^o&=K_{i0}-\sum_{j\in\cI} \tfrac12 \d^c\mm^j(K_{i0})\Bigl(K_j -
\sum_{k\in\cI} C_{kj} K_{k0}\Bigr)\\
&=K_{i0}-\sum_{j,k\in\cI} \Bigl(\eta_j(K_{k0})\tilde P_{ki} K_j
-\sigma_j\delta_{ij} C_{kj} K_{k0}\Bigr)\\
&=\sum_{k\in\cI} K_{k0}(\delta_{ki}+C_{ki}\sigma_i)
-\sum_{j,k\in\cI} \eta_j(K_{k0})K_j \tilde P_{ki} = \sum_{k\in\cI} X_k \tilde P_{ki}
\end{align*}
as required.
\end{proof}
Writing $\lamc=(c_1,\ldots c_\ell)$ with respect to the dual basis to
$w_i:i\in\cI$, $\eta^\lamc=\sum_{i\in\cI} c_i \eta_i$. The corresponding
momentum coordinates $\mu_i:i\in\cI$ are given by (cf.~\eqref{eq:M-momentum})
\begin{equation}\label{eq:mu-sigma}
\mu_i = \ip{\mu^\lamc,L_{i0}} = \sum_{j\in\cI} c_j \eta_j(K_{i0})
=\sum_{j\in\cI} c_j \sigma_j\tilde Q_{ji}=\sum_{j\in\cI} c_j Q_{ji}\sigma_i
=2\chi_i\sigma_i,
\end{equation}
where $\tilde Q=\tilde P^{-1}$, $Q=P^{-1}$ and $2\chi_i=\sum_{j\in\cI}c_jQ_{ji}$.
If we rewrite~\eqref{eq:mu-sigma} as
\begin{equation*}
c_j\sigma_j=\sum_{i\in\cI}\mu_i\tilde P_{ij}=\mu_j+\sum_{i\in\cI}\mu_iC_{ij}\sigma_j,
\end{equation*}
then we can specialize~\eqref{eq:gen-s3s} with $N_{jr}=L_{jr}$ (up to overall
scale) using
\begin{equation*}
L_{j0}(\mu^\lamc)=\mu_j,\qquad L_{j\infty}(\mu^\lamc)=-(L_{j0}+L_{j1})(\mu^\lamc)
=\sum_{i\in\cI}\mu_iC_{ij}- c_j.
\end{equation*}
We may also compute directly that the toral part of the metric on $\Ds$ is
\begin{equation}\label{eq:S3s-metric}
\begin{split}
h_\Ll\fib:=&\sum_{i,j,k\in\cI} c_i \d\eta_i(X_j,JX_k) \,\d t_j \,\d t_k
=\sum_{i,j,k,p\in\cI} \tfrac12 c_i Q_{ip}\,\d\d^c\mm^p(X_j,JX_k) \,\d t_j \,\d t_k\\
=&\sum_{i,j,k,p,q,r\in\cI} 2c_iQ_{ip}\,
\bigl(\sigma_p{\d\ang_{p0}}^2+(1-\sigma_p){\d\ang_{p1}}^2\bigr)
(X^o_q \tilde Q_{qj},X^o_r \tilde Q_{rk})\,\d t_j\, \d t_k\\
=&\sum_{i\in\cI} 4\chi_i\sigma_i(1-\sigma_i)
\biggl(\sum_{j\in\cI}\tilde Q_{ij} \,\d t_j\biggr)^2,
\end{split}
\end{equation}
where we use $\sigma_i(1-\sigma_i)^2+(1-\sigma_i)\sigma_i^2
=\sigma_i(1-\sigma_i)$. This agrees with~\eqref{eq:gen-s3s} since
\begin{align*}
\d\mu_i&=\sum_j c_j (\d\sigma_j\, \tilde Q_{ji}+\sigma_j\,\d\tilde Q_{ji})=
\sum_{j,k,l}c_j(Q_{jk}P_{kl}-\sigma_j\tilde Q_{jk}C_{kl})\,\d\sigma_l\,\tilde Q_{li}\\
&=\sum_{j,k,l} c_j Q_{jk}(P_{kl}-\sigma_k C_{kl})\,\d\sigma_l\,\tilde Q_{li}=
\sum_k 2\chi_k\, \d\sigma_k\, \tilde Q_{ki}.
\end{align*}

\subsection{Projective cubes}\label{s:pcubes} For Levi--K\"ahler quotients of
a $\ell$-fold product of $3$-spheres, the polytope $\Pol_{\g,\lamc}$ is an
$\ell$-cuboid, i.e., it has the combinatorics of a product of $m=\ell$
intervals (an $m$-cube).  Such a polytope is \emph{projectively equivalent to
  a cube} if the intersections of pairs of opposite facets lie in a
hyperplane: transforming this hyperplane to infinity, opposite facets become
parallel and meet the hyperplane at infinity in the facets of an
$(m-1)$-simplex, and all simplices are projectively equivalent. Projective
equivalence to an $m$-cube is automatic when $m=2$, but is restrictive for
$m\geq 3$ (when $m=3$, opposite facets of a generic cuboid meet in skew lines,
not coplanar lines).

The assumption of projective equivalence to an $m$-cube simplifies the
previous analysis, because we may take $N_{i\infty}$ to be the equation of the
hyperplane common to the pencils spanned by opposite facets, independent of
$i\in\cI=\{1,\ldots m\}$. Concretely, let $b_j\in\R$ for $0\leq j\leq m$ and
$\mu_j:0\leq j\leq m$ be affine coordinates on the affine space
$\As=\{(\mu_0,\mu_1,\ldots \mu_m):\sum_{j=0}^m b_j \mu_j =1\}$.

We now set $N_{i\infty}=\mu_0$, $N_{i0}=-\mu_i$ and $N_{i1}=\mu_i-\mu_0$ for
$1\leq i\neq m$. We thus have
\begin{equation}\label{xi}
\xi_i= \frac{\mu_i}{\mu_0} \qquad\text{and hence}\qquad
b_0+\sum_{i=1}^n b_i \xi_i = \sum_{i=0}^m \frac{b_i\mu_i}{\mu_0}
= \frac1 {\mu_0},
\end{equation}
so that the inverse transformation is
\begin{equation}\label{mu}
\mu_0 = \frac{1}{b_0+b_1\xi_1 + \cdots + b_m\xi_m},  \quad 
\mu_i = \xi_i\mu_0 = \frac{\xi_i}{b_0+b_1\xi_1 + \cdots + b_m\xi_m}.
\end{equation}
Differentiating $\mu_i$, we then have for the symplectic form 
\begin{equation}\label{omega}
\omega_{\Ll} = \sum_{i=1}^m \d\mu_i \wedge \d t_i
= \mu_0 \sum_{i=1}^m \d\xi_i \wedge \theta_i,
\end{equation}
where
\begin{equation}\label{theta}
\theta_i = \d t_i - \mu_0 b_i \sum_{j=1}^m \xi_j \d t_j
= \d t_i - b_i \sum_{j=1}^m \mu_j \d t_j.
\end{equation}
In particular $\d \theta_i = - b_i \omega_\Ll$. Letting
\begin{equation}\label{J}
J\d\xi_i := A_i(\xi_i)\theta_i, \quad 1\leq i\leq m,
\end{equation}
$J$ defines an integrable almost complex structure and the $t_i$ are
pluriharmonic.  We thus have the following diagonal form of $(h_\Ll,
\omega_\Ll)$
\begin{equation}\label{metric}
\begin{split}
h_\Ll &= \frac{1}{b_0+b_1\xi_1 + \cdots + b_m\xi_m}
\sum_{i=1}^m\Bigl(\frac{{\d\xi_i}^2}{A_i(\xi_i)} + A_i(\xi_i){\theta_i}^2\Bigr)\\
\omega_{\Ll} &= \frac1{b_0+b_1\xi_1 + \cdots +b_m \xi_m} \sum_{i=1}^m
\d\xi_i \wedge \theta_i,
\end{split}
\end{equation}
where the $1$-forms $\theta_i$ are given by \eqref{theta}.

The product of intervals $\xi_i\in [\al_{i0},\al_{i1}]$, $1\leq i\leq m$
transforms to the compact convex polytope $\Pol$ determined by the hyperplanes
$(\xi_i - \al_{ir})\mu_0=0$, (for $r\in\{0,1\}$, $1\leq i\leq n$). As before,
we set
\begin{equation*}
A_i(y):= a_i\prod_{r\in\cH}(y-\al_{ir}), \qquad L_{ir}(\mu) :=
\frac{2(\xi_i - \al_{ir})\mu_0}{A_i'(\al_{ir})}
=\frac{2(\mu_i-\alpha_{ir}\mu_0)}{A_i'(\al_{ir})},
\end{equation*}
where $r\in\cH=\{0,1,\infty\}$. Note that $L_{ir}\geq 0$ on $\Pol$ for
$1\leq i\leq m$ and $r\in\{0,1\}$, and that $\sum_{r\in\cH} L_{ir}=0$. We can
compute a K\"ahler potential from the symplectic potential $G_\Ll$ by Legendre
transform based at $\mu_j=0$ to get
\begin{equation}\label{kahler-potential}
H_{\Ll} = \sum_{i,r} \frac{\log |L_{ir}|}{A_i'(\al_{ir})}
= \sum_{i,r} \frac{\log |\xi_i - \al_{ir}|}{A_i'(\al_{ir})} + \ \mathrm{const}
= \sum_{i=1}^m \int^{\xi_i} \frac{\d s}{A_i(s)}.
\end{equation}

\subsection{Levi--K\"ahler metrics of convex quadrilaterals}\label{s:quad}

We now specialize to the case $m=2$, i.e.,
\begin{equation*}
N=\Sph^3\times \Sph^3= \{z\in \C^4\cong \C^2\otimes \C^2 \st
(\sigma_{10}+\sigma_{11})(z) =1, (\sigma_{20}+\sigma_{21})(z) =1\}.
\end{equation*}
By Theorem~\ref{t:product-spheres}, the compact K\"ahler $4$-orbifolds $(M,
h_{\Ll}, \omega_{\Ll})$ obtained as a Levi--K\"ahler quotient of $\Sph^3\times
\Sph^3$ by an abelian subgroup $G\sub\T^4$ are the compact toric $4$-orbifolds
whose rational Delzant polytope is a quadrilateral.  Note that from its very
construction, $h_\Ll$ is compatible with a second complex structure $\tilde J$
on $M$, coming from the quotient of the product CR structure $(\Ds,J_1-J_2)$
on $N=\Sph^3\times \Sph^3$, where $(\Ds,J) =(\Ds_1,J_1) \dsum (\Ds_2,J_2)$ is
the direct sum of the CR distributions of each $\Sph^3$ factor. Thus
$J=J_1+J_2$, and $\tilde J=J_1-J_2$ defines a second CR structure on $N$,
associated to the same distribution $\Ds\sub TN$, which commutes with $J$ and
induces the opposite orientation on $\Ds$.  In the terminology of \cite{ACG1},
$h_\Ll$ is K\"ahler with respect to $J$ and {\it ambihermitian} with respect
to commuting complex structures $(J,\tilde J)$ on $M$.  We are going to show
that $(h_\Ll, \tilde J)$ is, in fact, conformal to another $\T$-invariant
K\"ahler metric $({\tilde h}_\Ll,\tilde J)$ (which induces the opposite
orientation of $(M,J)$), i.e., that $(h_\Ll, J)$ and $({\tilde h}_\Ll, \tilde
J)$ define an {\it ambitoric} structure on $M$ in the sense of
\cite{ACG1}. These structures have been extensively studied and classified,
both locally~\cite{ACG1} and globally~\cite{ACG2}.

As any quadrilateral is a projective cube, the general form of the
Levi--K\"ahler metric $h_{\Ll}$ is described by \eqref{metric} but we shall
also describe below how this form is derived from the choice of the subgroup
$G$.  Following the notation in \S\ref{s:S3s}, we
specialize~\eqref{eq:S3s-metric} to $\ell=2$, and set
\begin{equation*}
C=\begin{pmatrix} \al & \ga \\ \be & \de \end{pmatrix},
\end{equation*}
so that
\begin{equation*}
A=\begin{pmatrix} 1+\al\sigma_1& \ga\sigma_1 \\ \be\sigma_2& 1+\de\sigma_2 
\end{pmatrix} \quad\text{and}\quad
\tilde A=
\begin{pmatrix} 1+\al\sigma_1 & \ga\sigma_2 \\ \be\sigma_1 & 1+\de\sigma_2
\end{pmatrix};
\end{equation*}
hence
\begin{equation*}
B=\frac 1 Z
\begin{pmatrix} 1+\de\sigma_2& -\ga\sigma_1 \\ -\be\sigma_2& 1+\al\sigma_1 
\end{pmatrix} \quad\text{and}\quad
\tilde B= \frac 1 Z
\begin{pmatrix} 1+\de\sigma_2 & -\ga\sigma_2 \\ -\be\sigma_1 & 1+\al\sigma_1
\end{pmatrix},
\end{equation*}
where $Z=(1+\al\sigma_1)(1+\de\sigma_2)-\be\ga\sigma_1\sigma_2 =1 +
\al\sigma_1 + \de\sigma_2 + (\al\de-\be\ga)\sigma_1\sigma_2$. Then
\begin{equation}\label{eq:c-moment}
\mu_1 = \frac{c_1 \sigma_1(1+\de\sigma_2) - c_2 \be \sigma_1\sigma_2}Z,
\quad
\mu_2 = \frac{- c_1 \ga \sigma_1\sigma_2 + c_2 (1+\al\sigma_1)\sigma_2}Z,
\end{equation}
while the toral part of the metric on $\Ds$ is
\begin{multline}\label{eq:H}
h_\Ll\fib=
\sigma_1(1-\sigma_1)\bigl((1+\de\sigma_2)c_1 - \be\sigma_2 c_2\bigr)
\bigl((1+ \de\sigma_2)\d t_1 - \ga\sigma_2 \d t_2\bigr)^2/Z^3 \\
+ \sigma_2(1-\sigma_2)\bigl((1+ \al\sigma_1)c_2 - \ga \sigma_1 c_1\bigr)
\bigl((1+\al\sigma_1)\d t_2 - \be\sigma_1 \d t_1\bigl)^2/Z^3.
\end{multline}

We now transform this expression into the ansatz~\eqref{metric} for projective
cubes (all quadrilaterals are projectively equivalent). To do this we first
find the base loci of the families of lines with $\sigma_1$ or $\sigma_2$
constant. We find that for $\sigma_1=c_2/(c_1\ga-c_2\al)$,
$1+\al\sigma_1=c_1\ga/(c_1\ga-c_2\al)$ and hence $(\mu_1,\mu_2)=(c_2/\ga,0)$,
independent of $\sigma_2$. Similarly for $\sigma_2=c_1/(c_2\be-c_1\de)$,
$(\mu_1,\mu_2)=(0,c_1/\beta)$, independent of $\sigma_1$. These coordinates
are examples of {\it Segre factorization structures}, as defined in
\cite{ACG2}. We transform the coordinate singularity to infinity by setting
\begin{align*}
\xi_1&=\sigma_1\Delta_1,\quad\Delta_1
=\frac{1}{c_2(1+\al\sigma_1)-c_1\ga\sigma_1}
=\frac{1+(c_1\ga-c_2\al)\xi_1}{c_2},\\
\xi_2&=\sigma_2\Delta_2,\quad\Delta_2
=\frac{1}{c_1(1+\de\sigma_2)-c_2\be \sigma_2}
=\frac{1+(c_2\be-c_1\de)\xi_2}{c_1}
\end{align*}
We then compute that $1+c_1\ga\xi_1+c_2\be\xi_2=c_1 c_2 \Delta_1\Delta_2 Z$ and
\begin{gather*}
\Delta_1+\al\xi_1=(1+c_1\ga\xi_1)/c_2 ,\quad \Delta_2+\de\xi_2=
(1+c_2\be\xi_2)/c_1,\\
(1+ \de\sigma_2)\d t_1 - \ga\sigma_2 \d t_2=
\bigl((1+c_2\be \xi_2)\d t_1 - c_1\ga\xi_2 \d t_2\bigr)/(c_1\Delta_2)\\
(1+\al\sigma_1)\d t_2 - \be\sigma_1 \d t_1=
\bigl((1+c_1\ga\xi_1) \d t_2 - c_2\be\xi_1 \d t_1\bigl)/(c_2\Delta_1),
\end{gather*}
so that, setting $\mu_0=1/(1+c_1\ga\xi_1+c_2\be\xi_2)$, we have
\begin{gather}\begin{split}
h_\Ll\fib=&c_1c_2^3\mu_0\Delta_1\xi_1(\Delta_1-\xi_1)
\bigl(\d t_1-\mu_0c_1\ga (\xi_1\d t_1 +\xi_2 \d t_2)\bigr)^2 \\
&+ c_1^3c_2\mu_0\Delta_2\xi_2(\Delta_2-\xi_2)
\bigl(\d t_2 - \mu_0 c_2\be( \xi_1 \d t_1+\xi_2 \d t_2)\bigl)^2,
\end{split}\\
\label{eq:tau-moment}
\mu_1=\frac{\sigma_1}{\Delta_2 Z}=\frac{c_1c_2\xi_1}{1+c_1\ga\xi_1+c_2\be\xi_2},
\quad
\mu_2=\frac{\sigma_2}{\Delta_1 Z}=\frac{c_1c_2\xi_2}{1+c_1\ga\xi_1+c_2\be\xi_2}
\end{gather}
This has the form~\eqref{metric} with
$A_1(\xi_1)=c_1c_2^3\Delta_1\xi_1(\Delta_1-\xi_1)$ and $A_2(\xi_2)=
c_1^3c_2\Delta_2\xi_2(\Delta_2-\xi_2)$.

We now relate the Levi--K\"ahler metrics to the local forms of ambitoric
metrics studied in \cite{ACG1}. The results depend crucially on whether
$\be=0$ (when the curves $\xi_1=$constant pass through the point at infinity),
$\ga=0$ (when the curves $\xi_2=$constant pass through the point at
infinity), or both (when we have a product structure). We break this down
into three cases as follows.

\subsubsection{The product case}\label{s:product}

This is the case when $\be=0, \ga=0$. Letting $x=\xi_1, y=\xi_2$ the metric
becomes
\[
h_\Ll = \frac{\d x^2}{A(x)} + \frac{\d y^2}{B(y)}
+ A(x)\d t_1^2 + B(y)\d t_2^2
\]
with $A(x),B(y)$ positive valued polynomials of degree $2$ or $3$, i.e., a
local product of extremal toric Riemann surfaces. The construction yields (up
to an equivariant isometry corresponding to affine transformations of $x$ and
$y$) all such products for which $A(x)$ and $B(y)$ are of degree $2$ or $3$
with distinct real roots.
                     
\subsubsection{The Calabi case}\label{s:calabi}

Without loss of generality, this is the case $\be=0$, $\ga\neq 0$, so that
the curves $\xi_1=$constant pass through the point at infinity. We now let
$x=1 + c_1\gamma \xi_1, y=-c_1\gamma \xi_2$ so that the metric becomes  
\begin{equation}\label{calabi}
h_\Ll= \frac{1}{x}
\Bigl(\frac{\d y^2}{B(y)} + B(y)\d t_2^2
+ \frac{\d x^2}{A(x)} + \frac{A(x)}{x^2}\bigl(\d t_1 +y\, \d t_2\bigr)^2\Bigr)
\end{equation}
i.e., given by the Calabi construction with respect to the variables $\bar x =
1/x, y$ (see e.g. \cite{ACG,Eveline}) starting from a toric extremal Riemann
surface $(\Sigma, g_{\Sigma}= \frac{\d y^2}{B(y)} + B(y)\d t_2^2)$, and taking
an extremal toric metric on the fibre associated to the profile function
$\Theta(\bar x) : = \bar A(\bar x)/\bar x$ with $\bar A (\bar x)= \bar x^4
A(1/\bar x)$. Once again, up to affine changes of $x$ and $y$, one covers all
toric metrics of Calabi type for which the functions $A(x),B(y)$ are
polynomials of degree $2$ or $3$ with distinct real roots.

\subsubsection{The negative orthotoric case}\label{s:ambitoric}

This is the generic case when $\be\ga \neq 0$. We can therefore let $x= 1 +
c_1\gamma \xi_, y =-c_2\beta \xi_2$ so that the metric becomes
\begin{equation}\label{ansatz}
h_\Ll= \frac{1}{x-y} \Bigl(\frac{\d x^2}{A(x)} + \frac{\d y^2}{B(y)}\Bigr)
+\frac{A(x)(\d\theta_1+ y \d\theta_2)^2 + B(y)(\d\theta_1+ x \d\theta_2)^2}
{(x-y)^3},
\end{equation}
where $A(x)$ and $B(y)$ are both polynomials of degree $2$ or $3$ with
distinct real roots. (In terms of \cite{ACG1}, the conformal oppositely
oriented K\"ahler metric $(\tilde h_{\Ll},\tilde J)$ is orthotoric.)

\smallbreak

A case by case inspection shows that the conformal factor $w$ such that the
oppositely oriented K\"ahler metric is $\tilde h_{\Ll}= (1/\afn^2)h_{\Ll}$,
where $\afn=1, 1/x$ or $1/(x-y)$, in the above three cases, respectively. We
observe that $\afn$ is an affine function in the momenta with respect to
$\omega_{\Ll}$, which vanishes at the (possibly infinite) intersection points
of the pair of opposite facets of $(\Delta, \Ll)$. We summarize the discussion
as follows.

\begin{prop}\label{p:quotient} Let $(h_\Ll,\omega_\Ll)$ be a Levi--K\"ahler
quotient of $\Sph^3\times \Sph^3\sub \C^2\times \C^2$, corresponding to a
subspace $\g \sub \ab$, and some $\lamc\in \g^*$. Then $(h_\Ll, \omega_\Ll)$
is ambitoric in the sense of \cite{ACG1}, and is either a product, of
Calabi-type or conformal and oppositely oriented to an orthotoric metric,
depending on whether $\g$ intersects nontrivially two, one or zero of the
$2$-dimensional subspaces $(\ab[1],\ab[2])$, where $\ab[i]=\R^2 \sub \R^4$ is
the Lie algebra of the $2$-torus $\T_i^2\sub\T^4$ naturally acting on the
$i$-th factor \textup($i=1,2$\textup) of $\C^4=\C^2\times \C^2$. Furthermore,
in all cases, $h_\Ll$ is expressed in terms of two arbitrary polynomials of
degree $2$ or $3$, each with real distinct roots, whereas the oppositely
oriented K\"ahler metric is $\tilde h_{\Ll}=(1/\afn^2)h_{\Ll}$ for a positive
affine function $\afn$ on the quadrilateral $\Pol_{\g,\lamc}$, vanishing at
the intersection points of its opposite facets. Conversely, any ambitoric
metric of the above mentioned types determined by two polynomials of degree
$2$ or $3$ with distinct real roots\footnote{This constraint comes from the
  fact that $\g$ is a subspace of the Cartan subalgebra $\ab$ consisting of
  diagonal elements of the Lie algebra $\su(1,2)\dsum\su(1,2)$ of the CR
  automorphisms of $\Sph^3\times \Sph^3$.} arises as a Levi--K\"ahler quotient
of $\Sph^3\times \Sph^3$.
\end{prop}

\subsection{Toric bundles and Levi--K\"ahler quotients of products of
  spheres}\label{s:generalized-calabi}

The Calabi and the product cases appearing in the analysis in the previous
subsection have a natural generalization to higher dimensions in the framework
of {\it semisimple rigid toric bundle construction} of \cite{ACGT,hfkg5},
where it is also referred to as the {\it generalized Calabi construction}. Let
us first recall briefly the setting of these works.

Let $\pi\colon M\to S$ be a bundle of toric k\"ahlerian manifolds or orbifolds
of the form $M=P\times_{T}V$, for an $\ell$-torus $T$, a principal $T$-bundle
$P$ over a k\"ahlerian manifold $S$ of dimension $2d$, and a toric
$2\ell$-manifold (or orbifold) $V$ with Delzant polytope
$\Delta\sub\tor^*$; thus $M$ has dimension $2m=2(d+\ell)$. We let $F_i$
($i=1,\ldots n$) denote the co-dimension one faces of $\Delta$ and $u_i$ the
primitive inward normals (with respect to a lattice $\Lam$). Let be the
dimension of $M$. Let $\theta\in\Omega^1(M,\tor)$ the connection $1$-form
induced by a principal $T$-connection on $P$, with curvature
$\Omega\in\Omega^1(S,\tor)$.  Suppose that $\Omega^0$ is a closed $2$-form on
$S$. Then the rigid toric bundle construction on $M$ is a K\"ahler metric of
the form
\begin{align}\label{generalized-calabi}
g&= g_0 +\ip{x, g_\Omega}
+\ip{dx,(\bH^V)^{-1},dx}+\ip{\theta, \bH^V,\theta},\\ \nonumber
\omega&= \Omega_0+\ip{x,\Omega}+\ip{dx\wedge\theta},\qquad\qquad
d\theta=\Omega,
\end{align}
where:
\begin{bulletlist}
\item $x\in C^\infty(M,\tor^*)$ is the momentum map of the $T$ action
with image $\Delta$;
\item $\bH^V \in C^\infty(\Delta, S^2\mathfrak t^*)$ is a matrix valued function
which, firstly, satisfies the boundary conditions that on any co-dimension one
face $F_i$, there is a function $h_i$ with
\begin{equation*}
\sum_t \bH^V_{st}(x) (u_i)_t= 0,\qquad
\sum_t\frac{\del \bH^V_{st}}{\del x_r}(x) (u_i)_t=(u_i)_r h_i(x)_s
\end{equation*}
and $\ip{h_i(x),u_i}:=\sum_s h_i(x)_s (u_i)_s=2$ for all $x\in F_i$;
secondly the inverse $(\bH^V)^{-1}\in C^\infty(\Delta, S^2\mathfrak t)$ of $\bH^V$
is the hessian of a function $G_V$ on $\Delta$; thirdly, $\bH^V$ induces a
positive definite metric on the interior of each face $F$ of $\Delta$ (as an
element of $S^2(\mathfrak t/\mathfrak t_F)^*$, where $\mathfrak t_F$ is the
isotropy algebra of $F$);
\item the metric $g_0+\ip{x,g_\Omega}$ associated to $\Omega^0+\ip{x,\Omega}$
  via the complex structure on $S$ is positive definite for all
  $x\in\Delta\sub\tor^*$.
\end{bulletlist}
Throughout, angle brackets denote natural contractions of $\mathfrak t$ with
$\mathfrak t^*$, and we omit pullbacks by $x$ and $\pi$. In particular $x$
itself will denote the standard $\mathfrak t^*$-valued coordinate on $\Delta$,
as well as its pullback to $M$.

The function $G_V$ is called a {\it symplectic potential} for $\bH^V$ and is
determined up to an affine function on $\tor^*$.  According
to~\cite[Thm.~2]{Abreu1} and \cite[Rem.~4.2]{ACGT}, the boundary and
positivity conditions above can be equivalently formulated in terms of $G_V$,
by requiring that $G_V$ is smooth and strictly convex on the interior
$\mathring{\Delta}$ of $\Delta$, such that
\begin{equation}\label{abreu}
\begin{split}
  & G_V -\frac{1}{2} \sum_{i=1}^kL_i \log L_i  \in C^{\infty}(\Delta) \\
  & \det (\Hess G_V)\prod_{i=1}^{k}L_k \in C^{\infty}(\Delta)
  \quad\text{and is strictly positive},
 \end{split}
 \end{equation}
where $L_i=\ip{u_i,x}-v_i, i=1, \ldots k$ are the labels defining $\Delta$.

Let $X$ be a holomorphic vector field os $S$ which is hamiltonian with respect
to $\Omega_0+\ip{x,\Omega}$ for all $x\in\Delta$. Thus $-\imath_X
(\Omega_0+\ip{x,\Omega})=\d f_0+\ip{x,\d f_\Omega}$ for functions $f_0\in
C^\infty(S,\R)$ and $f_\Omega\in C^\infty(S,\tor)$. Generalizing an
observation from the proof of \cite[Lemma~5]{hfkg5}, any such $X$ can be
lifted to a hamiltonian Killing vector field of $(M, g, \omega)$:
\begin{equation}\label{killing-lift}
\hat X= X^H + \ip{f_\Omega,K},
\end{equation}
where $K:= \grad_{\omega} x \in C^\infty(M, TM) \otimes \tor^*$ is
the family of hamiltonian vector fields generated by the principal $T$ bundle
$P$, and $X^H$ denotes the horizontal lift of $X$ (to the kernel of $\theta$).
Indeed, $-\imath_{\hat X} \omega = -\imath_X (\Omega_0+\ip{x,\Omega}) +
\ip{f_\Omega,\d x}= \d(f_0+\ip{x,f_\Omega})$, so $\hat X$ has hamiltonian
$f_0+\ip{x,f_\Omega}$ (omitting pullbacks of $f_0$ and $f_\Omega$ to $M$).

Suppose now that the family $g_0+\ip{x,g_\Omega}$ of K\"ahler metrics on $S$
is toric with respect to a fixed torus action of a torus $T_S$ with Lie
algebra $\tor_S$. For each fixed $x$, the momentum map of $T_S$ may be written
$\xi_0+\ip{x,\xi_\Omega}$, where $\xi\in C^\infty(S,\tor_S^*)$ and $\xi_\Omega
\in C^\infty(S,\tor_S^*\otimes \tor)$, and pulling back these functions to
$M$, $\xi_0+\ip{x,\xi_\Omega}$ is the momentum map for the $T_S$ action on $M$
defined by lifting the generators to $M$ using~\eqref{killing-lift}. Since
these lifts commute with $K$, $M$ is toric under the combined action of
$T\times T_S$.

\begin{lemma}\label{l:symplectic-fibering} Let $(M,g, \omega)$ be a K\"ahler
manifold or orbifold given by \eqref{generalized-calabi} with fibrewise
symplectic potential $G_V$, and suppose that $(S, \Omega_0+\ip{x,\Omega})$
is toric with respect to a fixed action of $T_S$ for all $x\in\Delta$.
Then, $(M,g,\omega)$ is toric and with respect lifted $T\times T_S$ action
and has a symplectic potential
\begin{equation}\label{eq:sf}
G_M = G_V + G_0 + \ip{x,G_\Omega},
\end{equation}
where for eached fixed $x$, $G_0 + \ip{x,G_\Omega}$ is a symplectic potential
for $g_0+\ip{x,g_\Omega}$ on $S$.
\end{lemma}
\begin{proof} The torus action $T$ on $(M,g,\omega)$ is \emph{rigid}, meaning
that the metric on the torus orbits depends only on the value of the momentum
map $x$. Hence for each fixed $\xi\in\tor_S$, the restriction of $G_M$ to
$\tor^*+\xi\cong\tor^*$ differs from $G_V$ by an affine function of $x$, hence
has the form~\eqref{eq:sf} for functions $G_0$ and $G_\Omega$ of
$\xi\in\tor_S$.  By construction, for each fixed $x\in\Delta$,
$(S,g_0+\ip{x,g_\Omega})$ is the K\"ahler quotient of $M$ by $T$ at momentum
level $x$. Hence by~\cite{CDG}, the restriction of $G_M$ to
$x+\tor_S^*\cong\tor_S^*$ is a symplectic potential for $g_0+\ip{x,g_\Omega}$,
hence so is $G_0 + \ip{x,G_\Omega}$.
\end{proof}

A rigid toric bundle metric~\eqref{generalized-calabi} is \emph{semisimple} if
$S$ is a product of K\"ahler manifolds $(S_j,\omega_j)$ and there exist
$c_j\in\R$ and $p_j\in\mathfrak t$ (for $j\in\{1,2,\ldots N\}$) such that
$\Omega_0+\ip{x,\Omega}=\sum_{j=1}^N (c_j + \ip{p_j,x})\omega_j$. If each
$(S_j,\omega_j)$ is toric under a torus $T_j$ with Lie algebra $\tor_j$, then
$(S, \Omega_0+\ip{x,\Omega})$ is toric for all $x$ under the product action of
$T_S=\prod_{j=1}^N T_j$ with $\tor_S=\bigoplus_{j=1}^N \tor_j$.  We refer to
this special case as the \emph{toric generalized Calabi ansatz}.

\begin{prop} \label{p:fibration} Suppose that $(g,\omega)$ is a K\"ahler
metric obtained by the toric generalized Calabi ansatz, where the
fibre $V$ is a compact toric orbifold with labelled polytope $(\Delta,
L_1, \ldots, L_k)$, and each factor $(S_j, \omega_j)$ of the base $S=
\prod_{j=1}^N S_j$ is a compact toric orbifold with labelled Delzant polytope
$(\Delta_j, L^j_1, \ldots, L^j_{k_j})$, respectively.  Then $(g, \omega)$ is a
toric K\"ahler metric on the compact symplectic orbifold with
labelled Delzant polytope
\[
\hat \Delta
=\{ L_i (x) \ge 0, \ \hat L^j_{r_j}:= (\langle p_j, x) + c_j) L^j_{r_j} \ge 0 \}.
\]
Moreover, if the fibrewise toric metric determined by $G^V$ and the K\"ahler
metrics $\omega_j$ are all Levi--K\"ahler quotients of product spheres, then
the resulting metric \eqref{generalized-calabi} is a Levi--K\"ahler quotient
of the overall product of spheres.
\end{prop}
\begin{proof} We check that $G_M$ satisfies the conditions \eqref{abreu}.
Lemma~\ref{l:symplectic-fibering} and the fact that $\langle p_j, x \rangle +
c_j$ is strictly positive on $\Delta$ imply that $G_M$ differs by a smooth
function
\[
G_M^0:= \frac{1}{2}\Big(\sum_{i=1}^{k} L_i \log L_i
+ \sum_{j=1}^N\Big(\sum_{r_j=1}^{k_j} \hat L^j_{r_j} \log \hat L^j_{r_j}\Big) \Big).
\]
It remains to see that $\det(\Hess(G_M)) \bigl(\prod_{i=1}^{k}
L_i\bigr)\bigl(\prod_{j, r_j} \hat L^j_{r_j}\bigr)$ is smooth and positive on
$\smash{\hat \Delta}$.  The determinant $\det(\Hess(G_M)^{-1})$ is, up to a
positive scale, the norm with respect to $g$ of the wedge product of the
Killing vector fields $(K_1, \ldots K_{\ell}, \hat K^j_{r_j})$ for $j=1, \ldots N$,
$r_j=1, \ldots d_j$.  Using \eqref{killing-lift} and the specific form
\eqref{generalized-calabi} of the metric $g$, one sees that
\[
\det(\Hess(G_M)^{-1})
= C \Big(\prod_{j=1}^{N} (\langle p_j, x \rangle + c_j)^{d_j}\Big)
\det (\Hess(G_V)^{-1}) \det(\Hess (G_j)^{-1}),
\]
where $C>0$ is constant. Using the compactification criteria~\eqref{abreu} for
each $G_V$ and $G_j$, and the fact that $\hat \Delta$ is a simple polytope,
near any point $y$ on a face $\hat F \sub \hat \Delta$ we have:
\begin{equation}\label{det}
\begin{split}
\det(\Hess(G_M)^{-1})
&= \delta\Bigl(\prod_{i=1}^k L_k\Bigr)\Bigl(\prod_{j=1}^N (\ip{p_j,x} + c_j)^{d_j}
\Bigl(\prod_{r_j=1}^{k_j} L^j_{r_j}\Bigr)\Bigr) \\
&= \delta'\Bigl(\prod_{i=1}^k L_k\Bigr)
\prod_{j=1}^N \Bigl(\prod_{r_j=1}^{k_j}\hat L^j_{r_j}\Bigr),
\end{split}
\end{equation}                                                          
where $\delta$ and $\delta'$ are smooth positive functions around $y \in
\smash{\hat \Delta}$.

For the second part, we assume by Theorem~\ref{t:symplectic-potential} that
$G_V= \frac{1}{2}\sum_{ir} L_{ir} \log L_{ir}$ and $G_j = \frac{1}{2}
\sum_{q_jr_j} L^j_{q_jr_j} \log L^j_{q_jr_j}$ with $\sum_r L_{ir}=0$ and
$\sum_{r_j} L^j_{q_jr_j}=0$. Then, by Lemma~\ref{l:symplectic-fibering},
\begin{equation*}
\begin{split}
G_M &= \frac{1}{2}\Big( \sum_{ir} L_{ir} \log L_{ir} + \sum_{j=1}^N (\langle p_j, x \rangle + c_j)\Big( \sum_{q_jr_j} L^j_{q_jr_j} \log L_{q_jr_j} \Big)\Big) \\
&= \frac{1}{2} \Big( \sum_{ir} L_{ir} \log L_{ir} + \sum_{j=1}^N \Big( \sum_{q_jr_j} \hat L^j_{q_jr_j} \log L_{q_jr_j} \Big)\Big)  \\
&= \frac{1}{2} \Big( \sum_{ir} L_{ir} \log L_{ir} + \sum_{j=1}^N \Big( \sum_{q_jr_j} \Big(\hat L^j_{q_jr_j} \log \hat L_{q_jr_j} - \sum_{q_jr_j} L_{q_jp_j}\log(\langle p_j, x \rangle + c_j)\Big) \Big)\Big)   \\
&= \frac{1}{2} \Big( \sum_{ir} L_{ir} \log L_{ir} + \sum_{j=1}^N \Big( \sum_{q_jr_j} \hat L^j_{q_jr_j} \log \hat L_{q_jr_j} \Big)\Big)
\end{split}
\end{equation*}
We conclude by using Theorem~\ref{t:symplectic-potential} again.
\end{proof}

\begin{rem}\label{r:local-explanation} The toric Calabi construction provides
a practical method for constructing new toric metrics from old ones. Suppose
$(g_j,\omega_j)$ and $(g^V, \omega_V)$ are toric K\"ahler metrics on
$\mathring{\Delta}_j \times \T^{d_j}$ and $\mathring{\Delta} \times \T^{\ell}$
respectively, for labelled (simple convex, compact) polytopes $\Delta_j =
\{x^j \in \R^{d_j} : L_{r_j}^j(x^j) \ge 0, \ r_j=1, \ldots k_j \}$
($j=1,\ldots N$) and $\Delta=\{x\in\R^\ell : L_i(x)\ge 0, \ i=1,\ldots N \}$.
Let $p_j = (p_{j1}, \ldots , p_{j\ell})$, $\theta = (\theta_1, \ldots,
\theta_{\ell})$ with
\[
\theta_i= dt_i +
\sum_{j=1}^{N}p_{ji}\Big(\sum_{r_j=1}^{d_j}x^j_{r_j}dt^j_{r_j}\Big),
\]
where the affine functions $\sum_{i=1}^{\ell} p_{ji} x_i + c_j$ are positive
on $\Delta$. Then, as in Propostion~\ref{p:fibration}, we get a toric K\"ahler
metric on the interior of $\hat \Delta$ times $\T^{\ell} \times \T^{d_1}
\times \cdots \times \T^{d_N}$, whose symplectic potential $G_M$ is given by
Lemma~\ref{l:symplectic-fibering}. Taking all the affine functions $L^j_{r_j}$
and $L_i$, and $(\sum_{i=1}^{\ell} p_{ji} x_i + c_j)$ be all with rational
coefficients, one gets a rational labelled polytope $\hat \Delta$, and if the
ingredient metrics $(g_j,\omega_j)$ and $(g^V, \omega_V)$ satisfty the Abreu
boundary conditions \eqref{abreu} for the corresponding labellings $L^j_{r_j}$
and $L_i$, then so does the metric given (locally) by the toric generalized
Calabi Ansatz, with the labelling defined in Proposition~\ref{p:fibration}.
Hence the metric compactifies on the compact toric orbifold $M$ given by
$\hat\Delta$ and these labels.
\end{rem}

\section{Curvature of Levi--K\"ahler quotients of products of spheres}
\label{s:c-lkq}
\subsection{Bochner condition}

Recall (see e.g.~\cite{ACG1,Bryant}) that if $R\in \Wedge^{1,1}\C^{m*}\otimes
\un(m)$ is a \emph{formal K\"ahler curvature tensor}, i.e.,
$R_{u,v}(w)+R_{v,w}(u)+R_{w,u}(v)=0$, then the \emph{Bochner part} of $R$ is
its orthogonal projection $\cB(R)$ onto the $U(m)$-submodule of formal
K\"ahler curvature tensors with vanishing Ricci trace. The \emph{Bochner
  tensor} $B^g$ of a K\"ahler manifold is then the (pointwise) Bochner part
$\cB(R^g)$ of its Riemannian curvature $R^g\in\Gamma(\Wedge^{1,1}T^*M\otimes
\un(TM))$. One can extend this definition to more general Hermitian curvature
tensors $R\in \Wedge^{2}\C^{m*}\otimes\un(m)$ (which do not a priori satisfy
the Bianchi identity) where one still denotes by $\mathcal{B}(R)$ the
orthogonal projection of an element $R$ onto the $U(m)$-submodule of formal
K\"ahler curvature tensors with vanishing Ricci trace.

In our language, S. Webster~\cite{webster} showed that in codimension one the
Bochner tensor of a Levi--K\"ahler quotient $M$ of CR manifold $N$ pulls back
to the Chern--Moser tensor of $N$, which vanishes when the $N$ is locally CR
diffeomorphic to a standard CR sphere $(\Sph^{2m+1}, \Ds, J)$. In particular,
every Levi--K\"ahler quotient $(M^{2m},g, J)$ of $\Sph^{2m+1}$ is
Bochner-flat. We generalize this to arbitrary codimension, by describing the
curvature of K\"ahler metrics arising as Levi--K\"ahler quotients of a product
of CR-spheres.

Let $N= \Sph^{2m_1+1} \times \cdots \times \Sph^{2m_{\ell}+1} \sub \C_\cS$
be a product of standard CR spheres and $\bigl(\Ds=\bigoplus_{i\in\cI}\Ds_i,
J=\bigoplus_{i\in\cI} J_i\bigr)$ be the product CR structure and denote by
$N_i\cong\Sph^{2m_i+1}$ the $i$-th factor of $N$, with projection $p_i\colon
N\to N_i$. The bundle $p_i^* TN_i$ is identified with the subbundle
$E^i:=\bigcap_{j\neq i} \ker (p_{j*})$ of $TN$ via the restriction
$p_{i*}\colon E_z^i \stackrel{\simeq}{\longrightarrow} T_{p_i(z)}N_i$. We
denote the projection $r_i\colon TN \to E^i$.

Let $(\g,\lamc)$ be a positive Levi pair (corresponding to $\Ll$) and assume
that $\g$ is the Lie algebra of a subtorus $G$ of $\Ab$; denote by $M = N/G$
the Levi--K\"ahler quotient and by $\pi\colon N \to M$ the quotient map.  The
global assumption on $\g$ is made purely to make statements about $M$ rather
than local quotients.  In addition to a K\"ahler structure $(\check{g}=h_\Ll,
J, \omega_\Ll)$, $M$ inherits of a $\check{g}$-orthogonal splitting of its
tangent space, namely
\begin{equation}\label{orthogonal}
TM = \bigoplus_{i\in\cI} \Dsm_i
\end{equation}
where $\Dsm_i=\pi_*(\Ds_i)$. We denote by $\check{\nabla}^i$ the connection on
$\Dsm_i$ induced by the Levi--Civita connection $\check{\nabla}$ of
$\check{g}$, by $R^{\check{\nabla}^i}$ the corresponding curvature tensor (a
section of $\Wedge^2T^*M\otimes\mathfrak{gl}(\Dsm_i,J)$, where
$\mathfrak{gl}(\Dsm_i,J)$ denotes the bundle of $J$-commuting endomorphisms of
$\Dsm_i$), and by $B^i:=\mathcal{B}^i(R^{\check{\nabla}^i})$ the Bochner
projection of
$R^{\check{\nabla}^i}|_{\Dsm_i}\in\Wedge^2\Dsm_i^*\otimes\mathfrak{gl}(\Ds_i,J)$.

\begin{prop}\label{BochnerFlatOnFactor} For each $i\in\cI$, $B^i=0$.
\end{prop}
We prove this result using the observation (see~\cite{david:weyltanaka}
and~\cite{webster}) that the Chern--Moser tensor of $(N_i, J_i, \Ds_i)$ may
be computed from the horizontal part of the curvature of the {\it Tanaka
  connection} (see e.g.~\cite{Blair}) associated to any contact form $\alpha$
compatible with the (codimension $1$) CR structure $(\Ds,J)$.  The
Chern--Moser tensor does not depend upon the chosen compatible contact
structure (it is a CR invariant). If the Reeb vector field $\Rv$ of $\alpha$
(determined by $\alpha(\Rv)=1$ and $\cL_\Rv\alpha=0$) is a transverse CR
vector field, the horizontal part of the Tanaka connection (i.e., its
restriction on $\Ds$) is the pullback to $N$ of the Levi-Civita connection of
the Levi--K\"ahler quotient $(M,g,J,\omega)$ of $N$ by $\Rv$ using the
identification $\Ds\cong \pi^*TM$.

In our situation, for each $w\in\prod_{j\neq i} N_j$, there is an embedding
$\iota_w\colon N_i\into N$. We shall (slightly abusively) still denote by
$\Ds_i$ the pull-back bundle $\iota_{w}^* \Ds_i$ and by $J_i$ the
corresponding (standard) CR-structure on $(N_i, \Ds_i)$. Recall that $N$ is
endowed with a $1$-form $\eta^{\lamc}$ with $d\eta^{\lamc} = \pi^{*}
\omega_\Ll$, where $\lamc \in \g^*$ is the value defining the Levi--K\"ahler
quoting $\omega_\Ll$. The pull-back $\alpha_{w,i}:= \iota_w^* \eta^{\lamc}$
thus defines a $1$-form on $N_i$, and the next Lemma implies that
$\alpha_{w,i}$ is a contact $1$-form compatible with the CR structure $(N_i,
J_i, \Ds_i)$.

\begin{lemma} \label{lemmaPseudoReeb}
For any $i\in \{1,\ldots \ell\}$ and $z\in N$, the subspace of $E_z^i$ defined
by
\[
\Lb{\g,i}_z :=  r_i ( K_{\eta(E^i_z)})
\]
has dimension $1$ and is transverse to $\Ds_i$. In particular, $\Lb{\g,i}\to
N$ is a real line bundle and there exists a unique vector field $\Rv_i \in
\Gamma(\Lb{\g,i})$ such that $\eta^{\lamc}(\Rv_i)=1$. Furthermore, for each
$w\in\prod_{j\neq i} N_j$, $\alpha_{w,i}$ is a contact form on $N_i$, which
induces $\iota^*_{w,i} \circ \pi^{*}\omega_{\Ll}$ as a transversal symplectic
structure and the vector field $\Rv_{w,i}$ on $(N_i,\Ds_i)$ defined
by $\iota_{w*}\Rv_{w,i}=\Rv_i$ is a Reeb vector field for $\alpha_{w,i}$.
\end{lemma}
\begin{proof} Since $\eta(E^i_z)= \eta\big((E^i/\Ds_i)_z\big)$, $K_{\eta(E^i_z)}
\sub \Rb\g$ is at most $1$-dimensional. Since $\sum_{i=1}^{\ell}
K_{\eta(E^i_z)}= K_{\eta(T_zN)} =\Rb\g_z$ is $\ell$-dimensional, it follows
that $K_{\eta(E^i_z)}$ is $1$-dimensional and $\dim(\Lb{\g,i}_z) \leq 1.$
As $\Rb\g$ is transversal to $\Ds$ ($(\g, \lamc)$ being Levi pair), each $X
\in E^{i}_z$ decomposes as $X=K_{\eta(X)} + X^{\Ds}$; applying $r_i$ (and
using $r_i(X)=X$, $r_i (\Ds)= \Ds_i$), we obtain $X = r_i(K_{\eta(X)}) +
X^{\Ds_i}$, showing $E^{i} = \Lb{\g,i} \oplus \Ds_i$.

For any $Y, Z \in \Ds_i$, $d\alpha_{w,i}(Y,Z)= - \alpha_{w,i}([Y,Z]) =
-\eta^{\lamc}([Y, Z])= \pi^* \omega_\Ll(Y,Z)$. In order to show that
$\Rv_{w,i}$ is Reeb field for $\alpha_{w,i}$ we need to check that for any $Z
\in \Ds_i$, $d\alpha_{w,i} (\Rv_{w,i}, Z)=d\eta^{\lamc} (\Rv_i, Z)=0$ (where
we have used that $E^i$ is integrable and $\iota_{w}^* (E_i) \cong TN_i$).
Writing $\Rv_i = r_i(K_{\eta(X)})= X - X^{\Ds_i}$ for $X\in E^i$, and
decomposing $X = K_{\eta(X)} + X^{\Ds}$, we have
\begin{equation*}
d\eta^{\lamc}(\Rv_i, Z) = d\eta^{\lamc}(X- X^{\Ds_i}, Z)
=d\eta^{\lamc}(K_{\eta(X)}, Z) + \sum_{j\neq i} d\eta^{\lamc}(X^{\Ds_j}, Z)=0,
\end{equation*}
where (in the last equality) the first term vanishes because $\eta^{\lamc}$
is $\g$-invariant, whereas the second term vanishes because for $j\neq i$,
$\Ds_i$ is $d\eta^{\lamc}$-orthogonal to $\Ds_j$, see \eqref{orthogonal}.
\end{proof}

We define a partial connection $\nabla^i\colon \Gamma(E^i) \times \Gamma(E^i)
\to \Gamma(E^i)$ on the involutive subbundle $E^i \sub TN$, which pulls
back by $\iota_{w}$ to the Tanaka connection of $\alpha_{w,i}$, as follows:
\begin{itemize}
\item $\nabla^i$ preserves $\Ds_i$;
\item $\Rv_i, J_i$ and $\omega$ are $\nabla^i$ parallel;
\item For any $X, Y \in \Gamma(\Ds_i)$, the torsion $T^{\nabla^i}$ satisfies
  $T^{\nabla^i}(X,Y) =\omega(X,Y) \Rv_i$ and $T^{\nabla^i}(\Rv_i, J_i X)=-J_i
  T^{\nabla^i}(\Rv_i, X).$
\end{itemize}
In particular, $\nabla^i$ satisfies 
\begin{equation*}
\begin{split}
g(\nabla^i_{X} Y, Z) &=\check{g}(\check{\nabla}_{\check{X}} \check{Y}, \check{Z})
= \check{g}(\check{\nabla}^{i}_{\check{X}} \check{Y}, \check{Z}) 
\end{split}
\end{equation*}
We next show that $\nabla^i$ can be extended to a full connection $\nabla$ on
$TN$ preserving $\Ds$, and such that:
\begin{itemize}
 \item $\nabla\restr{\Ds} =\pi^*\check{\nabla}$;
 \item $\iota_w^*(r_i\circ\nabla) = \nabla^i$ is the Tanaka connection on
   $(N_i, \alpha_i, \Ds_i, J_i)$.
\end{itemize}
The first condition tells us that the torsion of $X,Y\in \Gamma(\Ds)$ is the
vertical part of $-[X,Y]$, that is
\begin{equation}\label{torsionN}
  T^{\nabla}(X,Y) = -K_{\eta([X,Y])}.
\end{equation}
Hence for $X,Y \in \Gamma(\Ds_i)$, $r_i(T^{\nabla}(X,Y)) =
-r_i(K_{\eta([X,Y])})= \omega(X,Y)\Rv_i= T^{\nabla^i}(X,Y)$.

Let $\Lb{\g}:=\bigoplus_{i=1}^\ell\Lb{\g,i}$ be the rank $\ell$ subbundle of
$TN$ over $N$ which is everywhere transverse to $\Ds$. We extend the
endomorphism $J$ of $\Ds$ by zero on $\Lb{\g}$, and define a linear connection
$\nabla$ on
\[
TN =\bigoplus_{i\in\cI} E^i =\bigoplus_{i\in\cI}(\Ds_i\oplus\Lb{\g,i})
=\Ds\oplus\Lb\g
\]
such that
\begin{itemize}
 \item[(i)] $\nabla$ agrees with the pullback connection $\pi^*\nabla^g$ on
   $\Ds\cong\pi^* TM$,
 \item[(ii)] $\nabla\Rv_i =0$,
 \item[(iii)] $\nabla_sX = [s,X] +\nabla_Xs-\frac{1}{2} J(\cL_sJ)X$
for each section $s\in\Gamma(\Lb{\g})$ and $X\in\Gamma(\Ds)$.
\end{itemize} 
\begin{lemma} Let $\nabla$ be a connection on $N$ satisfying the conditions
{\rm (i)--(iii)} above. Then, $\iota_w^*(r_i\circ\nabla)$ is the Tanaka
connection $\nabla^i$ on $(N_i, \alpha_i, \Ds_i, J_i)$.
\end{lemma}
\begin{proof} We first show that $\nabla_s JX = J\nabla_sX$ for
any $X\in \Gamma(TN)$ and $s\in\Gamma(\Lb{\g})$. It is clear that $\nabla_s Jt
= J\nabla_st=0$ for $s, t\in\Gamma(\Lb{\g})$ and when $X\in \Gamma(\Ds)$, we
have $J\nabla_Xs=0$ as well as the decomposition
\[
[s,X]= [s,X]^{\Ds} -\nabla_X s
\]
with respect to the splitting $TN=\Ds \oplus\Lb\g$. Using this two facts and
condition (iii), a straightforward calculation shows $\nabla_s JX =
J\nabla_sX$. Together with this last identity, conditions (i)--(ii) ensure
that $\iota_w^*(r_i\circ \nabla)$ satisfies the first two properties of the
Tanaka connection of $(N_i, \Ds_i, J_i, \alpha_i)$.  It remains to check the
torsion property. From equation~\eqref{torsionN} it follows that
$\iota_w^*(r_i\circ \nabla)$ has the same torsion on $\Ds_i$. Moreover, using
again condition (iii) we get
\[
T^{\nabla}(s,JX)=-\tfrac{1}{2} J(\cL_sJ)JX
= \tfrac{1}{2} J^2(\cL_sJ)X=-JT^{\nabla}(s,X)
\]
for any $X\in \Gamma(TN)$ and $s\in\Gamma(\Lb{\g})$.
\end{proof}

\begin{proof}[Proof of Proposition~\textup{\ref{BochnerFlatOnFactor}}]
To compare the curvatures $\nabla^i$ and $\check{\nabla}^{i}$, we notice that
\begin{equation*}
\begin{split}
g(\nabla^i_{[X,Y]}Z, T) &= g(\nabla^i_{[X,Y]^{\Ds_i}}Z,T)
- \omega_{\Ll}(\check{X},\check{Y}) g(\nabla^i_{\Rv_i}Z, T)\\
&= g(\nabla_{[X,Y]^{\Ds}}Z, T)
- \omega_{\Ll}(\check{X}, \check{Y}) g(\nabla_{\Rv_i^{\mathcal{\Rb}}}Z, T)\\
&= \check{g}(\check{\nabla}^i_{[\check{X},\check{Y}]} \check{Z}, \check{T})
+ \omega_{\Ll}(\check{X}, \check{Y}) \check{A}^i(\check{Z}, \check{T}),
\end{split}
\end{equation*}
where to go from first line to the third we have decomposed $\omega(X,Y)\Rv_i=
\omega(X,Y)\Rv_i^{{\Rb\g}} - \sum_{j \neq i} [X,Y]^{\Ds_j}$ (with
$\Rv_i^{\Rb\g}$ being the projection of $\Rv_i$ to $\Rb\g$), and we view
$\check{A}^i(\check{Z}, \check{T}) = -g(\nabla_{\Rv_i^{\Rb\g}}Z, T)$ as a
$(0,2)$-tensor on $M$, since pullbacks to $N$ of smooth functions on $M$ are
$\Rb\g$-invariant.  It then follows that
\begin{equation*}
g(R^{\nabla^i}_{X,Y}Z,T)
= \check{g}(R^{\check{\nabla}^i}_{\check{X}, \check{Y}} \check{Z}, \check{T})
+ \omega_{\Ll}(\check{X}, \check{Y}) \check{A}^i(\check{Z}, \check{T}).
\end{equation*}                                             
Applying the projection $\mathcal{B}^{i}$ to the both sides, and using that
$\mathcal{B}^{i}(R^{\nabla^i})$ equals the Chern--Moser tensor of $N_i$ (which
is zero) as well as $\mathcal{B}^i(\omega_{\Ll} \otimes A^{i})=0$ (as
${\mathcal B}^{i}$ projects onto the space of $\omega_{\Ll}$-primitive
K\"ahler curvature tensors), we obtain $B^i=
\mathcal{B}^{i}(R^{\check{\nabla}^i})=0$.
\end{proof}

\subsection{Curvature for Levi--K\"ahler quotients of products of $3$-spheres}

We next use the explicit form~\eqref{eq:gen-s3s} of the K\"ahler metric
$h_{\Ll}$ in order to compute the its scalar curvature. Up to a factor of
$-1/2$, a Ricci potential is given by the log ratio of the symplectic volume
form to a holomorphic volume form. Using $\d\bt^{(1,0)}$ to compute the
latter, we readily obtain
\begin{equation*}
\sum_{i\in\cI} \log A_i(\xi_i) + \sum_{i\in\cI} \log N_{i\infty}(\mu^\lamc)
- 2\log \bigwedge_{i\in\cI} (\xi_i \d N_{i\infty} + \d N_{i0}).
\end{equation*}
The derivatives of the first two terms are straightforward to compute, using
that $\d (N_{j\infty}(\mu^\lamc))=-\sum_{i\in\cI} N_{i\infty}(\mu^\lamc)
\ip{\Qt_i,\d N_{j\infty}} \d\xi_i$. For the third, observe by Cramer's rule
that its exterior derivative is $\sum_{i\in\cI} 2a_{ii}\, \d\xi_i$ where the
coefficients $a_{ij}$ solve the linear system $\sum_{j\in\cI} (\xi_j \d
N_{j\infty} + \d N_{j0})a_{ij} = \d N_{i\infty}$, i.e., $a_{ij} =\ip{\Qt_j,\d
  N_{i\infty}}$. Thus
\begin{align} \label{ricci} 
\rho_\Ll=\frac12\bigl(\d\sum_{i\in\cI}
\alpha_i\theta_i\bigr),
\end{align}
 where
\begin{equation*}
\alpha_i= A_i'(\xi_i)
-\sum_{j\in\cI}\frac{\ip{\Qt_i,\d N_{j\infty}} N_{i\infty}(\mu^\lamc)}
{N_{j\infty}(\mu^\lamc)}A_i(\xi_i) - 2\ip{\Qt_i,\d N_{i\infty}}A_i(\xi_i)
\end{equation*}
and $\theta_i=\ip{\Qt_i,\d\bt}$. We are interested in the scalar curvature
only, defined by
\begin{align*}
s_{\Ll} = \frac{2m\rho_{\Ll}\wedge \omega^{m-1}}{\omega^m},\quad
\frac{\omega^{m-1}}{\omega^m} = - \frac 1 m
\sum_{j\in\cI}\frac{\prod_{k\neq j} \d\xi_k\wedge \theta_k}{N_{j\infty}(\mu^\lamc)\,
\d\xi_1\wedge\theta_1\wedge\cdots\wedge\d\xi_m\wedge\theta_m}.
\end{align*}
Straightforward computation using \eqref{ricci} then yields
\begin{equation}
\begin{split}
s_{\Ll} &=-\sum_{i\in\cI} \frac{A_i''(\xi_i)}{N_{i\infty}(\mu^\lamc)}
+\sum_{i\in\cI} 2A_i'(\xi_i)\biggl(\Qt_{ii} + \sum_{j\in\cI}\Qt_{ij}\biggr)\\
&\quad-\sum_{i\in\cI} A_i(\xi_i)N_{i\infty}(\mu^\lamc)
\biggl(2\frac{\ip{\Qt_i,\d N_{j\infty}}^2}
{N_{i\infty}(\mu^\lamc)^2} +\sum_{j\in\cI} (4\Qt_{ii}\Qt_{ij}-\Qt_{ij}^2)
+\sum_{j,k} \Qt_{ij}\Qt_{ik}\biggr).
\end{split}
\end{equation}
where $\Qt_{ij} = \ip{\Qt_i,\d N_{j\infty}} / N_{j\infty}(\mu^\lamc)$.

We now specialize to the case of projective cubes as in
Section~\ref{s:pcubes}, where $N_{i\infty}(\mu^\lamc)=\mu_0=1/(b_0+b_1\xi_1+\cdots
b_m\xi_m)$, independent of $1\leq i\leq m$. Using
\[
\d\mu_0 = \ip{\d N_{j\infty},\d\mu}= -\sum_{i=1}^m \mu_0 \d\xi_i 
\ip{\Qt_i,\d N_{j\infty}},
\]
we obtain immediately that $\Qt_{ij} = b_i$. The Ricci potential specializes
to give
\begin{equation}\label{rp}
\mu_0 ^{m+2} \prod_{i=1}^{m} A_i(\xi_i),
\end{equation}
as may be verified directly using~\eqref{theta} and~\eqref{metric}, while the
scalar curvature reduces to
\begin{equation}\label{s}
s_{\Ll} =-\sum_{i\in\cI} \frac{A_i''(\xi_i)}{\mu_0} +\sum_{i\in\cI} 2(m+1)
b_i A_i'(\xi_i)-\sum_{i\in\cI} (m+1)(m+2)\mu_0 b_i^2 A_i(\xi_i).
\end{equation}

\subsection{Projective cubes and $(\afn,p)$-extremality}\label{a:EM}

A \emph{labelled cuboid} $(\Delta,\Ll)$ is a labelled Delzant polytope which
has the combinatorics of an $m$-cube. By Theorem~\ref{t:product-spheres}, any
labelled cuboid $(\Delta,\Ll)$ admits a compatible toric metric $h_\Ll$, which
is Levi--K\"ahler quotient metric of an $m$-fold product of $3$-spheres. We
consider here the case that the cuboid $\Delta$ is a projective cube, i.e.,
the intersections of pairs of opposite facets lie in a hyperplane.  For any
labelled projective cube $(\Delta,\Ll)$, $h_\Ll$ is given by \eqref{metric},
and we provide here a characterization of this toric metric in terms of the
toric geometry of $(\Delta, \Ll)$, as developed in \cite{Abreu1,Guillemin,do}.

The starting point of our approach is based on a recent observation in
\cite{AM}, which in turn extends the formal GIT framework of
\cite{donaldson,fujiki} (realizing the scalar curvature of a K\"ahler metric
as a momentum map under the action of the group of hamiltonian
transformations) to a larger family of related GIT problems. Let $(M,\omega)$
be a compact symplectic manifold (or orbifold) and ${\rm Ham}(M,\omega)$ the
group of hamiltonian transformations. Fix a torus $\T \leq {\rm
  Ham}(M,\omega)$ and a positive hamiltonian $\afn>0$ with $\grad_{\omega}
\afn \in \mathfrak{t}:={\rm Lie}(\T)$. Let ${\mathcal C}^{\T}(M,\omega)$ be
the space of $\T$-invariant, $\omega$-compatible complex structures on
$(M,\omega)$ and ${\rm Ham}^{\T}(M, \omega)$ the subgroup of $\T$-equivariant
hamiltonian transformation, acting naturally on ${\mathcal
  C}^{\T}(M,\omega)$. The Lie algebra of ${\rm Ham}^{\T}(M, \omega)$ is
identified with the space $C_0^\infty(M)^{\T}$ of smooth, $\T$-invariant
functions of integral zero, endowed with the ${\rm Ham}^{\T}(M, \omega)$
bi-invariant inner product
\begin{equation}\label{inner}
\langle h_1, h_2 \rangle_{\afn,p} := \int_M h_1 h_2 \afn^{-(p+1)} v_{\omega}, 
\end{equation}
where $p$ is a real constant (which we call the \emph{conformal dimension})
and $v_{\omega}= \omega^m/m!$ is the volume form of $\omega$. The space
$\mathcal{C}^{\T}(M,{\omega})$ carries a formal Fr\'echet K\"ahler structure,
$({\bf J}, {\bf \Omega}^{\afn,p})$, defined by
\begin{equation*}
{\bf J}_{J} (\dot J) = J \dot J, \qquad 
{\bf \Omega}^{\afn,p}_{J}(\dot{J}_1, \dot{J}_2)
= \frac{1}{2} \int_M {\rm tr} \Big(J \dot{J}_1 \dot{J}_2\Big) w^{-(p-1)} v_{\omega},
\end{equation*}
where the tangent space of $\mathcal{C}^{\T}(M, \omega)$ at $J$ is identified
to be the Fr\'echet space of smooth sections $\dot J$ of ${\rm End}(TM)$
satisfying
\begin{equation*}
\dot{J} J + J \dot{J} =0,  \qquad
\omega(\dot{J} \cdot , \cdot) + \omega(\cdot, \dot{J} \cdot) =0.
\end{equation*}
The formal complex structure ${\bf J}$ is the same as the one in
\cite{donaldson,fujiki} whereas the modified formal symplectic form ${\bf
  \Omega}^{\afn,p}$ stays closed (as can easily be checked).

In the following, we denote by $g_J$ the K\"ahler metric corresponding to $J
\in \mathcal{C}^{\T}(M, \omega)$ and by $s_J$ and $\Delta_J$ the corresponding
scalar curvature and Laplace operator. We then have a straightforward (mutatis
mutandis) generalization of \cite[Thm.~1]{AM} (which corresponds to $p=2m$).

\begin{lemma}\label{l:GIT} The action of ${\rm Ham}^{\T}(M,\omega)$
on $(\mathcal{C}^{\T}(M, {\omega}), {\bf J}, {\bf \Omega}^{\afn,p})$ is
hamiltonian with a momentum map $\mu \colon \mathcal{C}^{\T}(M, {\omega}) \to
\bigl(C_0^\infty(M)^{\T}\bigr)^*$, whose value at $J$ is identified with the
$\langle \cdot, \cdot \rangle_{\afn,p}$-dual of
\begin{equation}\label{p-scal}
s_{J,\afn,p}:= \afn^2s_J - 2(p-1)\afn\Delta_J\afn - p(p-1)g_J^{-1}(\d\afn,\d\afn).
\end{equation}
\end{lemma}
Note that $s_{J,\afn,p}$ is the trace of the conformal modification
\begin{equation}\label{p-ricci}
\rho_{J,\afn,p}:= \afn^2 \rho_J + (p-1)\afn \d\d^c\afn - \tfrac12
p(p-1)\d\afn\wedge \d^c\afn
\end{equation}
of the Ricci form.
\begin{rem}\label{AK-futaki} As in~\cite{AM}, one can extend the definition
of $({\bf J}, {\bf \Omega}^{\afn,p})$ on $\mathcal{C}^{\T}(M, \omega)$ to the
larger Frech\'et space $\mathcal{AK}^{\T}(M,\omega)$ of $\T$-invariant
$\omega$-compatible almost-K\"ahler structures $J$. Then, using the formulae
in \cite[Ch.~8]{gauduchon-book} (see also \cite[Lemma~2.1 \&
  Prop.~3.1]{lejmi}), the momentum map for the action of ${\rm
  Ham}^{\T}(M,\omega)$ on $\mathcal{AK}^{\T}(M, {\omega})$ is still given by
Lemma~\ref{l:GIT}, except that in~\eqref{p-scal} we must take $s_{J}$ to be
the {\it Hermitian} scalar curvature of $g_J$ (the trace of the Ricci form of
the canonical Hermitian connection, see e.g.~\cite{lejmi}).

Using the contractibility of $\mathcal{AK}_{\omega}^G$, we obtain, as
in~\cite{AM}, a generalized \emph{Futaki invariant}
$\mathfrak{F}^{\T}_{\omega, \afn,p}\colon \mathfrak{t} \to \R$ of $(M,
\omega,\T,\afn,p)$: for any vector field $H \in \mathfrak{t}$ with a
hamiltonian $h$,
\begin{equation}\label{F}
\mathfrak{F}^{\T}_{\omega, \afn,p} (H) :=
\int_{M} \mathring s_{J,\afn,p} h \afn^{-(p+1)} v_{\omega}
\end{equation}
is independent of the choice of $J \in \mathcal{AK}^G_{\omega}$, where
$\mathring s_{J,\afn,p}$ is the $L^2$-projection of $s_{J,\afn,p}$ onto
functions with integral zero with respect to the volume form
$\afn^{-(p+1)} v_{\omega}$.
\end{rem}

Specializing to the case of a toric manifold (or orbifold) $(M,\omega,\T)$,
the above formalism allows one to extend the theory of extremal toric metrics
from~\cite{do} to the $(\afn,p)$-extremal toric case (the case $p=2m$ is
developed in detail in~\cite{AM}). In particular, we have the following
result:
\begin{prop}\label{t:(\afn,p)-extremal} Let $(M,\omega, \T)$ be a compact
toric orbifold with labelled Delzant polytope $(\Delta,\Ll)$ in $\R^m$ and
$\afn$ a positive affine-linear function on $\Delta$. Then,
\begin{enumerate}
\item[\rm (a)] There exists at most one (up to equivariant isometry)
  compatible toric metric $g_J$ on $(M,\omega,\T)$, for which $s_{J,\afn,p}$
  is an affine-linear function.
\item[\rm (b)] The affine-linear function in {\rm (a)} is uniquely determined
  by $(\Delta,\Ll,\afn,p)$.
\end{enumerate}
\end{prop}
\begin{defn}\label{d:(\afn,p)-extremal} The (unique) compatible toric metric
satisfying the condition (a) of Proposition~\ref{t:(\afn,p)-extremal} is called
the $(\afn,p)$-{\it extremal} metric of $(M,\omega, \T)$.
\end{defn}
\begin{thm}\label{t:LK-EM} Suppose $(\Delta, \Ll)$ is a labelled
projective cube in $\R^m$, corresponding to a compact toric orbifold $(M,
\omega, \T)$ and let $h_\Ll$ be the Levi--K\"ahler quotient metric defined
by~\eqref{metric}. Then $h_\Ll$ is the $(\afn,m+2)$-extremal metric of
$(\Delta, \Ll)$, where $\afn$ is the unique up to scale positive affine-linear
function on $\R^m$, vanishing on the hyperplane containing the intersections
of opposite facets of $\Delta$.
\end{thm}
\begin{proof} We take $\afn=\mu_0$ and apply~\eqref{p-ricci} with $p=m+2$,
$\rho_J=\rho_\Ll$ to obtain
\begin{equation}\label{rho}
\begin{split}
\rho_{\Ll,\mu_0} &= \mu_0^2 \rho_J + (m+1)\bigl(\mu_0 \d\d^c\mu_0 -
\tfrac12(m+2)\d\mu_0\wedge \d^c\mu_0\bigr) \\
&= -\tfrac{1}{2} \mu_0^2 \d\d^c\log\Bigl(\mu_0^{m+2}\prod_i A_i(\xi_i)\Bigr)
+ (m+1)\bigl(\mu_0^2 \d\d^c\log\mu_0-\tfrac{m}{2}\d\mu_0\wedge \d^c\mu_0\bigr)\\
&=\tfrac12\Bigl(-\mu_0^2\d\d^c\log\prod_i A_i(\xi_i)
+ m \mu_0^2\d\d^c\log\mu_0- m(m+1)\d\mu_0\wedge \d^c\mu_0 \Bigr).
\end{split}
\end{equation}
We now compute
\begin{gather*}
\d\d^c\log\prod_j A_i(\xi_i) =\sum_i\d\bigl( A_i'(\xi_i) \theta_i \bigr)
= \sum_i  A_i''(\xi_i) \d\xi_i\wedge \theta_i -\sum_{i,j} b_i\mu_0 A_i'(\xi_i)
\d\xi_j\wedge\theta_j \displaybreak[0]\\
\begin{split}
\d\d^c\log\mu_0 &= -\sum_i\d\bigl(b_i\mu_0 A_i(\xi_i)\theta_i \bigr)\\
&=-\sum_i b_i\mu_0 A_i'(\xi_i)\d\xi_i\wedge \theta_i + \sum_{i,j}
A_i(\xi_i)\bigl(b_ib_j \mu_0^2 \d\xi_j\wedge\theta_i + b_i^2 \mu_0^2
\d \xi_j\wedge\theta_j\bigr)
\end{split}\displaybreak[0]\\
\d\mu_0\wedge \d^c\mu_0 = \sum_{i,j} b_i^2\mu_0^4A_i(\xi_i)\d\xi_i\wedge \theta_j.
\end{gather*}                                     
It follows that
\begin{equation*}
s_{J,\mu_0} = \frac{2m \rho_{\Ll,\mu_0} \wedge \omega^{m-1}}{\omega^m}
=-\frac{\sum_{i=1}^m A''_i(\xi_i)}{b_0+b_1\xi_1 + \cdots + b_m\xi_m},
\end{equation*}
which (as $\deg A_j(\xi)\leq 3$) is an affine-linear function in momenta.
\end{proof}

\begin{rem}\label{r:EM-extremal}
We notice that for $m=2$, $m+2=2m$ and $s_{J,w,4}$ computes the scalar
curvature of the conformal oppositely oriented metric $\tilde h_\Ll =
(1/\afn^2) h_\Ll$ of Proposition~\ref{p:quotient}, i.e., $h_\Ll$ is
$\afn$-extremal in the sense of \cite{AM}.
\end{rem}

\subsection{Extremal Levi--K\"ahler quotients} Formula \eqref{s} shows that
for $\ell=m>2$, the Levi--K\"ahler metric $h_{\Ll}$ associated to a projective
cube cannot be extremal unless $b_i=0, i=1, \ldots m$, i.e., $M$ is the
product of weighted projective lines. However, we show below that when
$\ell=2$, Levi--K\"ahler quotients of a product of two spheres $\Sph^{2m_1+1}
\times \Sph^{2m_2+1}$ can provide new examples of extremal K\"ahler orbifolds.

\subsubsection{Extremal Levi--K\"ahler quotients of $\Sph^3 \times \Sph^3$}

As any quadrilateral is a projective cube, $h_\Ll$ is $(\afn,4)$-extremal by
Theorem~\ref{t:LK-EM}. Furthermore, $s_{J, \afn, 4}$ is the scalar curvature
of the conformal metric ${\tilde h}_\Ll = (1/\afn^2) h_\Ll$, see \cite{AM}.
By Proposition~\ref{p:quotient}, $(h_\Ll, J)$ is either a product, of Calabi
type, or a regular ambitoric K\"ahler metric of Segre type. We can then
use~\cite{ACG1,Eveline} to characterize the extremality of $h_{\Ll}$ as
follows.

\begin{prop}\label{p:extremal} Let $(M,\omega)$ be a compact toric $4$-orbifold
whose rational Delzant polytope is a labelled quadrilateral $(\Pol,\Ll)$ and
$h_\Ll$ the corresponding Levi--K\"ahler metric. If $(\Pol,\Ll)$ is a
parallelogram, then $(M,\omega_{\Ll}, h_\Ll)$ is an extremal toric orbifold
which is the K\"ahler product of two extremal weighted projective
lines\textup; otherwise $h_\Ll$ is extremal if and only if the oppositely
oriented ambitoric metric ${\tilde h}_\Ll= (1/\afn^2)h_\Ll$ has constant
scalar curvature, or equivalently, $\tilde h_\Ll$ is a conformally K\"ahler,
Einstein--Maxwell metric in the sense \cite{AM}, where $\afn$ is a positive
affine linear function on $\Delta$, determined up to positive scale by the
property that it vanishes where the opposite sides of $\Delta$ intersect.
\end{prop}
\begin{proof} The product case follows from Proposition~\ref{p:quotient} (see
Section~\ref{s:product}). When $h_\Ll$ is of Calabi-type, i.e. given by
\eqref{calabi} for polynomials $A(x)$ and $B(y)$ of degree $2$ or $3$, the
negative ambitoric metric $\tilde h_{\Ll}= x^2 h_{\Ll}$ is also of Calabi-type
with respect to the variables $(x, y)$ and functions $A(x)$ and $B(y)$. It
follows from \cite{ACG,Eveline} that $h_\Ll$ is extremal if and only if
$\tilde h_{\Ll}$ is extremal if and only if $B(y)$ has degree $2$ and
$(xA(x))''(0) = - B''(0)$. As $\deg A \leq 3$, this is precisely the condition
that $\tilde h_{\Ll}$ is of constant scalar curvature (see
\cite[Prop.~14]{ACG}). The case when $h_\Ll$ is regular ambitoric (i.e.,
negative orthotoric) is treated similarly, using the local form \eqref{ansatz}
and \cite[Prop.~11]{ACG}. Finally, as $\tilde h_{\Ll}= (1/\afn^2) h_{\Ll}$
with $\afn$ being a Killing potential with respect to $h_{\Ll}$, the scalar
curvature $\tilde h_{\Ll}$ is constant if and only if $\tilde h_{\Ll}$ defines
a conformally K\"ahler, Einstein--Maxwell metric, see \cite{AM}.
\end{proof}

\begin{rem} The metric $h_\Ll$ is extremal (and hence $\tilde h_\Ll$ is
Einstein--Maxwell) if and only if the affine function defined by $(\Delta,
\Ll)$ in Proposition~\ref{t:(\afn,p)-extremal}(b) is constant.  By an
observation originating in~\cite{Eveline}, for a given quadrilateral $\Delta$
this places two linear constraints on the labels $\Ll$, see also
\cite{AM}. Thus, the above characterization for the extremality of $h_\Ll$
lead to the following useful observation: given a compact convex quadrilateral
$\Pol$ which is not a parallelogram, there is a two-parameter family of inward
normals to the faces, such that the corresponding Levi--K\"ahler metric is
extremal.
\end{rem}

\subsubsection{CSC Levi--K\"ahler quotients of $\Sph^5\times \Sph^3$}
We discuss here examples of Levi--K\"ahler quotients of constant scalar
curvature (CSC), obtained from the generalized Calabi construction in
\S\ref{s:generalized-calabi}, where the base is $S= \C P^1$ equipped with a
Fubini--Study metric $\omega_S$ and the fibre $V$ is a toric orbifold with
Delzant image a simplex in $\R^2$. By Proposition~\ref{p:fibration}, the
resulting $6$-dimensional orbifold $(M,g, \omega)$ is obtained as a
Levi--K\"ahler quotient of $\Sph^5\times \Sph^3$ as soon as the fibrewise
metric is Bochner--flat (which is the condition to be Levi--K\"ahler in the
case of one factor). As the extremality condition is difficult to characterize
in general, we shall use the hamiltonian $2$-form ansatz with $\ell=2$ and
$N=1$ from \cite{hfkg1,ACGT}, which in turn is a special case of the
generalized Calabi ansatz \cite[\S4]{ACGT}. We briefly recall the
construction below and invite the Reader to consult \cite{hfkg1,ACGT} for
further details.

Let $(S, g_{S}, \omega_{S})$ be a compact Riemann orbi-surface and $\eta$ a
real constant. We build a K\"ahler metric $(g, \omega)$ with a hamiltonian
2-form of order $2$ and constant root $\eta$, defined on an orbifold fibration
over $M \to S$, with fibres isomorphic to an orbifold quotient of a weighted
projective plane. The K\"ahler metric $(g,\omega)$ is written on a dense
subset $\mathring{M} \sub M$ as follows (see \cite{ACGT}):
\begin{equation}\label{hfkg-metric}
\begin{split}
g &=(\eta - \xi_1)(\eta- \xi_2) g_{S}
+\frac{(\xi_1-\xi_2)p_c(\xi_1)}{F(\xi_1)} d\xi_1^2 + \frac{(\xi_2-\xi_1)p_c(\xi_2)}{F(\xi_2)} d\xi_2^2 \\
 & \ \ \ \ \  + \frac{F(\xi_1)}{(\xi_1- \xi_2)p_c(\xi_1)}\big(\theta_1 + \xi_1 \theta_2\big)^2 + \frac{F(\xi_2)}{(\xi_2- \xi_1)p_c(\xi_2)}\big(\theta_1 + \xi_2 \theta_2\big)^2  \\
\omega & =(\eta - \xi_1)(\eta- \xi_2)\omega_{S}
+d\sigma_1\wedge \theta_1 + d\sigma_2 \wedge \theta_2, \\
d\theta_1 & = - \eta \omega_{S},  \ \ d\theta_2 = \omega_{S} \ \ p_c(t) =  (t-\eta), \ \sigma_1= \xi_1 + \xi_2, \ \sigma_2= \xi_1\xi_2.
\end{split}
\end{equation}
Here $\xi_1 \in [-1, \beta]$, $\xi_2 \in [\beta, 1]$ (for some $|\beta|<1$)
are the {\it orthotoric} coordinates on the fibre, $|\eta|>1$, and
$F(x)$ is a smooth function which satisfies the positivity and boundary
conditions
\begin{enumerate}
\item[$\bullet$] $F(x)/p_c(x)>0$ on $(\beta, 1)$;  $F(x)/p_c(x)<0$ on $(-1, \beta)$;
\item[$\bullet$] $F(\pm 1)= F(\beta)=0$.
\end{enumerate}
It is easy to see that \eqref{hfkg-metric} is a special case of
\eqref{generalized-calabi}, with $N=1$, $(x_1,x_2)=(\sigma_1, \sigma_2)$,
$p_{11}=-\eta, p_{12}= 1, c_1= \eta^2$, and a toric orbifold fibre whose
Delzant polytope $\Delta$ is the the image of $[-1,\beta]\times [\beta, 1]$
under the map $(\sigma_1, \sigma_2)=(\xi_1,\xi_2, \xi_1\xi_2)$ and labelling
\[
L_{-1} = -c_{-1}(\sigma_1 + \sigma_2 +1),\quad
L_{+1}=-c_{1}(-\sigma_1 + \sigma_2 + 1),\quad
L_{\beta}= -c_{\beta}(-\beta\sigma_1 + \sigma_2 + \beta^2),
\]
where $c_{\pm 1}\bigl(F/p_c\bigr)'(\pm 1)=2=
c_{\beta}\bigl(F/p_c\bigr)'(\beta)$, see \cite[Prop.~9]{ACGT}. Here we assume
the usual rationality condition for the simplex $(\Delta, \Ll)$, which
certainly holds for $\beta, \eta, c_{\pm 1}, c_{\beta}\in\Q$.

We recall from \cite{hfkg1} that the metric \eqref{hfkg-metric} is extremal if
and only if the scalar curvature of $g_S$ is constant $s$ and $F(x)$ is a
polynomial of degree at most $5$ satisfying
\begin{equation}\label{condition}
F''(\eta)=-s.
\end{equation}
The metric $g$ is of constant scalar curvature if, furthermore, the degree of
$F(x)$ is at most $4$. Likewise, by the same result, the fibrewise orthotoric
metric
\begin{equation}\label{orthotoric}
\begin{split}
g^V &=\frac{(\xi_1-\xi_2)p_c(\xi_1)}{F(\xi_1)} d\xi_1^2
   + \frac{(\xi_2-\xi_1)p_c(\xi_2)}{F(\xi_2)} d\xi_2^2 \\
&\qquad + \frac{F(\xi_1)}{(\xi_1- \xi_2)p_c(\xi_1)}\big(dt_1 + \xi_1 dt_2\big)^2
      + \frac{F(\xi_2)}{(\xi_2- \xi_1)p_c(\xi_2)}\big(dt_1 + \xi_2 dt_2\big)^2  
\end{split}
\end{equation}
is Bochner--flat if and only if $F(x)/p_c(x)$ is a polynomial of degree at
most $4$. By Proposition~\ref{p:fibration}, taking in \eqref{metric} $(S,
\omega_{S})$ to be (an isometric orbifold quotient of) $\C P^1$ endowed with a
Fubini--Study metric and $F(x)=-c(x^2-1)(x-\beta)(x -\eta) (x-\gamma)$
(resp. $F(x)=-c(x^2-1)(x-\beta)(x -\eta)$), one gets an ansatz for extremal
(resp. constant scalar curvature) toric orbifolds, which are also
Levi--K\"ahler quotients of $\Sph^5 \times \Sph^3$.

We further specialize to the constant scalar curvature case,
i.e. $F(x)=-c(x^2-1)(x-\beta)(x -\eta)$ for a non-zero positive constant $c$.
Then, \eqref{condition} reduces to $2c(3\eta^2 - 2\beta \eta -1)= s$, whereas
the positivity conditions for $F(x)$ imply $\eta <-1$. Together with
$Scal_{S}>0$, these are the only constraints, subject to a rationality
condition which is trivially solved by taking $\beta, \eta, c$ rational.  For
instance, letting $\beta = 1/n, \eta = -n, c= 2/(3n^2+1), s=4$ gives rise to a
CSC Levi--K\"ahler quotient orbifold, which is not a product.

\end{document}